\documentclass[a4paper,11pt]{amsart}
\usepackage{mathrsfs}
 \usepackage[all]{xy}
\usepackage{amsmath,amssymb,amscd,bbm,amsthm,mathrsfs,dsfont}
\newtheorem{thm}{Theorem}[section]
\newtheorem{lem}{Lemma}[section]

\newtheorem{prop}{Proposition}[section]
\newtheorem{rem}{Remark}[section]

\begin{document}
\numberwithin{equation}{section}

 \title[On the $3$-representations of groups and the $2$-categorical Traces]{
On the $3$-representations of groups and the $2$-categorical Traces}
\author {Wei Wang}
\begin{abstract}  To $2$-categorify  the theory of group representations, we introduce the notions of
the $3$-representation  of a group  in  a strict $3$-category and the strict $2$-categorical action  of  a group  on  a strict   $2$-category. We also $2$-categorify the concept of the trace by introducing  the  $2$-categorical trace of a  $1$-endomorphism in a strict  $3$-category. For a
 $3$-representation $\rho$ of a group $G$ and an element $f$ of $G$, the $2$-categorical trace $\mathbb{T}r_2
\rho_f $ is a category. Moreover, the centralizer  of   $f$ in $G$ acts categorically  on  this $2$-categorical trace. We construct the induced strict $2$-categorical action of a finite group, and show that the $2$-categorical trace $ \mathbb{T}r_2$ takes an induced strict $2$-categorical action into an induced
categorical action of the initia groupoid.
As a corollary, we get the $3$-character formula of  the induced strict $2$-categorical action.
\end{abstract}

\thanks{Supported by National Nature Science
Foundation
  in China (No. 11171298)}

\thanks{Department of Mathematics,
Zhejiang University, Zhejiang 310027,
 P. R. China, Email:   wwang@zju.edu.cn}

 \maketitle

 \tableofcontents
\section{Introduction}

The notion of a
group acting on a category goes back to Grothendieck's Tohoku paper \cite{Gro}. Recently
Ganter, Kapranov \cite{GK} and Bartlett \cite{B-th}  categorified the concept of the trace  of a linear transformation by introducing the notion of the category trace.
This is a set associated to any endofunctor on a small category, and is a
 vector space in the linear case. Moreover, a functor commuting with the endofunctor   defines a linear transformation on this vector space, whose ordinary trace defines a joint trace. This allowed these authors to define $2$-characters. When a group acts on a   $k$-linear category, the joint trace of a commuting pair of group elements is the $2$-character
  of the
categorical action. This is an analogue of the character of the representation of a group on a vector space  and is  a $2$-class function. In general, an $n$-class function is a function
defined on $ n$-tuples of commuting elements of a group and invariant under simultaneous conjugation. Such functions  already appear  in equivariant Morava $E$-theory \cite{HKR}.
   The theory of $2$-representations was developed further in \cite{Bar}
  \cite{E1}   \cite{E2}  \cite{E3} \cite{GU} \cite{O} \cite{Wi} etc..

During the past two decades  an active direction of research has been the  categorification of    some   algebraic, geometric or analytic concepts. For example, $2$-vector spaces,  $2$-bundles (gerbes), $2$-connections  and $2$-curvatures.  All involve $2$-categorical constructions and have various  applications, such as a geometric definition of elliptic cohomology \cite{BBK},   $2$-gauge theory \cite{BH11}  \cite{BS}   and the $2$-dimensional Langlands correspondence \cite{K} \cite{Osi}. It is   believed that   higher categorification is necessary for many geometric  and physical applications. $3$-categorical constructions  already appear  in the theory of
$2$-gerbes ($3$-bundles) \cite{Br94}   \cite{Br10}   and   in  $3$-gauge theory \cite{MP11} \cite{SW13} \cite{Wa}, which involves more general $\mathbf{Gray}$-categories.
The purpose of this paper is to $2$-categorify  the theory of group representations and characters by introducing the notions of the $3$-representation  of   a group  in a  $3$-category, the strict  $2$-categorical action  of a  group  on  a   $2$-category and the $2$-categorical trace. The problem of  investigating representations of   groups in  higher categories has already been mentioned  in \cite{GK}.

A geometric motivation for considering  higher representations of   groups
is as  follows.  Suppose that $G$ is a Lie group and that $H$ is a Lie subgroup. Let $V$ be a finite dimensional representation of $H$. We can construct a homogeneous vector bundle $G\times_H V$ over the  homogeneous space $G/H$ as $G\times  V$  modulo the equivalent
relation
$$
   (g,v)\sim (gh,h^{-1}.v) \quad {\rm for}\quad g\in G, h\in H, v\in V.
$$
The space of sections of this bundle is exactly the space ${\rm Ind}_H^GV$ of the induced representation. When $V$ is a $2$- or $3$-representation of $H$,  a similar construction will give us a homogeneous  $2$- or $3$-bundle over the  homogeneous space $G/H$. This will provide us good examples of higher bundles in higher differential geometry and higher gauge theory. But for a higher representation $\pi$ of the Lie group $H$, the functors $\pi(h)$ usually depend on $h\in H$ ``discontinuously". Thus
it is not easy to describe the space of ``sections" of the resulting higher bundles. However, when $G$ and $H$ are finite, $G/H$ is discrete, and so we have a clear picture. This is why we only consider $3$-representations of a finite group in this paper.

For simplicity, we only consider strict $2$- and $3$-categories.
  A $3$-representation of a group  $G$ in a $3$-category is given by a $1$-isomorphism  for each element of $G$, a $2$-isomorphism for each pair of elements of $G$, and a $3$-isomorphism for each triple of elements of $G$. These $3$-isomorphisms  must satisfy  the $3$-cocycle condition. This condition has a   simple  geometric interpretation: the composition of $3$-isomorphisms corresponding to $5$ tetrahedrons   in the boundary of a $4$-simplex is equal to the identity $3$-arrow.
  Given a $2$-category $\mathcal{V}$,  a  {\it strict $2$-categorical action of $G$ on $\mathcal{V}$} is given by an endofunctor   of $\mathcal{V}$ for   each element of  $G$, a pseudonatural transformation   between functors   for each pair of elements of $G$, and a  modification for each triple  of elements of $G$.
 Details are given in Section 2.3-2.4.

Recall that given a $2$-representation $\varrho$ of a finite group $G$ in a $2$-category $\mathcal{V}$ and an element  $f$ of $G$, we have  a $1$-isomorphism £¤$$
   \varrho_f:x\rightarrow x ,
$$
where $x$ is an object of $\mathcal{V}$ that $G$ acts on.
In \cite{B-th} \cite{GK}, the authors introduced the notion of  the categorical trace $\mathbb{T}r \varrho_f $. This    is the set of $2$-arrows in $\mathcal{V}$,
whose $1$-source is the unit arrow $1_x$ and whose $1$-target is $\varrho_f$. The centralizer of $f$ in $G$ acts on this set naturally. In our case, given a $3$-representation $\rho$  of $G$ in a $3$-category $\mathcal{C}$ and an element $f$ of $ G$, we have a  $1$-isomorphism $\rho_f:x\rightarrow x$  in $\mathcal{C}$. The
$2$-categorical trace
 $ {\rm \mathbb{T}r}_2\rho_f
$
is a category. Its objects are $2$-arrows with $1$-source   the unit arrow $1_x$ and $1$-target $\rho_f$,  and its morphisms are $3$-isomorphisms between such $2$-isomorphisms:
  $$
      \xy 0;/r.22pc/:
 (-15,0)*+{ {x}}="1";
(15,0)*+{ {x} }="2";
 {\ar@/^1.33pc/ "1";"2"^{1_x}};
{\ar@/_1.33pc/ "1";"2"_{\rho_f}};
 {\ar@{=>}^{\chi} (0,4)*{}; (0,-4)*{}};
         \endxy, \qquad\qquad
          \xy 0;/r.22pc/:
 (-15,0)*+{ {x}}="1";
(15,0)*+{ {x} }="2";
 {\ar@/^1.33pc/ "1";"2"^{1_x}};
{\ar@/_1.33pc/ "1";"2"_{\rho_f}};
                  (0,4)*+{}="A";
(0,-4)*+{}="B";
{\ar@{=>}@/_.5pc/_\chi "A"+(-2.33,0) ; "B"+(-1.66,-.55)};
{\ar@{=}@/_.5pc/ "A"+(-2.33,0) ; "B"+(-2.33,0)};
{\ar@{=>}@/^.5pc/^{\chi'} "A"+(2.33,0) ; "B"+(1.66,-.55)};
{\ar@{=}@/^.5pc/ "A"+(2.33,0) ; "B"+(2.33,0)};
{\ar@3{->} (-2,0)*{}; (3,0)*{}};\endxy.
          $$
   Moreover, the centralizer of $f$ in $G$, denoted by $C_G(f)$, acts categorically on  the $2$-categorical trace ${\rm \mathbb{T}r}_2\rho_f$ in the following sense.
We can define an  invertible functor  $\psi_g$ acting on  $\mathbb{T}r_2\rho_f $ for each $  g \in C_G(f)$,  and for any $h, g \in C_G(f)$, define a
natural isomorphism
 $$
    \Gamma_{h,g}:  \psi_h  \circ \psi_g\longrightarrow\psi_{hg}$$
  between such functors on the category   $\mathbb{T}r_2\rho_f$. This construction is given in Section 3.  To prove  the  action to be categorical, we have to show the associativity in the
definition of categorical action, i.e.,
 \begin{equation}\label{eq:cat-tr-0}\Gamma_{k,h g} {\#} (\psi_k \circ\Gamma_{h,g})= \Gamma_{kh , g} {\#} (  \Gamma_{k, h  } \circ\psi_g ):\psi_k \circ\psi_h \circ\psi_g\longrightarrow\psi_{khg},
\end{equation}
for any $k,h,g\in C_G(f)$, where $ {\#} $ is the composition of natural transformations  between  functors on the category ${\rm \mathbb{T}r}_2\rho_f$. This is the most difficult and technical part of this paper. By applying the $3$-cocycle identity (\ref{eq:4-cocycle})  repeatedly, we prove  in Section  6 that
$$
   \left\{\psi_g,\Gamma_{h,g}\right\}_{g,h\in C_G(f)}
$$
is a categorical action  of    the centralizer $C_G(f)$ on the category ${\rm \mathbb{T}r}_2 \rho_f$.

An easy and interesting example of $3$-representations is the $1$-dimensional one, which is  given by a $3$-cocycle  on a finite group $G$. A {\it $3$-cocycle}
is a function $c: G\times G\times G\longrightarrow k^*$ such that
\begin{equation}\label{eq:3-cocycle-0}
  c(g_3,g_2,g_1) c(g_4, g_3 g_2,g_1)c(g_4, g_3 ,g_2 )=c(g_4, g_3, g_2g_1)c(g_4  g_3, g_2,g_1)
\end{equation} for any $g_4,\ldots, g_1\in G$. Here $k$ is  a field of characteristic $0$. Such a $3$-cocycle gives us a strict action of $G$ on a $2$-category with only one object,  one $1$-arrow and the set of $2$-arrows isomorphic to $ k^*$. For an element  $f$ of $G$, its $2$-categorical trace ${\rm \mathbb{T}r}_2\rho_f$ is a  category with only one object  and the set of $1$-arrows isomorphic to $ k^*$. For any $h$ and $ g$ in  the centralizer  $C_G(f)$, we can construct  an element $\Gamma_{h,g}$ (\ref{eq:2-cocycle}) from
the $3$-cocycle $c$ in (\ref{eq:3-cocycle-0}) such that $\Gamma_{*,*}$ is a $2$-cocycle on the centralizer. This  can be proved quite easily and elementarily  by using  the condition (\ref{eq:3-cocycle-0}) for $3$-cocycles repeatedly in Section 6.1. This corresponds step by step to the proof of the general case carried out  in Section 6.4. It can be viewed as a simple model of the proof of (\ref{eq:cat-tr-0}). The difficulty in the general case is that we have to handle diagrams, while in the $1$-dimensional case we only need to handle element of the field $k$.

Suppose that $\mathcal{C}$ is a $ {k}$-linear $3$-category. Then $\mathbb{ {T}}r_2  \rho_f $ is also a $ {k}$-linear category.
 If $k,g$ and $f $ are pairwise commutative, then
  $  \psi_k  $  and  $  \psi_g  $ are $k$-linear endofunctors acting on $\mathbb{ {T}}r_2  \rho_f $.
We define the $3$-character of a $3$-representation $\rho$  to be
 $$
    \chi_\rho( f,g , k ):=  \hskip 2mm{\rm the}\hskip 2mm{\rm
joint}\hskip 2mm{\rm trace}\hskip 2mm{\rm of}\hskip 2mm{\rm functors}\hskip 2mm   \psi_k  \hskip 2mm{\rm and}\hskip 2mm   \psi_g \hskip 2mm{\rm on}\hskip 2mm   {\rm \mathbb{T}r}_2\rho_f .
 $$
It is
the  trace of the linear transformation induced by the functor $  \psi_k  $   on  the  $ {k}$-vector space  $ \mathbb{{T}}r  \psi_g $.

Suppose that   a subgroup $H $ of a finite group $G$ acts strictly $2$-categorically on a $2$-category $\mathcal{V}$. In Section 4, we define the induced   $2$-category
${\rm Ind}_H^G(\mathcal{V})$ and strict $2$-categorical  action of $G$ on it. In Section 5, we calculate the $2$-categorical trace of the induced strict $2$-categorical  action as
\begin{equation}\label{eq:3-trace}
   \mathbb{T}r_2({\rm Ind}_H^G\rho )={\rm Ind}_{\Lambda(H)}^{\Lambda(G)} \mathbb{T}r_2(\rho ),
\end{equation}
where $ \Lambda(H)  $ and $\Lambda(G) $ are initia groupoids associated to groups $H$ and $G$, respectively. As a corollary, we derive the $3$-character  of the induced strict $2$-categorical action, which coincides with the formula in \cite{HKR} for $n$-characters when $n=3$. These results are the generalization of induced   categorical  action and the $2$-character formula in  \cite{GK}.

It would be interesting to investigate the $m$-representation of a group in an $m$-category, the $m$-cocycle condition  and  $(m-1)$-categorical trace for a positive integer $m >3$.

I would like to thank the  anonymous referee for his/her many inspiring and  valuable suggestions.
\section{The $3$-representations of groups}
\subsection{Strict  $2$-categories}
 A {\it $2$-category} is a category enriched over the category of all small categories. In particular,
a strict $2$-category  $  \mathcal{ C }$    consists of collections $ \mathcal C_0$ of objects, $ \mathcal C_1$ of arrows  and $ \mathcal C_2$ of $2$-arrows,
together with

$\bullet$ functions $s_n, t_n :  \mathcal C_i\rightarrow  \mathcal C_n$ for all $0 \leq n < i \leq 2$,  called
{\it $n$-source} and {\it $n$-target},

$\bullet$ functions $\#_n :  \mathcal C_{n+1} \times \mathcal C_{n+1}\rightarrow  \mathcal C_{n+1}$ for all $n=0,1$, called {\it vertical composition},

$\bullet$ a function  $\#_0 :  \mathcal C_2  \times \mathcal C_{2}\rightarrow  \mathcal C_2$, called the {\it  horizontal composition},

$\bullet$ a function  $ 1_{*}   :  \mathcal C_i \rightarrow  \mathcal C_{i+1 }$ for $i=0,1$, called the {\it identity}.

For a $1$-arrow $     \xy0;/r.22pc/:
(-8,0)*+{x}="4";
(8,0)*+{y}="6";
{\ar@{->} "4";"6"^{A}};
\endxy  $,  its $0$-source  and   $0$-target are $x$ and $y$, respectively.
For a  $2$-arrow $
      
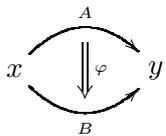
\xy0;/r.22pc/:
(-10,0)*+{  x }="1";
(10,0)*+{    y}="2";
  {\ar@{=>}^{ \scriptscriptstyle \varphi  } (0, 4)*{};(0,-4)*{}} ;
 {\ar@/^1.35pc/ "1";"2"^{ \scriptscriptstyle  A } };
  {\ar@/_1.35pc/ "1";"2"_{ \scriptscriptstyle B } };
\endxy$ in $\mathcal C_2$, its $1$-source  and   $1$-target are $     \xy0;/r.22pc/:
(-8,0)*+{x}="4";
(8,0)*+{y}="6";
{\ar@{->} "4";"6"^{A}};
\endxy  $ and $     \xy0;/r.22pc/:
(-8,0)*+{x}="4";
(8,0)*+{y}="6";
{\ar@{->} "4";"6"^{B}};
\endxy  $, respectively, while its $0$-source  and   $0$-target are $x$ and $y$, respectively.

Two  $1$-arrows $A$ and $A'$ are called {\it $0$-composable} if  the   $0$-target of $A$ coincides with the $0$-source  of $A'$. In this case, their vertical composition
is $ A\#_0 A':    \xy0;/r.22pc/:
(-8,0)*+{x}="4";
(8,0)*+{y}="6";(24,0)*+{z}="8";
{\ar@{->} "4";"6"^{A}};
 {\ar@{->} "6";"8"^{A'}};
\endxy  $.
Two  $2$-arrows $\phi$ and $\psi$ are called {\it  $1$-composable} if the   $1$-target of $\phi$ coincide with the $1$-source  of $\psi$. In this case, their vertical composition
$
  \phi\#_1 \psi
$ is
$$
   
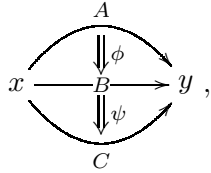
\xy 0;/r.22pc/:
   (-12,0)*+{x}="4";
(12,0)*+{y}="6";
{\ar@{->}|-{B} "4";"6"};
{\ar@/^1.85pc/^{A} "4";"6"};
{\ar@/_1.85pc/_{C} "4";"6"};
{\ar@{=>}^{\phi} (0,7)*{};(0,1.5)*{}} ;
{\ar@{=>}^{\psi} (0,-1.5)*{};(0,-7)*{}} ;
\endxy,
$$
where $A=s_1(\phi)$, $B=t_1(\phi)=s_1(\psi)$, $C=t_1(\psi)$, $x=s_0(\phi)=s_0(\psi)$, $y=t_0(\phi)=t_0(\psi)$. In general, two arrows are composable if the target matching condition is satisfied.

Two  $2$-arrows $\phi$ and $\psi$ are called {\it horizontally composable}¡¡ ({\it $0$-composable})¡¡ if the   $0$-target of $\phi$ coincides with the $0$-source  of $\psi$. In this case, their horizontal composition
$
  \phi\#_0 \psi
$ is
$$
     
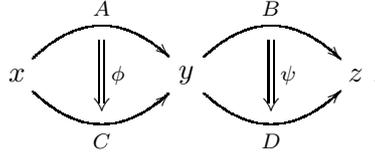
\xy0;/r.22pc/:
(-12,0)*+{x}="4";
(12,0)*+{y}="6"; (36,0)*+{z}="8";
{\ar@/^1.55pc/^{A} "4";"6"};
{\ar@/_1.55pc/_{C} "4";"6"};
{\ar@{=>}^{\phi} (0,5)*{};(0,-5 )*{}} ; {\ar@/^1.55pc/^{B} "6";"8"};
{\ar@/_1.55pc/_{D} "6";"8"};
{\ar@{=>}^{\psi} (24,5)*{};(24,-5 )*{}} ;
\endxy  .
  $$
   In particular, when $\phi=1_A$  we call $1_A\#_{0}\psi$    {\it     whiskering from left by $1$-arrow $A$}, and denote it by
$$
  A\#_{0}\psi:\qquad \xy0;/r.22pc/:
(-10,0)*+{ y }="1";(-30,0)*+{ x }="0";
(10,0)*+{z  }="2";
  {\ar@{=>}^{ \psi } (0, 4)*{};(0,-4)*{}} ;
 {\ar@/^1.35pc/^{ B  } "1";"2" };
  {\ar@/_1.35pc/_{ D  } "1";"2" };{\ar@{->}^{A   } "0";"1" };
\endxy,
$$
  Similarly,   we define     {\it    whiskering from  right  by a $1$-arrow}.

The identities satisfy
\begin{equation}\label{eq:identities}   \begin{array}{l}
   1_x\#_0 A=A= A\#_0 1_y  ,\quad\qquad {\rm for}\quad {\rm any}\quad 1-{\rm arrow}\quad A:x\longrightarrow y;\\
   1_A\#_1 \phi=\phi= \phi\#_1 1_B  ,\qquad\quad  {\rm for}\quad {\rm any}\quad   2-{\rm arrow}\quad \phi:A \Longrightarrow B.
  \end{array}\end{equation}

The composition $\#_p$ satisfies the {\it associativity}
\begin{equation}\label{eq:composition-associativity}
   (\phi\#_p \psi)\#_p\omega=  \phi\#_p (\psi \#_p\omega),
\end{equation}
 if the corresponding arrows are $p$-composable, for $p=0$ or $1$.

 The horizontal composition satisfies the {\it interchange  law}:
 \begin{equation}\label{eq:interchanging-law}
    (A\#_0\psi)\#_1(\phi\#_0D)=\phi \#_0\psi=(\phi\#_0 B)\#_1(C\#_0\psi ).
 \end{equation}
 Namely,
$$
     
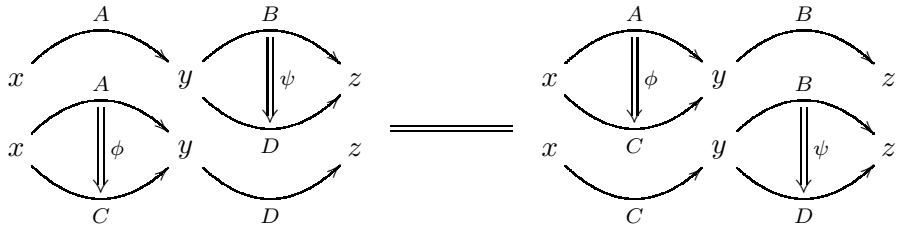
\xy0;/r.22pc/:
(-12,0)*+{x}="4";
(12,0)*+{y}="6"; (36,0)*+{z}="8";
{\ar@/^1.55pc/ "4";"6"^{A}};
 {\ar@/^1.55pc/^{B} "6";"8"};
{\ar@/_1.55pc/_{D} "6";"8"};
{\ar@{=>}^{\psi} (24,6)*{};(24,-6 )*{}} ;
(-12,-10)*+{x}="04";
(12,-10)*+{y}="06"; (36,-10)*+{z}="08";
{\ar@/^1.55pc/^{A} "04";"06"};
{\ar@/_1.55pc/_{C} "04";"06"};
{\ar@{=>}^{\phi} (0,-4)*{};(0,-16 )*{}} ;
{\ar@/_1.55pc/ "06";"08"_{D}};
\endxy \xy0;/r.22pc/:
  (0,-7)*+{ }="1";
(20,-7)*+{ }="2";
{\ar@{=}^{ }   "1" ;"2" };
   \endxy \xy0;/r.22pc/:
(-12,0)*+{x}="4";
(12,0)*+{y}="6"; (36,0)*+{z}="8";
{\ar@/^1.55pc/^{A} "4";"6"};
{\ar@/_1.55pc/_{C} "4";"6"};
{\ar@{=>}^{\phi} (0,6)*{};(0,-6 )*{}} ; {\ar@/^1.55pc/ "6";"8"^{B}};
(-12,-10)*+{x}="14";
(12,-10)*+{y}="16"; (36,-10)*+{z}="18";
{\ar@/_1.55pc/ "14";"16"_{C}};
{\ar@/^1.55pc/^{B} "16";"18"};
{\ar@/_1.55pc/_{D} "16";"18"};
{\ar@{=>}^{\psi} (24,-4)*{};(24,-16 )*{}} ;
\endxy
 $$
the vertical composition of left two $2$-arrows coincides with the vertical composition of right two $2$-arrows. They are both equal to the horizontal composition $\phi \#_0\psi$.  The interchange  law allows us  to change the order of compositions of $2$-arrows, up to whiskerings. This is essentially the paste theorem for $2$-categories  (cf. \S 2.13 in \cite{KV}).

The   interchange  law (\ref{eq:interchanging-law}) is a special case of the following more general {\it compatibility condition} for different compositions. If $(\beta,\beta'),(\gamma ,\gamma' ) \in \mathcal{C}_{k }\times \mathcal{C}_{k } $ are  $p$-composable and $(\beta ,  \gamma),(\beta' ,\gamma' )\in
\mathcal{C}_{k }\times  \mathcal{C}_{k } $ are $q$-composable, $p,q=0,1$,
then we have
\begin{equation}\label{eq:asso2}
   (\beta\#_{p }\beta') \#_{q} (\gamma \#_{p }\gamma' )=  (\beta \#_{q}  \gamma)\#_{p }(\beta'  \#_{q}\gamma' ).
\end{equation}
 The the left-hand side of the interchange  law (\ref{eq:interchanging-law}) is exactly the compatibility condition (\ref{eq:asso2}) with
$
    p=0 ,   q=1 ,   \beta=1_A ,  \beta' =\psi ,    \gamma=\phi ,    \gamma'=1_D ,
$
by using the property (\ref{eq:identities}) of identities. (\ref{eq:identities}) (\ref{eq:composition-associativity}) and (\ref{eq:asso2}) are the main axioms that a strict  $2$-category satisfies.

A $1$-arrow
 $A: x \rightarrow y $ is
called  {\it   invertible} or a {\it   $1$-isomorphism}, if there exists another $1$-arrow
 $B
 : y \rightarrow x$ such that $1_x =A\#_0B $
 and
$  B\#_0A
 = 1_y$. A strict  $2$-category in which every $1$-arrow is invertible is called a {\it strict $2$-groupoid}.
A $2$-arrow $\varphi:
A\Rightarrow B$ is called {\it invertible} or a {\it$2$-isomorphism} if there exists another
$2$-arrow $\psi :
B\Rightarrow A$
  such that $\psi\#_1\varphi = 1_B$
  and $\varphi\#_1 \psi = 1_A$.
$\psi$ is uniquely determined and called the {\it inverse of $\varphi$}.

Let $\mathcal{S}$ and $\mathcal{T}$ be two strict $2$-categories. A {\it (strict) $2$-functor} $F : \mathcal{S }\rightarrow \mathcal{T}$
is an assignment
of a $2$-arrow
$$
   \xy  0;/r.30pc/:
(-10,0)*+{ \scriptscriptstyle F (X)}="1";
(10,0)*+{ \scriptscriptstyle F (Y)}="2";
  {\ar@{=>}^{ \scriptscriptstyle F(\varphi) } (0, 5)*{};(0,-5)*{}} ;
 {\ar@/^1.95pc/^{ \scriptscriptstyle F(f)} "1";"2" };
  {\ar@/_1.95pc/_{ \scriptscriptstyle F(g)} "1";"2" };
\endxy
$$
 to each $2$-arrow $
      \xy0;/r.22pc/:
(-10,0)*+{ \scriptscriptstyle  X }="1";
(10,0)*+{ \scriptscriptstyle  Y }="2";
  {\ar@{=>}^{ \scriptscriptstyle \varphi  } (0, 5)*{};(0,-5)*{}} ;
 {\ar@/^1.35pc/^{ \scriptscriptstyle  f } "1";"2" };
  {\ar@/_1.35pc/_{ \scriptscriptstyle  g } "1";"2" };
\endxy$
such that $F$ preserves compositions $\#_p$ and identities. More explicitly, we have

$\bullet$
$F(\varphi\#_1\psi  ) = F(\varphi ) \#_1 F(\psi)$ and $F(1_f ) = 1_{F(f)}$
for all composable $2$-arrows $\varphi$ and  $\psi$  and any $0$- or $1$-arrow $ f$;

$\bullet$
$F(g) \#_0 F(f) = F(g \#_0 f)$
for all composable $1$-arrows $g $ and $f$, and
$F(\varphi ) \#_0 F(\psi) = F( \varphi \#_0\psi)$
for all horizontally composable $2$-arrows $\varphi$ and  $\psi$.

  Let $F_1$ and $F_2$ be two  $2$-functors   from $\mathcal{S}$ to $\mathcal{T}$. A {\it pseudonatural
transformation} $\rho: F_1 \rightarrow F_2$ is an assignment of a $1$-arrow $\rho(X)$ in $\mathcal{T}$ to each object $  X  $ in $\mathcal{S}$ and a $2$-isomorphism $\rho(f)$
\begin{equation}\label{eq:pseudo-trans0}
    \xy 0;/r.30pc/:
(-10,0)*+{ \scriptscriptstyle F_1(X)}="1";
(10,0)*+{ \scriptscriptstyle F_1(Y)}="2";
(-10,-20)*+{ \scriptscriptstyle F_2(X)}="3";
(10,-20)*+{ \scriptscriptstyle F_2(Y)}="4";
{\ar@{->}^{ \scriptscriptstyle F_1(f)} "1";"2" };
{\ar@{->}_{ \scriptscriptstyle F_2(f)} "3";"4" };
{\ar@{->}_{ \scriptscriptstyle \rho(X)} "1";"3" };
{\ar@{->}^{ \scriptscriptstyle \rho(Y)} "2";"4" };
{\ar@{=>}^{ \scriptscriptstyle \rho(f) } (7, -3)*{};(-7,-17)*{}} ;
\endxy
\end{equation} in $\mathcal{T}$ to each $1$-arrow $f : X \rightarrow Y$ in $\mathcal{S}$
such that they satisfy two
axioms

$\bullet$ The composition of $1$-arrows in $\mathcal{S}$:
\begin{equation}\label{eq:pseudo-trans1}
   \xy  0;/r.30pc/:
(-10,0)*+{ \scriptscriptstyle F_1(X)}="1";
(10,0)*+{ \scriptscriptstyle F_1(Y)}="2";
(-10,-20)*+{ \scriptscriptstyle F_2(X)}="3";
(10,-20)*+{ \scriptscriptstyle F_2(Y)}="4";
{\ar@{->}^{ \scriptscriptstyle F_1(f)} "1";"2" };
{\ar@{->}_{ \scriptscriptstyle F_2(f)} "3";"4" };
{\ar@{->}_{ \scriptscriptstyle \rho(X)} "1";"3" };
{\ar@{->}|-{ \scriptscriptstyle \rho(Y)} "2";"4" };
{\ar@{=>}^{ \scriptscriptstyle \rho(f) } (7, -3)*{};(-7,-17)*{}} ;
(30,0)*+{ \scriptscriptstyle F_1(Z)}="12";
(30,-20)*+{ \scriptscriptstyle F_2(Z)}="14";
{\ar@{->}^{ \scriptscriptstyle F_1(g)} "2";"12" };
{\ar@{->}_{ \scriptscriptstyle F_2(g)} "4";"14" };
{\ar@{->}^{ \scriptscriptstyle \rho(Z)} "12";"14" };
{\ar@{=>}^{ \scriptscriptstyle \rho(g) } (27, -3)*{};(13,-17)*{}} ;
\endxy
\xy 0;/r.30pc/:
  (0,0)*+{ }="1";
(10,0)*+{ }="2";
{\ar@3{=}^{ }   "1"+(0,-10);"2"+(0,-10)};
   \endxy
\xy  0;/r.30pc/:
(-10,0)*+{ \scriptscriptstyle F_1(X)}="1";
(10,0)*+{ \scriptscriptstyle F_1(Z)}="2";
(-10,-20)*+{ \scriptscriptstyle F_2(X)}="3";
(10,-20)*+{ \scriptscriptstyle F_2(Z)}="4";
{\ar@{->}^{ \scriptscriptstyle F_1(f \#_0 g)} "1";"2" };
{\ar@{->}_{ \scriptscriptstyle F_2(f\#_0 g)} "3";"4" };
{\ar@{->}_{ \scriptscriptstyle \rho(X)} "1";"3" };
{\ar@{->}^{ \scriptscriptstyle \rho(Z)} "2";"4" };
{\ar@{=>}^{ \scriptscriptstyle \rho(f\#_0 g) } (7, -3)*{};(-7,-17)*{}} ;
\endxy;
\end{equation}

$\bullet$ The compatibility with $2$-arrows:

\begin{equation}\label{eq:pseudo-trans2}
     \xy  0;/r.30pc/:
(-10,13)*+{ \scriptscriptstyle F_1(X)}="1";
(10,13)*+{ \scriptscriptstyle F_1(Y)}="2";
(-10,-7)*+{ \scriptscriptstyle F_2(X)}="3";
(10,-7)*+{ \scriptscriptstyle F_2(Y)}="4";
{\ar@{->}^{ \scriptscriptstyle F_1(f)} "1";"2" };
{\ar@{->}^{ \scriptscriptstyle F_2(f)} "3";"4" };
{\ar@{->}_{ \scriptscriptstyle \rho(X)} "1";"3" };
{\ar@{->}^{ \scriptscriptstyle \rho(Y)} "2";"4" };
{\ar@{=>}_{ \scriptscriptstyle \rho(f) } (7, 10)*{};(-7,-3)*{}} ;
  {\ar@/_3.3pc/_{ \scriptscriptstyle F_2(g)} "3";"4" };
  {\ar@{=>}^{ \scriptscriptstyle F_2(\varphi) } (0, -8)*{};(0,-17)*{}} ;
\endxy\xy 0;/r.30pc/:
  (0,0)*+{ }="1";
(10,0)*+{ }="2";
{\ar@3{=}^{ }   "1"+(0,0);"2"+(0,0)};
   \endxy  \xy  0;/r.30pc/:
(-10,3)*+{ \scriptscriptstyle F_1(X)}="1";
(10,3)*+{ \scriptscriptstyle F_1(Y)}="2";
(-10,-17)*+{ \scriptscriptstyle F_2(X)}="3";
(10,-17)*+{ \scriptscriptstyle F_2(Y)}="4";
{\ar@{->}_{ \scriptscriptstyle F_1(g)} "1";"2" };
{\ar@{->}_{ \scriptscriptstyle F_2(g)} "3";"4" };
{\ar@{->}_{ \scriptscriptstyle \rho(X)} "1";"3" };
{\ar@{->}^{ \scriptscriptstyle \rho(Y)} "2";"4" };
{\ar@{=>}^{ \scriptscriptstyle \rho(g) } (7, 0)*{};(-7,-14)*{}} ;
  {\ar@/^3.3pc/^{ \scriptscriptstyle F_1(f)} "1";"2" };
  {\ar@{=>}^{ \scriptscriptstyle F_1(\varphi) } (0, 13)*{};(0, 4)*{}} ;
\endxy
\end{equation}
for any $2$-arrow $\varphi:f\Rightarrow g$.

Let $F_1, F_2 : \mathcal{S} \rightarrow \mathcal{T}$ be two strict $2$-functors and let $ \rho_1,\rho_2 : F_1 \rightarrow F_2$
be pseudonatural transformations. A {\it  modification} $\Phi: \rho_1\xy 0;/r.30pc/:
  (0,0)*+{ }="1";
(6,0)*+{ }="2";
{\ar@3{->}    "1"+(0, 0);"2"+(0, 0)};
   \endxy \rho_2$ is an assignment
of a $2$-arrow
$$
   \xy  0;/r.30pc/:
(-10,0)*+{ \scriptscriptstyle F_1 (X)}="1";
(10,0)*+{ \scriptscriptstyle F_2 (X)}="2";
  {\ar@{=>}^{ \scriptscriptstyle  \Phi(X) } (0, 4)*{};(0,-4)*{}} ;
 {\ar@/^1.35pc/^{ \scriptscriptstyle \rho_1(X)} "1";"2" };
  {\ar@/_1.35pc/_{ \scriptscriptstyle \rho_2(X)} "1";"2" };
\endxy
$$ in $\mathcal{T} $ to any object $ X $ in $\mathcal{S}$,
which satisfies
\begin{equation}\label{eq:modification}
     \xy  0;/r.30pc/:
(-10,0)*+{ \scriptscriptstyle F_1(X)}="1";
(10,0)*+{ \scriptscriptstyle F_1(Y)}="2";
(-10,-20)*+{ \scriptscriptstyle F_2(X)}="3";
(10,-20)*+{ \scriptscriptstyle F_2(Y)}="4";
{\ar@{->}^{ \scriptscriptstyle F_1(f)} "1";"2" };
{\ar@{->}_{ \scriptscriptstyle F_2(f)} "3";"4" };
{\ar@{->}|-{ \scriptscriptstyle \rho_1(X)} "1";"3" };
{\ar@{->}^{ \scriptscriptstyle \rho_1(Y)} "2";"4" };
{\ar@{=>}^{ \scriptscriptstyle \rho_1(f) } (7, -3)*{};(-7,-17)*{}} ;
  {\ar@/_3.85pc/_{ \scriptscriptstyle \rho_2(X)} "1";"3" };
  {\ar@{<=}^{ \scriptscriptstyle \Phi(X) } (-22, -10)*{};(-14,-10)*{}} ;
\endxy\xy 0;/r.30pc/:
  (0,0)*+{ }="1";
(10,0)*+{ }="2";
{\ar@3{=}^{ }   "1"+(0,-10);"2"+(0,-10)};
   \endxy  \xy  0;/r.30pc/:
(-10,0)*+{ \scriptscriptstyle F_1(X)}="1";
(10,0)*+{ \scriptscriptstyle F_1(Y)}="2";
(-10,-20)*+{ \scriptscriptstyle F_2(X)}="3";
(10,-20)*+{ \scriptscriptstyle F_2(Y)}="4";
{\ar@{->}^{ \scriptscriptstyle F_1(f)} "1";"2" };
{\ar@{->}_{ \scriptscriptstyle F_2(f)} "3";"4" };
{\ar@{->}_{ \scriptscriptstyle \rho_2(X)} "1";"3" };
{\ar@{->}|-{ \scriptscriptstyle \rho_2(Y)} "2";"4" };
{\ar@{=>}^{ \scriptscriptstyle \rho_2(f) } (7, -3)*{};(-7,-17)*{}} ;
  {\ar@/^3.85pc/^{ \scriptscriptstyle \rho_1(Y)} "2";"4" };
  {\ar@{<=}^{ \scriptscriptstyle \Phi(Y) } (14, -10)*{};(22,-10)*{}} ;
\endxy.
\end{equation}

\subsection{Strict $3$-categories}
A  {\it $3$-category} is a category enriched over the category of all small strict $2$-categories. In particular,
a {\it strict $3$-category } $  \mathcal{ C }$    consists of collections $ \mathcal C_0$ of objects, $ \mathcal C_1$ of $1$-arrows, $ \mathcal C_2$ of $2$-arrows,  and $ \mathcal C_3$ of $3$-arrows,
together with

$\bullet$ functions $s_n, t_n :  \mathcal C_i\rightarrow  \mathcal C_n$ for all $0 \leq n < i \leq 3$,  called
{\it $n$-source} and {\it $n$-target},

$\bullet$ functions $\#_n :  \mathcal C_{n+1} \times \mathcal C_{n+1}\rightarrow  \mathcal C_{n+1}$ for all $n=0,1,2$, called {\it vertical composition},

$\bullet$ a function  $\#_p :  \mathcal C_i  \times \mathcal C_{i}\rightarrow  \mathcal C_i$, $p+2\leq i $,  called the {\it  horizontal composition},

$\bullet$ a function  $ 1_{*}   :  \mathcal C_i \rightarrow  \mathcal C_{i+1 }$ for $i=0,1$, called {\it identity}.

For a $3$-arrow $\varphi: \xy 0;/r.22pc/:
(0,10)*{};
(0,-10)*{};
(-20,0)*+{x}="1";
(0,0)*+{y}="2";
{\ar@/^1.33pc/^f "1";"2"};
{\ar@/_1.33pc/_{f'} "1";"2"};
(-10,4)*+{}="A";
(-10,-4)*+{}="B";
{\ar@{=>}@/_.5pc/_\gamma "A"+(-1.33,0) ; "B"+(-.66,-.55)};
{\ar@{=}@/_.5pc/ "A"+(-1.33,0) ; "B"+(-1.33,0)};
{\ar@{=>}@/^.5pc/^{\gamma'} "A"+(1.33,0) ; "B"+(.66,-.55)};
{\ar@{=}@/^.5pc/ "A"+(1.33,0) ; "B"+(1.33,0)};
{\ar@3{->} (-12,0)*{}; (-7,0)*{}};
\endxy $,  its $2$-source  and   $2$-target are $\gamma$ and $\gamma'$ respectively. The $3$-arrows $\varphi$ and   $\varphi': \xy 0;/r.22pc/:
(0,10)*{};
(0,-10)*{};
(-20,0)*+{x}="1";
(0,0)*+{y}="2";
{\ar@/^1.33pc/^f "1";"2"};
{\ar@/_1.33pc/_{f'} "1";"2"};
(-10,4)*+{}="A";
(-10,-4)*+{}="B";
{\ar@{=>}@/_.5pc/_{\gamma'} "A"+(-1.33,0) ; "B"+(-.66,-.55)};
{\ar@{=}@/_.5pc/ "A"+(-1.33,0) ; "B"+(-1.33,0)};
{\ar@{=>}@/^.5pc/^{\gamma''} "A"+(1.33,0) ; "B"+(.66,-.55)};
{\ar@{=}@/^.5pc/ "A"+(1.33,0) ; "B"+(1.33,0)};
{\ar@3{->} (-12,0)*{}; (-7,0)*{}};
\endxy $   are $2$-composable, and their composition $\varphi\#_2\varphi'$ is $  \xy 0;/r.22pc/:
(0,10)*{};
(0,-10)*{};
(-20,0)*+{x}="1";
(0,0)*+{y}="2";
{\ar@/^1.33pc/^f "1";"2"};
{\ar@/_1.33pc/_{f'} "1";"2"};
(-10,4)*+{}="A";
(-10,-4)*+{}="B";
{\ar@{=>}@/_.5pc/_\gamma "A"+(-1.33,0) ; "B"+(-.66,-.55)};
{\ar@{=}@/_.5pc/ "A"+(-1.33,0) ; "B"+(-1.33,0)};
{\ar@{=>}@/^.5pc/^{\gamma''} "A"+(1.33,0) ; "B"+(.66,-.55)};
{\ar@{=}@/^.5pc/ "A"+(1.33,0) ; "B"+(1.33,0)};
{\ar@3{->} (-12,0)*{}; (-7,0)*{}};
\endxy $.

In a   strict $3$-category,
$0$-, $1$- and $2$-arrows behave as in a $2$-category.  We call two $3$-arrows $\varphi$ and $\psi$ {\it horizontally $p$-composable} if the   $p$-target of $\varphi$ coincides with the $p$-source  of $\psi$, $p=0,1$, and denote their horizontal composition as $\varphi\#_p\psi$.

For a $2$-arrow $\delta$,   $3$-arrows $1_\delta$  and $\varphi$   are {\it horizontally $1$-composable}  if    the   $1$-target of $\delta$ coincides with the $1$-source  of $\varphi$. In this case,
$$ \delta\#_1\varphi:=1_\delta\#_1\varphi \qquad\qquad
    \xy 0;/r.22pc/:
(0,10)*{};
(0,-10)*{};
(-20,0)*+{x}="1";
(0,0)*+{y}="2";
{\ar@/^1.33pc/ "1";"2"|-f};{\ar@/^3.63pc/^g "1";"2"};{\ar@{=>}^{\delta} (-10,16) ; (-10,8)};
{\ar@/_1.33pc/_{f'} "1";"2"};
(-10,4)*+{}="A";
(-10,-4)*+{}="B";
{\ar@{=>}@/_.5pc/_\gamma "A"+(-1.33,0) ; "B"+(-.66,-.55)};
{\ar@{=}@/_.5pc/ "A"+(-1.33,0) ; "B"+(-1.33,0)};
{\ar@{=>}@/^.5pc/^{\gamma'} "A"+(1.33,0) ; "B"+(.66,-.55)};
{\ar@{=}@/^.5pc/ "A"+(1.33,0) ; "B"+(1.33,0)};
{\ar@3{->} (-12,0)*{}; (-7,0)*{}};
\endxy,
$$
  is called {\it whiskering from above by a $2$-arrow $\delta$}. It is similar to define whiskering from below:
$$
    \xy 0;/r.22pc/:
(0,10)*{};
(0,-10)*{};
(-20,0)*+{x}="1";
(0,0)*+{y}="2";
{\ar@/^1.33pc/ "1";"2"^f};{\ar@/_3.63pc/_{g'} "1";"2"};{\ar@{=>}^{\eta} (-10,-8) ; (-10,-16)};
{\ar@/_1.33pc/"1";"2"|-{f'} };
(-10,4)*+{}="A";
(-10,-4)*+{}="B";
{\ar@{=>}@/_.5pc/_\gamma "A"+(-1.33,0) ; "B"+(-.66,-.55)};
{\ar@{=}@/_.5pc/ "A"+(-1.33,0) ; "B"+(-1.33,0)};
{\ar@{=>}@/^.5pc/^{\gamma'} "A"+(1.33,0) ; "B"+(.66,-.55)};
{\ar@{=}@/^.5pc/ "A"+(1.33,0) ; "B"+(1.33,0)};
{\ar@3{->} (-12,0)*{}; (-7,0)*{}};
\endxy
$$
There is also {\it whiskering from left ( or right) by a $1$-arrow} $A\#_0\varphi:=1_{1_A}\#_0\varphi$ (or $  \varphi\#_0 B$):
$$
   \xy 0;/r.22pc/:
(0,10)*{};
(0,-10)*{};
(-20,0)*+{x}="1";(-40,0)*+{z}="0";(20,0)*+{w}="10";
(0,0)*+{y}="2";
{\ar@/^1.33pc/^f "1";"2"};{\ar@{->}^{A}  "0";"1"};{\ar@{-->}^{B}  "2";"10"};
{\ar@/_1.33pc/_{f'} "1";"2"};
(-10,4)*+{}="A";
(-10,-4)*+{}="B";
{\ar@{=>}@/_.5pc/_\gamma "A"+(-1.33,0) ; "B"+(-.66,-.55)};
{\ar@{=}@/_.5pc/ "A"+(-1.33,0) ; "B"+(-1.33,0)};
{\ar@{=>}@/^.5pc/^{\gamma'} "A"+(1.33,0) ; "B"+(.66,-.55)};
{\ar@{=}@/^.5pc/ "A"+(1.33,0) ; "B"+(1.33,0)};
{\ar@3{->} (-12,0)*{}; (-7,0)*{}};
\endxy
$$

The properties of identities,  the associativity and  the compatibility condition for different compositions, similar to (\ref{eq:identities}) (\ref{eq:composition-associativity}) and  (\ref{eq:asso2}) for a strict $2$-category, also hold in a strict $3$-category.
See page 8 of \cite{Lei} for an explicit definition of   a strict $m$-category.

A {\it  strict  $3$-functor} (or a {\it    functor}) is a map     preserving compositions  and identities.

\begin{rem} In a strict $3$-category, the interchange  law (\ref{eq:interchanging-law}) for the horizontal composition of $2$-arrows is also satisfied.  But in general, a $3$-category does not satisfy the interchange  law.
   $\mathbf{Gray}$-categories are the greatest possible semi-strictification of   $3$-categories, and appear naturally in  $3$-gauge theory \cite{Wa}. The $3$-representation     in   a $\mathbf{Gray}$-category is more natural, but is much more complicated. So we restrict to the  $3$-representation   in   strict $3$-categories in this paper.
\end{rem}

In a  strict $3$-category $\mathcal{C}$, a $1$-arrow $B:x\rightarrow y$ is called {\it a  $1$-isomorphism} if there exists
 $1$-arrow $C:y\rightarrow x$ such that there exist   $2$-isomorphisms $u :1_y\Longrightarrow C\#_k B$ and  $v:1_x\Longrightarrow B\#_k C $. We  call $C$ a
 {\it quasi-inverse to $B$}, and vise versa. However, when $k=2$ or $3$,   we call   a $k$-arrow {\it a  $k$-isomorphism} if it is strictly invertible.

\subsection{The $3$-representations   of a group   in  a strict   $3$-category }

Let $\mathcal{C}$ be a strict $3$-category and let $G$ be a group. $G$ can be viewed as a strict $3$-category with only one object $\bullet$, $G$ as the set of $1$-arrows $g:\bullet\rightarrow \bullet$, the set of $2$-arrows
consisting of the identities of   $1$-arrows, and the set of $3$-arrows consisting of the identities of   $2$-arrows.
A {\it $3$-representation of a group $G$ in $\mathcal{C}$ } is a weak  functor $\rho$ from $G$ to
$\mathcal{C}$ in the following sense. We have

(1) an object $x$ of $\mathcal{C}$;

(2) for each $g\in G$, a $1$-isomorphism   $\rho_g :x\rightarrow x$;

(3) for each $h, g\in G$, a $2$-isomorphism
$
 \phi_{h ,g }:\rho_h \rho_g \Longrightarrow \rho_{h g}
$
 (here and in the following we write $\rho_h \#_0\rho_g$ as $\rho_h \rho_g$ for simplicity), corresponding to the $2$-cell
$$
     
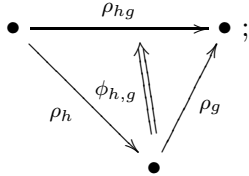
\xy0;/r.22pc/:
(-10,0)*+{ \bullet }="1";
(20,0)*+{\bullet }="2";
(10,-20)*+{ \bullet}="3";
{\ar@{->}^{\rho_{h g}  } "1";"2" };
{\ar@{->}_{\rho_h } "1";"3" };
{\ar@{<-}^{\rho_{ g}} "2";"3" };
{\ar@{<=}_{\phi_{h,g} } (8, -2)*{};(10,-15)*{}} ;
\endxy;
$$

(4) for each $g_3,g_2,g_1\in G$, a $3$-isomorphism, called the {\it associator},
\begin{equation}\label{eq:3-cocycle}
     \Phi_{ g_3,g_2,g_1}:(\rho_{g_3} \#_0\phi_{g_2,g_1} ) \#_1 \phi_{ g_3,g_2 g_1}\xy 0;/r.30pc/:
  (0,0)*+{ }="1";
(10,0)*+{ }="2";
{\ar@3{->}^{ }  "1" ;"2" };
   \endxy (\phi_{g_3,g_2}\#_0 \rho_{g_1})\#_1
   \phi_{g_3 g_2,g_1} ,
\end{equation} corresponding to the $3$-cell
$$
     
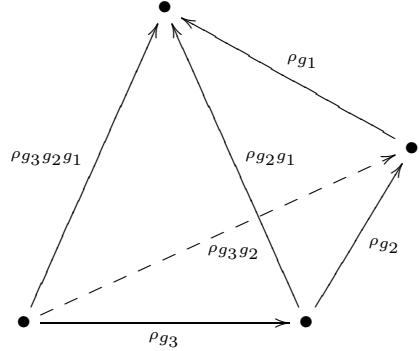
\xy0;/r.22pc/:
(0,0 )*+{\bullet}="1";
(40, 0)*+{\bullet}="2";
( 55,25)*+{\bullet}="3";
(20,45)*+{\bullet}="4";
 {\ar@{->} "1";"2"_{\scriptscriptstyle\rho_{ g_3 }} };
{\ar@{-->}_{\scriptscriptstyle\rho_{ g_3g_2 }  }"1"; "3"};
{\ar@{->} "2";"3"_{\scriptscriptstyle\rho_{  g_2 } } };
{\ar@{->}^{\scriptscriptstyle \rho_{ g_3g_2g_1} } "1";"4" };
{\ar@{->}_{\scriptscriptstyle\rho_{  g_2g_1}  } "2";"4" };
{\ar@{->} "3";"4"_{\scriptscriptstyle\rho_{ g_1}} };
 \endxy.
$$
It can be viewed as  exchanging the diagonals of the quadrilateral:
\begin{equation}\label{eq:diagonal}
     \xy0;/r.22pc/:
(0,0 )*+{\bullet}="1";
(40, 0)*+{\bullet}="2";
( 55,25)*+{\bullet}="3";
(20,45)*+{\bullet}="4";
 {\ar@{->} "1";"2"_{\scriptscriptstyle\rho_{ g_3 }} };
{\ar@{->} "2";"3"_{\scriptscriptstyle\rho_{  g_2 } } };
{\ar@{->}^{\scriptscriptstyle \rho_{ g_3g_2g_1} } "1";"4" };
{\ar@{->}^{\scriptscriptstyle\rho_{\scriptscriptstyle  g_2g_1}  } "2";"4" };
{\ar@{->} "3";"4"_{\scriptscriptstyle\rho_{ g_1}} };
{\ar@{=>}_{\scriptscriptstyle \phi_{g_2,g_1}} (52,25)*+{  };(33,15)*+{  } };
{\ar@{=>}^{\scriptscriptstyle \phi_{g_3,g_2g_1}} (35,3)*+{  };(10,19)*+{  } };
 \endxy \xy 0;/r.30pc/:
  (0,0)*+{ }="1";
(10,0)*+{ }="2";
{\ar@3{->}^{ }   "1"+(0,18);"2"+(0,18)};
   \endxy
 \xy0;/r.22pc/:
(0,0 )*+{\bullet}="1";
(40, 0)*+{\bullet}="2";
( 55,25)*+{\bullet}="3";
(20,45)*+{\bullet}="4";
 {\ar@{->} "1";"2"_{\scriptscriptstyle\rho_{ g_3 }} };
{\ar@{-->}^{\scriptscriptstyle\rho_{ g_3g_2 }  }"1"; "3"};
{\ar@{->} "2";"3"_{\scriptscriptstyle\rho_{  g_2 } } };
{\ar@{->}^{\scriptscriptstyle \rho_{ g_3g_2g_1} } "1";"4" };
{\ar@{->} "3";"4"_{\scriptscriptstyle\rho_{ g_1}} };
{\ar@{=>}_{ \scriptscriptstyle\phi_{g_3g_2,g_1}} (45,25)*+{  };(15,25)*+{  } };
{\ar@{=>}_{\scriptscriptstyle \phi_{g_3,g_2 }} (38,1)*+{  };(28,13)*+{  } };
 \endxy,\end{equation}
which can also be drawn   in the following form:
  \begin{equation}\label{eq:associator}
  \xy 0;/r.28pc/:
 (-15,0)*+{ {x}}="1";
(-30,0)*+{ x}="3";
(-45,0)*+{ x}="5";(-60,0)*+{ x}="7";(-50,-5)*+{\scriptscriptstyle{\phi_{g_3,g_2g_1}}  }="07";
(-30,-11)*+{  }="11";{\ar@{=>}^{\scriptscriptstyle{\phi_{ g_2,g_1}} } "3"; "11"};
(-45,-13)*+{ }="12";{\ar@{=>}_{} "5"; "12"};
{\ar@{-->}^{\scriptscriptstyle\rho_{g_1}} "3"; "1"};
{\ar@{-->}^{\scriptscriptstyle\rho_{g_2}} "5"; "3"};
{\ar@{-->}^{\scriptscriptstyle\rho_{g_3}} "7"; "5"};
 {\ar@{-->}@/_4.13pc/_{\scriptscriptstyle\rho_{g_3g_2g_1}} "7";"1"};
  {\ar@{-->}@/_2.99pc/|-{\scriptscriptstyle\rho_{g_2g_1}}"5";"1"};
   \endxy\xy0;/r.28pc/:
  (0,0)*+{ }="1";
(10,0)*+{ }="2";
{\ar@3{->}^{ }   "1"+(0,-4);"2"+(0,-4)};
   \endxy   \xy 0;/r.282pc/:
 (-15,0)*+{ {x}}="1";
(-30,0)*+{ x}="3";
(-45,0)*+{ x}="5";(-60,0)*+{ x}="7";(-45,-11)*+{  }="12";(-25,-5)*+{ \scriptscriptstyle{\phi_{g_3g_2,g_1}}   }="17";
(-30,-13)*+{  }="11";{\ar@{=>}^{\scriptscriptstyle{\phi_{g_3,g_2 }} } "5"; "12"};{\ar@{=>}^{} "3"; "11"};
{\ar@{-->}^{\scriptscriptstyle\rho_{g_1}} "3"; "1"};
{\ar@{-->}^{\scriptscriptstyle\rho_{g_2}} "5"; "3"};
{\ar@{-->}^{\scriptscriptstyle\rho_{g_3}} "7"; "5"};
 {\ar@{-->}@/_4.13pc/_{\scriptscriptstyle\rho_{g_3g_2g_1}} "7";"1"};
  {\ar@{-->}@/_2.99pc/|-{\scriptscriptstyle\rho_{g_3g_2}} "7";"3"};
   \endxy;
 \end{equation}

(5) a $2$-isomorphism
$
     \phi_1:\rho_1\Longrightarrow 1_x;
$

such that the following conditions are satisfied:

$\bullet$ $\phi_{1,g}=\phi_1\#_0 \rho_g,$  $\phi_{ g, 1}=\rho_g\#_0\phi_1 $.

$\bullet$ the {\it $3$-cocycle condition} that for any $g_4,\ldots, g_1\in G$, we have
\begin{equation}\label{eq:4-cocycle} \begin{array}{r}
 \left  \{ [\rho_{g_4}\#_0\Phi_{g_3,g_2,g_1}]\#_1 \phi_{g_4,g_3g_2 g_1} \right\}\#_2\left\{[ \rho_{g_4}\#_0\phi_{ g_3 ,g_2}\#_0\rho_{g_1}
]\#_1\Phi_{g_4,g_3g_2, g_1}\right\}\,\,\,\\ \qquad \#_2\left\{[\Phi_{g_4,g_3,g_2 }\#_0\rho_{g_1}] \#_1   \phi_{g_4 g_3g_2, g_1}  \right\}\,\,\,\\
= \left  \{[(\rho_{g_4}\rho_{g_3})\#_0\phi_{ g_2, g_1} ]\#_1 \Phi_{g_4,g_3,g_2 g_1} \right\}\#_2\left\{[\phi_{ g_4, g_3}\#_0 (\rho_{g_2}\rho_{g_1})]\#_1
\Phi_{g_4g_3,g_2, g_1}\right\}.
 \end{array}\end{equation}

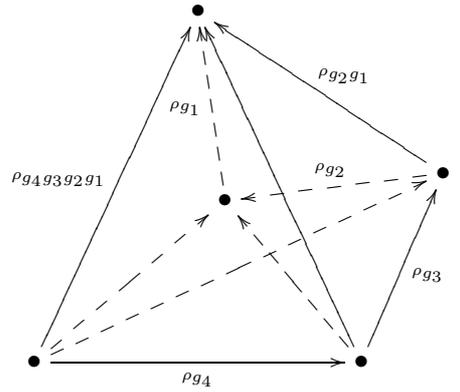
$$
     \xy 0;/r.17pc/:
(0,0 )*+{\bullet}="1";
(60, 0)*+{\bullet}="2";
( 75,35)*+{\bullet}="3";
(35,30)*+{\bullet}="4";
(30,65)*+{\bullet}="5";
 {\ar@{->}_{\scriptscriptstyle\rho_{\scriptscriptstyle g_4 }} "1";"2" };
{\ar@{-->}|-{  }"1"; "3"};
{\ar@{->}_{\scriptscriptstyle\rho_{ g_3} } "2";"3" };
{\ar@{-->}|-{ } "1";"4" };
{\ar@{-->}|-{ } "2";"4" };
{\ar@{-->}_{\scriptscriptstyle\rho_{g_2} } "3";"4" };
{\ar@{->}^{\scriptscriptstyle\rho_{\scriptscriptstyle g_4 g_3g_2 g_1} } "1";"5" };
{\ar@{->}|-{ } "2";"5" };
{\ar@{->}_{\scriptscriptstyle \rho_{ g_2 g_1}} "3";"5" };{\ar@{-->}^{\scriptscriptstyle\rho_{ g_1} } "4";"5" };
 \endxy.
$$
Equivalently, the composition of the $3$-isomorphisms represented by $5$ tetrahedrons above in the boundary of a $4$-simplex is the identity. This comes from the fact that the boundary of the corresponding $4$-simplex in the $3$-category $G$ is the identity $3$-arrow.

\begin{rem} (1)
For simplicity, we assume in this paper that $\rho_1=1_x$ and that
    $ \phi_1$   is the identity.

    (2) The $3$-cocycle $\{ \Phi_{ g_3,g_2,g_1}\}$ defines an element of the  $3$-dimensional non-abelian   cohomology. A first attempt at an explicit description of the  $3$-dimensional non-abelian  cohomology
    of a group goes back to Dedecker \cite{De}. See section 4 of \cite{Br94} for $3$-dimensional non-abelian $\rm \check{C}$ech cocycles, which can be used to construct a $2$-gerbe.
    \end{rem}
\subsection{The $3$-cocycle condition }   We will give a clear geometric  description of the $3$-cocycle condition (\ref{eq:4-cocycle}) in terms of $5$ tetrahedrons   in the boundary of a $4$-simplex above. This is equivalent to triviality  of the $3$-holonomy.  See section 5 C of \cite{Wa} for the  $3$-holonomy in the lattice $3$-gauge theory (the cubical case), where
      $3$-gauge theory from the point of view of $\mathbf{Gray}$-categories is investigated.

In the left-hand side of the $3$-cocycle condition (\ref{eq:4-cocycle}), the first $3$-isomorphism is
\begin{equation}\label{eq:A1}
A_1= [\rho_{g_4}\#_0\Phi_{g_3,g_2,g_1}]\#_1 \phi_{g_4,g_3g_2 g_1} .
\end{equation}
Here $ \Phi_{g_3,g_2,g_1} $ is a $3$-isomorphism whiskered from left by the $1$-isomorphism $\rho_{g_4}$, and  $\rho_{g_4}\#_0\Phi_{g_3,g_2,g_1} $ is   whiskered from below by the $2$-isomorphism $\phi_{g_4,g_3g_2 g_1}$.  $A_1$ corresponds to the $3$-cell
 $$
     \xy 0;/r.17pc/:
(0,0 )*+{\bullet}="1";
(60, 0)*+{\bullet}="2";
( 75,35)*+{\bullet}="3";
(40,25)*+{\bullet}="4";
(30,65)*+{\bullet}="5";(43,46)*+{\scriptscriptstyle\rho_{b}}="15";
(47,12)*+{\scriptscriptstyle\rho_{a}}="25";
 {\ar@{->}^{\scriptscriptstyle\rho_{g_4}} "1";"2" };
{\ar@{->}_{\scriptscriptstyle\rho_{g_3} } "2";"3" };
{\ar@{->}^{\scriptscriptstyle\rho_{g_4 g_3g_2 g_1} } "1";"5" };
{\ar@{-->}|-{  } "2";"4" };
{\ar@{-->}_{\scriptscriptstyle\rho_{g_2 }} "3";"4" };
{\ar@{->}|-{  } "2";"5" };
{\ar@{->}_{\scriptscriptstyle\rho_{g_2 g_1} } "3";"5" };{\ar@{-->}^{ \rho_{g_1} } "4";"5" };
( 30,-8)*+{ \scriptstyle    {\rm The}\, \,{\rm 3-arrow}\,     A_1   }="30";
{\ar@{=>}^{\scriptscriptstyle\phi_{g_4,g_3g_2 g_1} } (48, 3);(16,28)};
 \endxy
$$whose $2$-source and $2$-target are the $2$-isomorphisms
\begin{equation}\label{eq:s-A1}\begin{array}{l}
 s_2(A_1)=  [(\rho_{g_4}\rho_{g_3})\#_0\phi_{ g_2 ,g_1}]\#_1[\rho_{g_4}\#_0\phi_{ g_3,g_2 g_1}]\#_1 \phi_{g_4,g_3g_2 g_1}: \rho_{g_4}\rho_{g_3}\rho_{g_2}\rho_{g_1}\longrightarrow\rho_{g_4  g_3  g_2 g_1}, \\
 t_2(A_1)=  [ \rho_{g_4}\#_0\phi_{ g_3 ,g_2}\#_0\rho_{g_1} ]\#_1[\rho_{g_4}\#_0\phi_{ g_3g_2, g_1}]\#_1 \phi_{g_4,g_3g_2 g_1} : \rho_{g_4}\rho_{g_3}\rho_{g_2}\rho_{g_1}\longrightarrow\rho_{g_4  g_3  g_2 g_1},
\end{array}\end{equation}
 corresponding to   $2$-cells
  $$
     \xy 0;/r.17pc/:
(0,0 )*+{\bullet}="1";
(60, 0)*+{\bullet}="2";
( 75,35)*+{\bullet}="3";
(40,25)*+{\bullet}="4";
(30,65)*+{\bullet}="5";
 {\ar@{->}^{\scriptscriptstyle\rho_{g_4}} "1";"2" };
{\ar@{->}_{\scriptscriptstyle\rho_{g_3} } "2";"3" };
{\ar@{->}^{\scriptscriptstyle\rho_{g_4 g_3g_2 g_1 } } "1";"5" };
{\ar@{-->}_{\scriptscriptstyle\rho_{g_2 }} "3";"4" };
{\ar@{->}|-{ } "2";"5" };
{\ar@{->}_{\scriptscriptstyle\rho_{g_2 g_1} } "3";"5" };{\ar@{-->}^{\scriptscriptstyle \rho_{g_1} } "4";"5" };( 30,-8)*+{ \scriptstyle    {\rm The} \,\,  2-{\rm arrow}\,\,
s_2(A_1) }="30";{\ar@{==>}|-{ } (40, 28);(50,50)};{\ar@{=>}|-{ } (69, 35);(55,20)};{\ar@{=>}|-{ } (45, 3);(20,28)};
 \endxy\xy 0;/r.25pc/:
  (0,0)*+{ }="1";
(15,0)*+{ }="2";
{\ar@3{->}^{ }   "1"+(0,23);"2"+(0,23)};
   \endxy
     \xy 0;/r.17pc/:
(0,0 )*+{\bullet}="1";
(60, 0)*+{\bullet}="2";
( 75,35)*+{\bullet}="3";
(40,25)*+{\bullet}="4";
(30,65)*+{\bullet}="5";
(43,46)*+{\scriptscriptstyle\rho_{b}}="15";
(47,12)*+{\scriptscriptstyle\rho_{a}}="25";
 {\ar@{->}^{\scriptscriptstyle\rho_{g_4}} "1";"2" };
{\ar@{->}_{\scriptscriptstyle\rho_{g_3 }} "2";"3" };
{\ar@{-->} "2";"4"   };
{\ar@{->}_{\scriptscriptstyle\rho_{g_2 }} "3";"4" };
{\ar@{->}^{\scriptscriptstyle\rho_{g_4 g_3g_2 g_1 } } "1";"5" };
{\ar@{->}_{}"2";"5" };
{\ar@{-->}^{\scriptscriptstyle\rho_{g_1} } "4";"5" };( 30,-8)*+{ \scriptstyle   {\rm The} \,\,  2-{\rm arrow}\, \,   t_2(A_1) }="30";
{\ar@{=>}|-{ } (69, 30);(55,20)};{\ar@{=>}|-{ } (45, 3);(20,28)};{\ar@{==>}|-{ } (42, 23);(42.6,38)};
 \endxy
$$
respectively, where $\rho_a :=\rho_{g_3g_2  }, \rho_b :=\rho_{g_3g_2 g_1}$. It is fundamental in this paper to write down the   $p$-arrow corresponding to $p$-cells as whiskered vertical compositions.
For example,  $s_2(A_1)$ in (\ref{eq:s-A1}) is the composition of the following three whiskered $2$-isomorphisms.
 $$
     
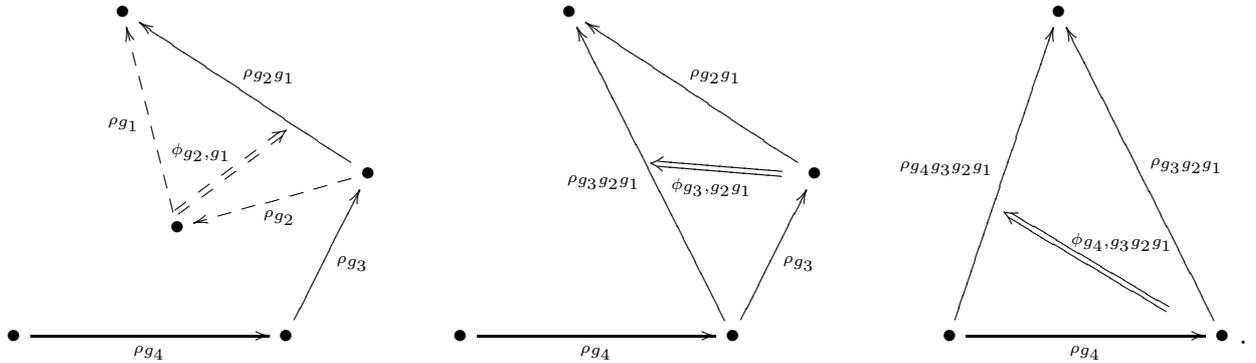
\xy 0;/r.17pc/:(0,0 )*+{\bullet}="1";
(50, 0)*+{\bullet}="2";
( 65,30)*+{\bullet}="3";
(30,20)*+{\bullet}="4";
(20,60)*+{\bullet}="5";{\ar@{->}_{\scriptscriptstyle\rho_{g_4}} "1";"2" };
 {\ar@{->}_{\scriptscriptstyle\rho_{g_3} } "2";"3" };
{\ar@{-->}^{\scriptscriptstyle\rho_{g_2 }} "3";"4" };
{\ar@{->}_{\scriptscriptstyle\rho_{g_2 g_1} } "3";"5" };{\ar@{-->}^{\scriptscriptstyle \rho_{g_1} } "4";"5" }; {\ar@{==>}^{\scriptscriptstyle \phi_{ g_2 ,g_1}} (30, 23);(50,38)};
 \endxy \qquad
 \xy 0;/r.17pc/:
(0,0 )*+{\bullet}="1";
(50, 0)*+{\bullet}="2";( 65,30)*+{\bullet}="3";
(20,60)*+{\bullet}="5";
 {\ar@{->}_{\scriptscriptstyle\rho_{g_4}} "1";"2" };
{\ar@{->}_{\scriptscriptstyle\rho_{g_3} } "2";"3" };
{\ar@{->}^{\scriptscriptstyle\rho_{ g_3 g_2 g_1} } "2";"5" };
{\ar@{->}_{\scriptscriptstyle\rho_{g_2 g_1} } "3";"5" }; {\ar@{=>}^{\scriptscriptstyle\phi_{ g_3,g_2 g_1} } (59, 30);(35,32)};
 \endxy \qquad \xy 0;/r.17pc/:
(0,0 )*+{\bullet}="1";
(50, 0)*+{\bullet}="2";
(20,60)*+{\bullet}="5";
 {\ar@{->}_{\scriptscriptstyle\rho_{g_4}} "1";"2" };
{\ar@{->}^{\scriptscriptstyle\rho_{g_4 g_3g_2 g_1 } } "1";"5" };
{\ar@{->}_-{\scriptscriptstyle\rho_{  g_3g_2 g_1} } "2";"5" };
 {\ar@{=>}_{\scriptscriptstyle \phi_{g_4,g_3g_2 g_1}} (40, 5);(10,23)};
 \endxy.
$$

The second $3$-isomorphism in the left-hand side of the $3$-cocycle condition (\ref{eq:4-cocycle}) is
$
   A_2=   [ \rho_{g_4}\#_0\phi_{ g_3 ,g_2}\#_0\rho_{g_1} ] \#_1 \Phi_{g_4,g_3 g_2  ,g_1} ,
$
corresponding to the $3$-cell
$$
     \xy 0;/r.17pc/:
(0,0 )*+{\bullet}="1";
(60, 0)*+{\bullet}="2";
( 75,35)*+{\bullet}="3";
(40,25)*+{\bullet}="4";
(30,65)*+{\bullet}="5";
 {\ar@{->}^{\scriptscriptstyle\rho_{g_4}} "1";"2" };
{\ar@{->}_{\scriptscriptstyle\rho_{g_3 }} "2";"3" };
{\ar@{-->}_{ \rho_a  } "1";"4" };
{\ar@{-->}^{\scriptscriptstyle\rho_{b} } "2";"4" };
{\ar@{->}_{\scriptscriptstyle\rho_{g_2 }} "3";"4" };
{\ar@{->}|-{ } "1";"5" };
{\ar@{->}|-{ } "2";"5" };
{\ar@{-->}^{\scriptscriptstyle\rho_{g_1} } "4";"5" };( 30,-8)*+{ \scriptstyle   {\rm The} \,\, 3-{\rm arrow}\,\,    A_2}="30";
 \endxy
$$(here $\rho_a :=\rho_{g_4g_3g_2  }, \rho_b :=\rho_{g_3g_2  }$) with  $2$-source $s_2(A_2)=  t_2(A_1)$ in (\ref{eq:s-A1}) and  $2$-target
\begin{equation}\label{eq:t-A2}
 t_2(A_2)=  [ \rho_{g_4}\#_0\phi_{ g_3 ,g_2}\#_0\rho_{g_1} ]\#_1[\phi_{g_4, g_3g_2 }\#_0\rho_{g_1}]\#_1 \phi_{g_4g_3g_2 ,g_1}
 \end{equation}corresponding to   $2$-cells
$$
     \xy 0;/r.17pc/:
(0,0 )*+{\bullet}="1";
(60, 0)*+{\bullet}="2";
( 75,35)*+{\bullet}="3";
(40,25)*+{\bullet}="4";
(30,65)*+{\bullet}="5";
 {\ar@{->}^{\scriptscriptstyle\rho_{g_4}} "1";"2" };
{\ar@{->}_{\scriptscriptstyle\rho_{g_3 }} "2";"3" };
{\ar@{->}^{\scriptscriptstyle\rho_a } "1";"4" };
{\ar@{->}_{\scriptscriptstyle\rho_b  } "2";"4" };
{\ar@{->}_{\scriptscriptstyle\rho_{g_2 }} "3";"4" };
{\ar@{->}^{\scriptscriptstyle \rho_{ g_4 g_3g_2g_1 }} "1";"5" };
{\ar@{->}_{\scriptscriptstyle\rho_{g_1} } "4";"5" };
 {\ar@{=>}  (70,32);(59,15) };
  {\ar@{=>}  (55,3);(26,15) }; {\ar@{=>}  (38,25);(20,35) };
( 30,-8)*+{ \scriptstyle   {\rm The} \, \, 2-{\rm arrow}\,   t_2( A_2)}="30";
 \endxy.
$$
And the third $3$-isomorphism in the left-hand side of the $3$-cocycle condition (\ref{eq:4-cocycle}) is
$$
A_3=[\Phi_{g_4,g_3,g_2 }\#_0\rho_{g_1}] \#_1   \phi_{g_4 g_3 g_2 ,g_1} ,
$$corresponding to the $3$-cell
$$
     \xy 0;/r.17pc/:
(0,0 )*+{\bullet}="1";
(60, 0)*+{\bullet}="2";
( 75,35)*+{\bullet}="3";
(35,30)*+{\bullet}="4";
(30,65)*+{\bullet}="5";
 {\ar@{->}^{\scriptscriptstyle\rho_{ g_4 }} "1";"2" };
{\ar@{-->}|-{  }"1"; "3"};
{\ar@{->}_{\scriptscriptstyle\rho_{ g_3} } "2";"3" };
{\ar@{->}^{\scriptscriptstyle\rho_{a}} "1";"4" };
{\ar@{->}|-{ } "2";"4" };
{\ar@{->}_{\scriptscriptstyle\rho_{g_2} } "3";"4" };
{\ar@{->}|-{ } "1";"5" };
 {\ar@{->}_{\scriptscriptstyle\rho_{ g_1} } "4";"5" }; {\ar@{=>}_{ \phi_b} (30, 30);(20,42)};
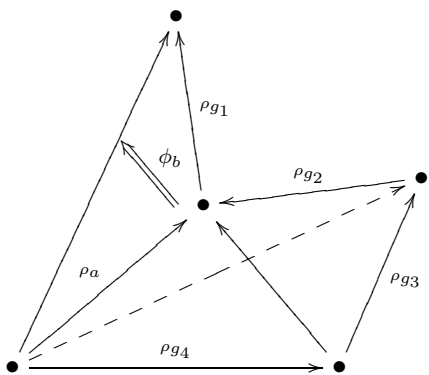
 ( 30,-8)*+{ \scriptstyle    {\rm The} \, \, 3-{\rm arrow}\, \,   A_3 }="30";
 \endxy
$$
(here $\rho_a :=\rho_{g_4g_3g_2  }$, $\phi_b:=\phi_{g_4 g_3 g_2 ,g_1}$)  with $2$-source   $s_2(A_3)=  t_2(A_2)$  in (\ref{eq:t-A2})  and  $2$-target
  \begin{equation}\label{eq:t-A3}
  t_2(A_3)=   [ \phi_{ g_4 ,g_3}\#_0(\rho_{g_2}\rho_{g_1}) ]\#_1[\phi_{ g_4g_3, g_2}\#_0\rho_{g_1}]\#_1 \phi_{g_4g_3g_2, g_1},
  \end{equation}corresponding to  the $2$-cells
  \begin{equation}\label{eq:t-A3'}
     \xy 0;/r.17pc/:
(0,0 )*+{\bullet}="1";
(60, 0)*+{\bullet}="2";
( 75,35)*+{\bullet}="3";
(35,30)*+{\bullet}="4";
(30,65)*+{\bullet}="5";
 {\ar@{->}^{\scriptscriptstyle\rho_{g_4}} "1";"2" };
{\ar@{->}_{\scriptscriptstyle\rho_{ b  } }"1"; "3"};
{\ar@{->}_{\scriptscriptstyle\rho_{g_3 }} "2";"3" };
{\ar@{->}^{\scriptscriptstyle\rho_a } "1";"4" };
{\ar@{->}_{\scriptscriptstyle\rho_{g_2} } "3";"4" };
{\ar@{->}^{\scriptscriptstyle\rho_{ g_4 g_3 g_2g_1} } "1";"5" };
 {\ar@{->}_{\scriptscriptstyle\rho_{g_1} } "4";"5" };( 30,-8)*+{ \scriptstyle   {\rm The} \,\,  2-{\rm arrow}\, \,   t_2(A_3) }="30";
 {\ar@{=>}|-{ } (57, 2);(45,10)};
  {\ar@{=>}|-{ } (50, 27);(35,23)};
   {\ar@{=>}|-{ } (30, 30);(20,42)};
 \endxy,
    \end{equation}where $\rho_a :=\rho_{g_4g_3g_2  }$, $\rho_b:=\rho_{g_4 g_3  }$.
  Then the composition $A_1\#_2A_2\#_2A_3$ of  $3$-isomorphisms is the left-hand side of the $3$-cocycle condition (\ref{eq:4-cocycle}), whose $2$-source is $s_2(A_1)$ in (\ref{eq:s-A1}) and $2$-target is
  $t_2(A_3)$ in (\ref{eq:t-A3}).

On the right-hand side of the $3$-cocycle condition (\ref{eq:4-cocycle}),  the first $3$-isomorphism is
$$
   A_1' = [(\rho_{g_4}\rho_{g_3})\#_0\phi_{ g_2, g_1} ]\#_1 \Phi_{g_4,g_3,g_2 g_1} ,
$$
corresponding to the $3$-cell$$
     \xy 0;/r.17pc/:
(0,0 )*+{\bullet}="1";
(60, 0)*+{\bullet}="2";
( 75,35)*+{\bullet}="3";
(35,30)*+{\bullet}="4";
(30,65)*+{\bullet}="5";
 {\ar@{->}^{\scriptscriptstyle\rho_{ g_4 }} "1";"2" };
{\ar@{-->}|-{  }"1"; "3"};
{\ar@{->}_{\scriptscriptstyle\rho_{ g_3} } "2";"3" };
{\ar@{-->}_{\scriptscriptstyle\rho_{g_2} } "3";"4" };
{\ar@{->}|-{ } "1";"5" };
{\ar@{->}|-{ } "2";"5" };
{\ar@{->}_{\scriptscriptstyle\rho_{g_2g_1} } "3";"5" };{\ar@{-->}^{\scriptscriptstyle\rho_{ g_1} } "4";"5" };( 30,-8)*+{ \scriptstyle    {\rm The} \,\,  3-{\rm arrow}\, \,    A_1'  }="30";
 \endxy
$$
with $2$-source   $s_2(A_1)$ in (\ref{eq:s-A1}) and   $2$-target   \begin{equation}\label{eq:t-A1'}
 t_2(A_1')=  [ (\rho_{g_4}\rho_{g_3}) \#_0\phi_{ g_2 ,g_1} ]\#_1[\phi_{g_4, g_3  }\#_0 \rho_{g_2 g_1} ]\#_1 \phi_{g_4 g_3,g_2 g_1} ,
 \end{equation} corresponding to the left $2$-cells in the following diagram:
$$
    
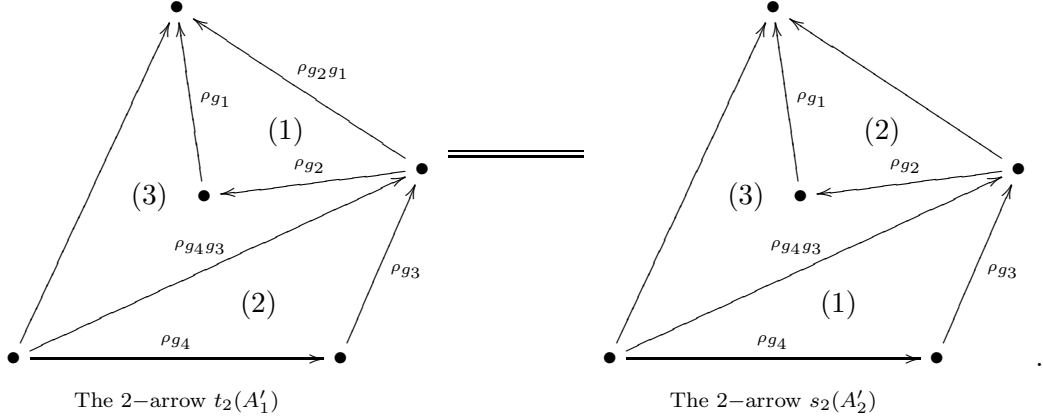
\xy 0;/r.17pc/:
(0,0 )*+{\bullet}="1";
(60, 0)*+{\bullet}="2";
( 75,35)*+{\bullet}="3";
(35,30)*+{\bullet}="4";
(30,65)*+{\bullet}="5";
(45,10)*+{(2)}="11";
(50,42)*+{(1)}="12";
(25,30)*+{(3)}="13";
 {\ar@{->}^{\scriptscriptstyle\rho_{g_4}} "1";"2" };
{\ar@{->}^{\scriptscriptstyle\rho_{g_4g_3} }"1"; "3"};
{\ar@{->}_{\scriptscriptstyle\rho_{g_3} } "2";"3" };
{\ar@{->}_{\scriptscriptstyle\rho_{g_2} } "3";"4" };
{\ar@{->}|-{ } "1";"5" };
{\ar@{->}_{\scriptscriptstyle\rho_{g_2g_1} } "3";"5" };{\ar@{->}_{\scriptscriptstyle\rho_{g_1} } "4";"5" };( 30,-8)*+{ \scriptstyle    {\rm The} \,\,  2-{\rm arrow}\,\,
t_2(A_1') }="30";
 \endxy \xy0;/r.28pc/:
  (0,0)*+{ }="1";
(15,0)*+{ }="2";
{\ar@{=}^{ }   "1"+(0,23);"2"+(0,23)};
   \endxy\xy 0;/r.17pc/:
(0,0 )*+{\bullet}="1";
(60, 0)*+{\bullet}="2";
( 75,35)*+{\bullet}="3";
(35,30)*+{\bullet}="4";
(30,65)*+{\bullet}="5";
(42,10)*+{(1)}="11";
(50,42)*+{(2)}="12";
(25,30)*+{(3)}="13";
 {\ar@{->}^{\scriptscriptstyle\rho_{g_4}} "1";"2" };
{\ar@{->}^{\scriptscriptstyle\rho_{g_4g_3} }"1"; "3"};
{\ar@{->}_{\scriptscriptstyle\rho_{ g_3}} "2";"3" };
{\ar@{->}_{\scriptscriptstyle\rho_{g_2} } "3";"4" };
{\ar@{->}|-{ } "1";"5" };
{\ar@{->}|-{ } "3";"5" };{\ar@{->}_{\scriptscriptstyle\rho_{g_1} } "4";"5" };( 30,-8)*+{ \scriptstyle   {\rm The} \,\,  2-{\rm arrow}\, \,   s_2( A_2') }="30";
 \endxy.
$$ By the interchange  law
(\ref{eq:interchanging-law}) for horizontal compositions,
we can interchange  $2$-isomorphism  (1) and (2) identically in the left $2$-cells above to get the $2$-isomorphism
\begin{equation}\label{eq:s-A2'}
s_2(A_2') =[\phi_{g_4, g_3
}\#_0(\rho_{g_2}\rho_{g_1})]\#_1 [  \rho_{g_4 g_3}  \#_0\phi_{ g_2 ,g_1} ]\#_1\phi_{g_4 g_3,g_2 g_1},
\end{equation} corresponding to  the right $2$-cells above.
 The last $3$-isomorphism is
$$A_2'=[\phi_{g_4,g_3 }\#_0(\rho_{g_2}\rho_{g_1})]\#_1
 \Phi_{g_4g_3,g_2,g_1}
$$
$$
    
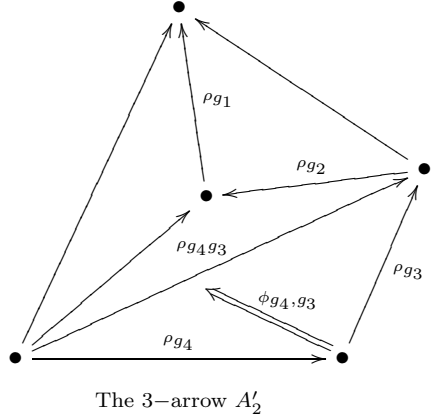
\xy 0;/r.17pc/:
(0,0 )*+{\bullet}="1";
(60, 0)*+{\bullet}="2";
( 75,35)*+{\bullet}="3";
(35,30)*+{\bullet}="4";
(30,65)*+{\bullet}="5";
 {\ar@{->}^{\scriptscriptstyle\rho_{g_4}} "1";"2" };
{\ar@{->}^{\scriptscriptstyle\rho_{g_4 g_3} }"1"; "3"};{\ar@{->}|-{ }"1"; "4"};
{\ar@{->}_{\scriptscriptstyle\rho_{ g_3}} "2";"3" };
{\ar@{->}_{\scriptscriptstyle\rho_{g_2} } "3";"4" };
{\ar@{->}|-{ } "1";"5" };
{\ar@{->}|-{ } "3";"5" };{\ar@{->}_{\scriptscriptstyle\rho_{g_1 }} "4";"5" };{\ar@{=>}_{\scriptscriptstyle\phi_{g_4,g_3 } } (58, 2);(35,13)};
( 30,-8)*+{ \scriptstyle   {\rm The} \,\,  3-{\rm arrow}\, \,    A_2' }="30";
 \endxy
$$
whose $2$-target is exactly the $2$-isomorphism $t_2(A_3)$ in   (\ref{eq:t-A3})-(\ref{eq:t-A3'}).

It is not easy to draw  several $3$-cells corresponding to the composition of $3$-arrows  in a $3$-category $\mathcal{C}$.
For this reason,  let us consider the associated $2$-category    $\mathcal{C}^{+}$  such that
$$
   (\mathcal{C}^{+})_i:=\mathcal{C}_{i+1} ,
$$
and    $i$-source
and    $i$-target are  $s_{i+1}$  and $t_{i+1}$,  $i=0,1,2$, respectively. Functions $\widetilde{\#}_p:\mathcal{C}^{+}_k\times
\mathcal{C}^{+}_k\longrightarrow  \mathcal{C}^{+}_k$ are described by  arrows $ {\#}_{p+1}:\mathcal{C}_{k+1}\times
\mathcal{C}_{k+1}\longrightarrow  \mathcal{C}_{k+1}$, and identities $\widetilde{1}:\mathcal{C}_{k-1}^{+}\rightarrow\mathcal{C}_{k }^{+}$ are defined in a similar manner. $\mathcal{C}^{+}$ is a strict $2$-category since $Hom_{\mathcal{C}}(x,y)$ is a strict $2$-category for any objects $x,y$ of $\mathcal{C}$, by the fact that a strict  $3$-category is a category enriched over the category of all small strict $2$-categories.  We also define $\mathcal{C}^{++}$ to be the  category with
$$
   (\mathcal{C}^{++})_i:=\mathcal{C}_{i+2}
$$
and the $i$-source
and  $i$-target are now $s_{i+2}$  and $t_{i+2}$,  $i=0,1 $, respectively. The function  $\widetilde{\#}_0:\mathcal{C}^{++}_1\times
\mathcal{C}^{++}_1\longrightarrow  \mathcal{C}^{++}_1$ becomes $ {\#}_{ 2}:\mathcal{C}_{3}\times
\mathcal{C}_{3}\longrightarrow  \mathcal{C}_{3}$.  $\mathcal{C}^{++}$ is a  category by the same reason.

In the corresponding strict $2$-category $\mathcal{C}^{+}$, $3$-isomorphism $A_1$ in (\ref{eq:A1}) is represented by  the following $2$-isomorphism:
\begin{equation}\label{eq:Omega0}
   \xy0;/r.22pc/:
   (-10,0 )*+{  \rho_{g_4}\rho_{g_3}\rho_{g_2}\rho_{g_1}}="1";
( 45,15)*+{ \rho_{g_4}\rho_{g_3 g_2}\rho_{g_1} }="3";
(-10,40 )*+{  \rho_{g_4}\rho_{g_3}\rho_{g_2 g_1}  }="5";
(  45,55)*+{  \rho_{g_4}\rho_{g_3 g_2 g_1} }="7";
(95,55)*+{  \rho_{g_4  g_3  g_2g_1} }="8";
(25,3)*+{ \scriptstyle \rho_{g_4} \#_0 \phi_{  g_3,g_2 }\#_0 \rho_{g_1} }="08";
(16,52)*+{ \scriptstyle  \rho_{g_4 }\#_0\phi_{g_3, g_2g_1}  }="09";
{\ar@{<-}^{  } "3";"1" };
{\ar@{<-}_{ (\rho_{g_4}\rho_{g_3})\#_0 \phi_{g_2,g_1}} "5";"1" };
 {\ar@{<-}^{  \rho_{g_4} \#_0\phi_{g_3 g_2,g_1 } } "7";"3" };
 {\ar@{<-}_{} "7";"5" }; {\ar@{->}"7";"8"^{ \phi_{g_4, g_3g_2 g_1}  } };
 {\ar@{=>}^{ A_1  } (3,35)*{};(43,20)*{}} ;
 \endxy
   \end{equation} Here  the upper and lower boundaries in (\ref{eq:Omega0}) (as $1$-arrows in $\mathcal{C}^{+}$)  represent the source  $s_2(A_1) $ and target  $t_2(A_1) $ in (\ref{eq:s-A1}) (as  $2$-isomorphisms in $\mathcal{C} $)
  respectively.
 To draw the picture neatly, we omit the whiskering parts. Then the $3$-cocycle condition
(\ref{eq:4-cocycle}) can be expressed simply as an identity of $2$-isomorphisms in $\mathcal{C}^{+}$ as follows:
\begin{equation}\label{eq:Omega}
   \xy 0;/r.19pc/:
   (-10,0 )*+{ \bullet}="1";
(40, 0)*+{ \bullet }="2";
( 15,15)*+{\bullet }="3";
(55,15)*+{ \bullet }="4";
(-10,40 )*+{ \bullet  }="5";
(  15,55)*+{ \bullet }="7";
(55,55)*+{\bullet }="8";
 {\ar@{->}_{\scriptscriptstyle\phi_{ g_4,g_3} }  "1";"2" };
{\ar@{<-}_{\scriptscriptstyle\phi_{  g_3,g_2 } } "3";"1" };
{\ar@{<-}_{\scriptscriptstyle \phi_{g_2,g_1}} "5";"1" };
{\ar@{<-}^{ \scriptscriptstyle\phi_{g_4 g_3, g_2}} "4";"2" };
 {\ar@{<-}^{\scriptscriptstyle\phi_{a } } "7";"3" };
{\ar@{->}^{\scriptscriptstyle\phi_{  g_4, g_3 g_2} } "3";"4" };
{\ar@{<-}^{\scriptscriptstyle\phi_{g_4g_3g_2, g_1} } "8";"4" };
 {\ar@{<-}_{\scriptscriptstyle\phi_{g_3, g_2g_1} } "7";"5" };
 {\ar@{->}^{\scriptscriptstyle\phi_{g_4, g_3g_2 g_1}  } "7";"8" };
 {\ar@{=>}_{ A_1  } (-5,38)*{};(13,20)*{}} ;
 {\ar@{=>}^{    A_2  } (20,50)*{};(50,20)*{}} ;
  {\ar@{=>}_{   A_3  } (15,12)*{};(35,3)*{}} ;
 \endxy  \xy0;/r.22pc/:
  (0, 0)*+{ }="1";
(15,0)*+{ }="2";
{\ar@{=}^{ }   "1"+(0,20);"2"+(0,20)};
   \endxy
  \xy 0;/r.19pc/:
   (-10,0 )*+{\bullet }="1";
(40, 0)*+{ \bullet }="2";
( 40,40)*+{\bullet }="3";
(60,15)*+{\bullet  }="4";
(-10,40 )*+{ \bullet  }="5";
(  15,55)*+{ \bullet }="7";
(60,55)*+{\bullet }="8";
 {\ar@{->}_{\scriptscriptstyle\phi_{ g_4,g_3} }  "1";"2" };
{\ar@{<-}_{\scriptscriptstyle\phi_{  g_2, g_1 } } "3";"2" };
{\ar@{<-}_{\scriptscriptstyle \phi_{g_2,g_1}} "5";"1" };
{\ar@{<-}^{\scriptscriptstyle \phi_{ g_4 g_3, g_2}} "4";"2" };
 {\ar@{<-}^{\scriptscriptstyle\phi_{b} } "8";"3" };
{\ar@{->}_{\scriptscriptstyle\phi_{   g_4,g_3} } "5";"3" };
{\ar@{<-}^{\scriptscriptstyle\phi_{ g_4 g_3 g_2 ,g_1} } "8";"4" };
 {\ar@{<-}_{\scriptscriptstyle\phi_{g_3, g_2g_1} } "7";"5" };
 {\ar@{->}^{\scriptscriptstyle\phi_{g_4, g_3 g_2 g_1}  } "7";"8" };
 {\ar@{=}|-{   } (0,20)*{};(20,20)*{}} ;
 {\ar@{=>}_{   A_2' } (43,35)*{};(58,18)*{}} ;
  {\ar@{=>}_{  A_1' } (15,52)*{};(35,43)*{}} ; \endxy
\end{equation}where $\phi_a :=\phi_{g_3g_2 , g_1}$, $\phi_b:=\phi_{g_4 g_3, g_2g_1}$.
Here $\bullet$'s above represent  $1$-isomorphisms in $\mathcal{C}$. The $2$-isomorphisms in (\ref{eq:s-A1}), (\ref{eq:t-A2}), (\ref{eq:t-A3}), (\ref{eq:t-A1'}) and (\ref{eq:s-A2'}) are represented by
 $1$-isomorphisms in  (\ref{eq:Omega}). Now the $3$-cocycle condition (\ref{eq:Omega}) can be viewed as the commutativity of the $2$-isomorphisms in the boundary of the following   cube in $\mathcal{C}^{+}$:
\begin{equation}\label{eq:Omega-cube}
   \xy0;/r.22pc/:
   (-10,0 )*+{ \bullet}="1";
(40, 0)*+{ \bullet }="2";
( 15,15)*+{\bullet }="3";( 40,40)*+{\bullet }="13";
(65,15)*+{ \bullet }="4";
(-10,40 )*+{ \bullet  }="5";
(  15,55)*+{ \bullet }="7";
(65,55)*+{\bullet }="8";
 {\ar@{->}  "1";"2"_{\scriptscriptstyle\phi_{   g_4,g_3}}  };
{\ar@{<--}_{\scriptscriptstyle\phi_{  g_3,g_2 } } "3";"1" };
{\ar@{<-} "5";"1"_{\scriptscriptstyle \phi_{g_2,g_1}} };
{\ar@{<-} "4";"2"^{ \scriptscriptstyle\phi_{g_4 g_3, g_2}} };
 {\ar@{<--} "7";"3"_{\scriptscriptstyle \phi_{a } } };
{\ar@{-->}|-{  } "3";"4" };
{\ar@{<-} "8";"4"^{\scriptscriptstyle\phi_{g_4g_3g_2, g_1} } };
 {\ar@{<-}_{\scriptscriptstyle\phi_{g_3, g_2g_1} } "7";"5" };
 {\ar@{->} "7";"8"^{\scriptscriptstyle\phi_{g_4, g_3g_2 g_1}  } };
 {\ar@{==>}_{ A_1  } (-5,38)*{};(13,25)*{}} ;
 {\ar@{==>}_{    A_2  } (23,35)*{};(37,25)*{}} ;
  {\ar@{==>}_{   A_3  } (15,12)*{};(35,3)*{}} ;
  {\ar@{<-}^{\scriptscriptstyle\phi_{  g_2, g_1 } } "13";"2" };
 {\ar@{<-}^{\scriptscriptstyle\phi_{b } } "8";"13" };
{\ar@{->}|-{ } "5";"13" };{\ar@{=}|-{   } (5,20)*{};(29,20)*{}} ;
 {\ar@{=>}^{   A_2' } (45,35)*{};(62,18)*{}} ;
  {\ar@{=>}^{  A_1' } (17,52)*{};(35,43)*{}} ;
 \endxy
\end{equation}
\begin{rem} (1) In the upper boundaries of diagrams in (\ref{eq:Omega}), the number  of group elements in the second subscripts of $\phi_{*,*}$'s is
increasing: $ g_1$, $  g_2g_1$
  $g_3 g_2g_1$,  while in the lower boundaries  it is the number  of group elements in the first subscripts of $\phi_{*,*}$'s which are increasing: $ g_4$, $  g_4g_3$
  $g_4 g_3g_2$.

  (2) (\ref{eq:Omega}) or (\ref{eq:Omega-cube}) is similar to the pentagon condition of bicategories, but here we actually have more complicated whiskering (cf. (\ref{eq:Omega0})).
\end{rem}
Given a strict  $2$-category $\mathcal{V}$, there exists an {\it associated $3$-category $\mathcal{V}^*$} for which  $\mathcal{V}^*_0$ consists of   one object
$\mathcal{V}$,
$\mathcal{V}^*_1$ consists of all functors from $\mathcal{V}$ to $\mathcal{V}$, $\mathcal{V}^*_2$  consists of   all pseudonatural transformations
and $\mathcal{V}^*_3$ consists of  all modifications. This  is a $3$-category. Because   only $3$-representations  of a group in a  strict $3$-category are developed, we have to  consider a strict $3$-subcategory $\mathcal{W}$ of  $\mathcal{V}^*$ for a strict  $2$-categories $\mathcal{V}$.
We call  a $3$-representation of $G$ in such a  strict  $3$-subcategory $\mathcal{W} $ a  {\it strict $2$-categorical action of $G$ on $\mathcal{V}$}.
In particular,
we have  an endofunctor $\rho_g :\mathcal{V}\rightarrow \mathcal{V}$ for each $g \in G$,
a  pseudonatural transformation  $
     \phi_{h ,g }:\rho_h\#_0 \rho_g \Longrightarrow \rho_{h g}
$ for each $h, g\in G$, and
 a  modification $
     \Phi_{ g_3,g_2,g_1}
$ (the associator in (\ref{eq:3-cocycle}))
  for each $g_3,g_2,g_1\in G$. Here $\rho_h\#_0 \rho_g$ is the composition of functors:
   $$
    \rho_h\#_0 \rho_g(w):=\rho_h (\rho_g(w))$$
  for $w\in \mathcal{V}$. By the definition of $3$-representations,  the endofunctor  $\rho_g$, the pseudonatural transformation  $
     \phi_{h ,g }
$ and the modification  $
     \Phi_{ g_3,g_2,g_1}$ must all be invertible in $\mathcal{W} \subset\mathcal{V}^*$.

For example,
for the $2$-category $\mathcal{V}$ used  in the
$1$-dimensional $3$-representation  in Subsection \ref{sub:1-dim}, its $\mathcal{V}^*$ is a strict $3$-category. For the general action  of $G$ on a $2$-category $\mathcal{V}$, we need to develop $3$-representation    of a group  in  a  $\mathbf{Gray}$-category, since the semi-strictification of a $3$-category is a $\mathbf{Gray}$-category.

When a $2$-category $\mathcal{V}$ is viewed as a   $3$-category with only identity $3$-arrow, a $3$-representation of $G$ in $\mathcal{V}$ is   a {\it $2$-representation} if the
the associator $3$-isomorphism in
(\ref{eq:3-cocycle}) is the identity, so that the $3$-cocycle condition (\ref{eq:4-cocycle}) holds trivially. This coincides with the definition of the
$2$-representation in the strict sense in section 2.2 of \cite{GK}.  And for a category $\mathcal{V}$,  a $2$-representation of $G$ in the   $2$-category $\mathcal{V}^*$
is a  {\it categorical action of $G$ on $\mathcal{V}$}.

 \section{The $2$-categorical traces of   $3$-representations}
\subsection{The $2$-categorical trace of a $1$-endomorphism}
 Let $\mathcal{C}$ be a $3$-category, $x\in \mathcal{C}$ and $A:x\rightarrow x$ be a $1$-endomorphism. Then  $A$ is an object of the $2$-category
 ${\rm Hom}_\mathcal{C}  (x,x)$. The {\it $2$-categorical trace}
 of $A$ is defined as
 $$
    \mathbb{T}r_2 (A)=  {\rm Hom}_\mathcal{C} (1_x,A),
 $$
 which is a category. This  is a subcategory of $\mathcal{C}^{ ++}$.

 Let $A:x\rightarrow x$ be a $1$-endomorphism for $x\in \mathcal{C}_0$, and let the $1$-arrow $C:y\rightarrow x$
 be a quasi-inverse to a  $1$-arrow $B:x\rightarrow y$. Then for any $2$-arrow  $\chi:1_x\Longrightarrow A$ in $\mathbb{T}r_2 (A)_0$, the composition
 $$
     1_y\xrightarrow{u} C\#_0 B \xy0;/r.22pc/:
  (0,0)*+{ }="1";
(10,0)*+{ }="2";
{\ar@{=}^{ }   "1"+(0,0);"2"+(0,0)};
   \endxy C\#_0 1_x\#_0 B\xrightarrow{C\#_0 \chi\#_0 B} C\#_0 A\#_0 B
 $$
defines a functor
 $$ \begin{array}{rl}\Psi(C,B,u):\mathbb{T}r_2 (A)_0&\longrightarrow\mathbb{T}r_2 (C\#_0 A\#_0 B)_0,\\
 (\chi:1_x\Longrightarrow A)&\longmapsto u\#_1 [ C\#_0 \chi\#_0 B],
     \end{array}
 $$
 corresponding to the diagram
 $$
      \xy 0;/r.22pc/:
 (-15,0)*+{ {x}}="1";
(15,0)*+{ {x} }="2";
(-30,0)*+{ {y}}="3";
(30,0)*+{ {y} }="4";
{\ar@{->}^{C} "3"; "1"};
{\ar@{->}^{B} "2"; "4"};
{\ar@/^1.33pc/"1";"2"|-{1_x} };
{\ar@/_1.33pc/_{A} "1";"2"};
 {\ar@{=>}^{\chi} (0,3)*{}; (0,-4)*{}};
 {\ar@{=>}^{u} (0,13)*{}; (0,7)*{}};
 {\ar@/^3.33pc/^{1_y} "3";"4"},
\endxy;
 $$
 and for any $3$-arrow  $\gamma:\chi\xy0;/r.22pc/:
  (0,0)*+{ }="1";
(10,0)*+{ }="2";
{\ar@3{->}^{ }  "1" ;"2" };
   \endxy \chi'$ in $\mathbb{T}r_2 (A)_1$, we have
  $$ \begin{array}{rl}\mathbb{T}r_2 (A)_1&\longrightarrow\mathbb{T}r_2 ( C\#_0 A\#_0 B)_1,\\
 \gamma&\longmapsto u\#_1[ C\#_0 \gamma\#_0 B],
     \end{array}\qquad
       \xy 0;/r.22pc/:
 (-15,0)*+{ {x}}="1";
(15,0)*+{ {x} }="2";
(-30,0)*+{ {y}}="3";
(30,0)*+{ {y} }="4";
{\ar@{->}^{C} "3"; "1"};
{\ar@{->}^{B} "2"; "4"};
{\ar@/^1.33pc/"1";"2"|-{1_x} };
{\ar@/_1.33pc/_{A} "1";"2"};
 {\ar@{=>}^{u} (0,13)*{}; (0,7)*{}};
 {\ar@/^3.33pc/^{1_y} "3";"4"};(0,4)*+{}="A";
(0,-4)*+{}="B";
{\ar@{=>}@/_.5pc/_\chi "A"+(-1.33,0) ; "B"+(-.66,-.55)};
{\ar@{=}@/_.5pc/ "A"+(-1.33,0) ; "B"+(-1.33,0)};
{\ar@{=>}@/^.5pc/^{\chi'} "A"+(1.33,0) ; "B"+(.66,-.55)};
{\ar@{=}@/^.5pc/ "A"+(1.33,0) ; "B"+(1.33,0)};
{\ar@3{->} (-2,0)*{}; (3,0)*{}};
\endxy.
 $$
   \begin{prop}\label{prop:functor} $\Psi(C,B,u):\mathbb{T}r_2 (A) \longrightarrow\mathbb{T}r_2 (C\#_0 A\#_0 B)$ is a functor.
 \end{prop}
 \begin{proof} For    $2$-arrows $ \chi, \chi',\widetilde{\chi}:1_x\Longrightarrow A$  and   $3$-arrows   $\gamma:\chi\xy0;/r.22pc/:
  (0,0)*+{ }="1";
(10,0)*+{ }="2";
{\ar@3{->}^{ }  "1" ;"2" };
   \endxy \chi'$, $\widetilde{\gamma}:\chi'\xy0;/r.22pc/:
  (0,0)*+{ }="1";
(10,0)*+{ }="2";
{\ar@3{->}^{ }  "1" ;"2" };
   \endxy  \widetilde{\chi}$, we have the composition ${\gamma}\#_2 \widetilde\gamma:\chi\xy0;/r.22pc/:
  (0,0)*+{ }="1";
(10,0)*+{ }="2";
{\ar@3{->}^{ }  "1" ;"2" };
   \endxy  \widetilde{\chi}$.  Then by using repeatedly the compatibility condition (\ref{eq:asso2})
 for
 compositions, we find
 $$\begin{array}{ll}
     \Psi(C,B,u)({\gamma})\#_2   \Psi(C,B,u)( \widetilde{\gamma})&=\{ u \#_1[C\#_0{\gamma}\#_0 B ]\}\#_2\{
u   \#_1  [ C\#_0\widetilde \gamma\#_0 B]\}\\&
      =u\#_1 [ C\#_0 ({\gamma}\#_2\widetilde\gamma)\#_0 B]\\&=\Psi(C,B,u)({\gamma}\#_2\widetilde\gamma).
  \end{array}$$
  Thus $\Psi(C,B,u) $ is a functor.
   \end{proof}

\subsection{The $2$-categorical trace $ {\rm \mathbb{T}r}_2\rho_f $}
Let $\rho$ be a $3$-representation of $G$ in a $3$-category $\mathcal{C}$. Fix an object $x$ in $\mathcal{C}$ that $G$ acts on. For $f\in G$, let $\rho_f:x \rightarrow x $ be a  $1$-isomorphism  in $\mathcal{C}$. Recall that ${\rm
\mathbb{T}r}_2\rho_f$   is a category whose objects are $2$-arrows with source $1_x$ and target $ \rho_f$  and the morphisms are  $3$-arrows  between them. In the sequel,  we will use the notation
$$
g^*:=   g^{-1}
$$
  for simplicity.
For any  $g  $ commuting with $f$ and a $2$-arrow $\chi: 1_{x}\xy0;/r.22pc/:
  (0,0)*+{ }="1";
(10,0)*+{ }="2";
{\ar@{=>}^{ }  "1" ;"2" };
   \endxy
\rho_f$ in $({\rm \mathbb{T}r}_2\rho_f)_0$, we define a $2$-arrow $\psi_g(\chi): 1_{x}\xy0;/r.22pc/:
  (0,0)*+{ }="1";
(10,0)*+{ }="2";
{\ar@{=>}^{ }  "1" ;"2" };
   \endxy
\rho_f$ by
\begin{equation}\label{eq:psi-g}
    \psi_g( \chi):= u_g \#_1\left[\rho_g\#_0\chi\#_0\rho_{g^*}\right ]\#_1\left[\phi_{g,f}\#_0\rho_{g^*}\right ]\#_1 \phi_{gf,g^*}.
\end{equation}
This is
given by the composition of $2$-arrows in the following   diagram
\begin{equation}\label{eq:psi-g-graph}
      \xy0;/r.22pc/:
 (-15,0)*+{ {x}}="1";
(15,0)*+{ {x} }="2";
(-30,0)*+{ x}="3";
(30,0)*+{ x}="4";(20,-5)*+{\scriptscriptstyle \phi_{gf,g^*}  }="5";
{\ar@{->}^{\rho_g} "3"; "1"};
{\ar@{->}^{\rho_{g^*}} "2"; "4"};
 {\ar@/^1.93pc/ "1";"2"|-{1_x}};
{\ar@/_1.93pc/ "1";"2"|-{\rho_f}}; {\ar@/^3.93pc/^{1_x} "3";"4"};
 {\ar@{=>}^{\chi} (0,5)*{}; (0,-5)*{}};
 {\ar@{=>}^{u_g} (0,16)*{}; (0,10)*{}};
 {\ar@/_2.93pc/ "3";"2"^{\rho_{gf}}};
 {\ar@/_3.93pc/ "3";"4"_{\rho_{ {gfg^*}}=\rho_f}};
    {\ar@{=>}_{\scriptscriptstyle \phi_{g,f}} (-15,-2)*{}; (-15,-10)*{} };
    {\ar@{=>}^{} ( 15,-2)*{}; ( 15,-12)*{} };
     \endxy,
 \end{equation}where   $ u_g=\phi_{g,g^*}^{-1}:1_x\Longrightarrow  \rho_{g }\rho_{g^*}$.
      For a  $3$-arrow  $\Theta:\chi\xy0;/r.22pc/:
  (0,0)*+{ }="1";
(6,0)*+{ }="2";
{\ar@3{->}    "1"+(0, 0);"2"+(0, 0)};
   \endxy \chi'$, we define $\psi_g(\Theta)$ as a  $3$-arrow  whiskered by corresponding  $2$-isomorphisms in (\ref{eq:psi-g-graph}). In other words,
\begin{equation}\label{eq:psi-g-1}
   \psi_g(\Theta)= u_g \#_1\left[\rho_g\#_0\Theta\#_0\rho_{g^*}\right ]\#_1\left[(\phi_{g,f}\#_0\rho_{g^*})\#_1 \phi_{gf,g^*}\right
   ]:\psi_g(\chi)\xy0;/r.22pc/:
  (0,0)*+{ }="1";
(6,0)*+{ }="2";
{\ar@3{->}    "1"+(0, 0);"2"+(0, 0)};
   \endxy \psi_g(\chi')
\end{equation}is a  $3$-arrow corresponding  to the diagram
$$
      \xy0;/r.22pc/:
 (-15,0)*+{ {x}}="1";
(15,0)*+{ {x} }="2";
(-30,0)*+{ x}="3";
(30,0)*+{ x}="4";(20,-5)*+{\scriptscriptstyle \phi_{gf,g^*}  }="5";
{\ar@{->}^{\rho_g} "3"; "1"};
{\ar@{->}^{\rho_{g^*}} "2"; "4"};
 {\ar@/^1.93pc/ "1";"2"|-{1_x}};
{\ar@/_1.93pc/ "1";"2"|-{\rho_f}}; {\ar@/^3.963pc/^{1_x} "3";"4"};
 {\ar@{=>}^{u_g} (0,16)*{}; (0,10)*{}};
 {\ar@/_2.93pc/ "3";"2"^{\rho_{gf}}};
 {\ar@/_3.93pc/ "3";"4"_{\rho_{ {gfg^*}}=\rho_f}};
    {\ar@{=>}_{\scriptscriptstyle \phi_{g,f}} (-15,-2)*{}; (-15,-10)*{} };
    {\ar@{=>}^{ } ( 15,-2)*{}; ( 15,-12)*{} };(0,5)*+{}="A";
(0,-5)*+{}="B";
{\ar@{=>}@/_.5pc/_\chi "A"+(-1.33,0) ; "B"+(-.66,-.55)};
{\ar@{=}@/_.5pc/ "A"+(-1.33,0) ; "B"+(-1.33,0)};
{\ar@{=>}@/^.5pc/^{\chi'} "A"+(1.33,0) ; "B"+(.66,-.55)};
{\ar@{=}@/^.5pc/ "A"+(1.33,0) ; "B"+(1.33,0)};
{\ar@3{->} (-2,0)*{}; (3,0)*{}};
     \endxy
 $$
 in the $3$-category $ \mathcal{C}  $.
Then $\psi_g$ defines an endofunctor $  \psi_g$ on ${\rm \mathbb{T}r}_2\rho_f  $ by the proof of  Proposition
\ref{prop:functor}.
Namely, we have
 $$
    \psi_g(\Theta\widetilde{\#}_0\Theta')=\psi_g(\Theta)\widetilde{\#}_0\psi_g(\Theta')  $$ for any  $3$-arrow $\Theta':\chi'\xy0;/r.22pc/:
  (0,0)*+{ }="1";
(10,0)*+{ }="2";
{\ar@3{->}^{ }  "1" ;"2" };
   \endxy\chi''$, where $\widetilde{\#}_0$ is the composition in the category $\mathcal{C}^{ ++}$ ($\widetilde{\#}_0= {\#}_{ 2} $).

In Section 3.4, we will construction
a natural isomorphism  $\Gamma_{h,g}:  \psi_h \circ\psi_g\longrightarrow\psi_{hg} $ for given $g,h\in C_G(f)$. It gives us  natural isomorphisms $\Gamma_{g^*,g}:  \psi_{g^*} \circ\psi_g\longrightarrow\psi_{1} $ and $\Gamma_{g,g^*}:\psi_g\circ\psi_{g^*}\longrightarrow\psi_{1} $. Thus $\psi_g$ for each $g\in C_G(f)$ is an equivalence of  the category ${\rm \mathbb{T}r}_2\rho_f  $.
\subsection{The adjoint $2$-isomorphisms}\label{adjoint}
For a $2$-isomorphism $
      \xy 0;/r.22pc/:
 (10,0)*+{ y}="1";
(-10,0)*+{{x}}="5";
 {\ar@/^1.33pc/^{\chi_1} "5";"1"};
  {\ar@/_1.33pc/_{\chi_2} "5";"1"}; {\ar@{=>}^{\phi} (0,5)*{}; (0,-5)*{}};
  \endxy
$ in a $2$-category $\mathcal{V}$,
    we define the {\it adjoint $2$-isomorphism} $\phi^\dag$ to be $\xy 0;/r.22pc/:
 (10,0)*+{{x}}="1";
(-10,0)*+{y}="5";
 {\ar@/^1.33pc/^{\chi_1^{-1}}"5";"1"};
  {\ar@/_1.33pc/_{\chi_2^{ -1}}"5";"1"}; {\ar@{=>}^{\phi^\dag} (0,5)*{}; (0,-5)*{}};
  \endxy $ by the composition of  arrows
   \begin{equation}\label{eq:adjoint}
       \xy 0;/r.22pc/:(30,0)*+{ {x}}="11";
(-30,0)*+{ y}="15";
 (10,0)*+{y}="1";
(-10,0)*+{ {x}}="5";
 {\ar@/^1.33pc/^{\chi_2} "5";"1"};
  {\ar@/_1.33pc/_{\chi_1} "5";"1"}; {\ar@{=>}^{\phi^{-1}} (0,5)*{}; (0,-5)*{}};
  {\ar@{->}^{ \chi_1^{-1} } "15";"5" };
  {\ar@{->}^{ \chi_2^{ -1} } "1";"11" };
  \endxy .
    \end{equation}This is a $2$-isomorphism
with  inverted $1$-source and $1$-target. This operation will be used later.
    See also section 2 of \cite{Ma} for the definition of    similar adjoint $2$-arrows, but   $\phi^{-1}$ in (\ref{eq:adjoint}) is replaced there by $\phi$.

 \begin{prop}\label{prop:adjoint} (1)   For any pair of  $2$-isomorphisms $
      \xy 0;/r.22pc/:
 (8,0)*+{ y}="1";
(-8,0)*+{{x}}="5";
 {\ar@/^1.33pc/^{\chi_1} "5";"1"};
  {\ar@/_1.33pc/_{\chi_2} "5";"1"}; {\ar@{=>}^{\phi} (0,5)*{}; (0,-5)*{}};
  \endxy
$ and  $
      \xy 0;/r.22pc/:
 (8,0)*+{ y}="1";
(-8,0)*+{x}="5";
 {\ar@/^1.33pc/^{\chi_2} "5";"1"};
  {\ar@/_1.33pc/_{\chi_3} "5";"1"}; {\ar@{=>}^{\psi} (0,5)*{}; (0,-5)*{}};
  \endxy
$, we have $(\phi \#_1\psi)^\dag=\phi^\dag\#_1\psi^\dag $.

(2)  For any $1$-isomorphism $\chi_0:z\longrightarrow x$,  we have  $(\chi_0\#_0\phi)^\dag=\phi^\dag\#_0\chi_0^{-1}$; and     for $1$-isomorphism $\widetilde{\chi}_0 :y\longrightarrow z$, we have   $(\phi\#_0\widetilde{\chi}_0 )^\dag=\widetilde{\chi}_0^{ -1}\#_0\phi^\dag$.

(3) For a $2$-isomorphism $
      \xy 0;/r.22pc/:
 (8,0)*+{ z}="1";
(-8,0)*+{y}="5";
 {\ar@/^1.33pc/^{\widetilde{\chi}_1 } "5";"1"};
  {\ar@/_1.33pc/_{\widetilde{\chi}_2 } "5";"1"}; {\ar@{=>}^{\widetilde{\phi} } (0,5)*{}; (0,-5)*{}};
  \endxy
$, we have   $(\phi\#_0\widetilde{\phi} )^\dag  =\widetilde{\phi}^{ \dag}  \#_0\phi^\dag$, i.e.,   $\xy 0;/r.22pc/:
 (8,0)*+{ y}="1";
(-8,0)*+{z}="5";(24,0)*+{ x}="6";
 {\ar@/^1.33pc/^{\widetilde{\chi}_1^{ -1}} "5";"1"};
  {\ar@/_1.33pc/_{\widetilde{\chi}_2^{ -1}} "5";"1"}; {\ar@{=>}^{\widetilde{\phi}^{ \dag}} (0,5)*{};
  (0,-5)*{}};{\ar@/^1.33pc/^{\chi_1^{-1}} "1";"6"};
  {\ar@/_1.33pc/_{\chi_2^{-1}} "1";"6"}; {\ar@{=>}^{\phi^\dag} (16,5)*{}; (16,-5)*{}};
  \endxy$.
\end{prop}
  \begin{proof} (1) $\phi^\dag\#_1\psi^\dag=(\phi \#_1\psi)^\dag  $ follows from
$$
   \xy 0;/r.26pc/:
(40,0)*+{  y}="3";
(56,0)*+{  x}="4";
  (24,0)*+{    x }="11";
(-24,0)*+{   y}="15";
 (8,0)*+{  y}="1";
(-8,0)*+{   {x}}="5";
{\ar@/^1.33pc/^{\chi_2} "5";"1"};
  {\ar@/_1.33pc/_{\chi_1} "5";"1"}; {\ar@{=>}^{ \scriptscriptstyle \phi^{-1}} (0,4)*{}; (0,-5)*{}};
  {\ar@{->}^{ \chi_1^{-1} } "15";"5" };
  {\ar@{->}^{ \chi_2^{ -1} } "1";"11" };
  {\ar@/^1.33pc/^{\chi_3} "11";"3"};
  {\ar@/_1.33pc/_{\chi_2} "11";"3"};
   {\ar@{=>}^{\scriptscriptstyle  \psi^{-1}} (32,4)*{}; (32,-5)*{}};
  {\ar@{->}^{ \chi_3^{-1} } "3";"4" };  \endxy
  \xy0;/r.22pc/:
  (0,0)*+{ }="1";
(15,0)*+{ }="2";
{\ar@{=}^{ }   "1"+(0,0);"2"+(0,0)};
   \endxy \xy 0;/r.23pc/:(24,0)*+{  {x}}="11";
(-24,0)*+{  y}="15";
 (8,0)*+{  y}="1";
(-8,0)*+{   {x}}="5";
 {\ar@/^1.33pc/|-{\chi_2} "5";"1"};
  {\ar@/_1.33pc/_{\chi_1} "5";"1"}; {\ar@{=>}^{\scriptscriptstyle \phi^{-1}} (0,4)*{}; (0,-5)*{}};
  {\ar@{->}^{ \chi_1^{-1} } "15";"5" };
  {\ar@{->}^{ \chi_3^{ -1} } "1";"11" };
  {\ar@/^3.33pc/^{\chi_3} "5";"1"};{\ar@{=>}|-{\scriptscriptstyle\psi^{-1}} (0,14)*{}; (0,7 )*{}};
  \endxy
    $$
    by $\xy 0;/r.27pc/: (-40,0)*+{  x}="25";
(-24,0)*+{   y}="15";
 (8,0)*+{  y}="1";
(-8,0)*+{   {x}}="5";
 {\ar@/^1.33pc/^{\chi_3} "5";"1"};
  {\ar@/_1.33pc/_{\chi_2} "5";"1"}; {\ar@{=>}^{ \scriptscriptstyle \psi^{-1}} (0,4)*{}; (0,-4)*{}};
  {\ar@{->}^{ \chi_2^{-1} } "15";"5" };  {\ar@{->}^{ \chi_2 } "25";"15" };
   \endxy \xy0;/r.22pc/:
  (0,0)*+{ }="1";
(8,0)*+{ }="2";
{\ar@{=}^{ }   "1"+(0,0);"2"+(0,0)};
   \endxy \xy 0;/r.27pc/:
 (8,0)*+{  y}="1";
(-8,0)*+{   {x}}="5";
 {\ar@/^1.33pc/^{\chi_3} "5";"1"};
  {\ar@/_1.33pc/_{\chi_2} "5";"1"}; {\ar@{=>}^{ \scriptscriptstyle\psi^{-1}} (0,4)*{}; (0,-4)*{}};
     \endxy $ and the interchange  law
(\ref{eq:interchanging-law}) for horizontal compositions.

(2) follows from the fact that $(\chi_0\#_0\phi)^\dag$ is
     $$
   \xy 0;/r.27pc/:
(40,0)*+{  y}="3";
(56,0)*+{  x}="4";(72,0)*+{z}="14";
  (24,0)*+{    x }="11";
(-24,0)*+{   y}="15";
 (8,0)*+{  z}="1";
(-8,0)*+{   {x}}="5";
{\ar@{->}^{ \chi_0^{-1} } "5";"1"};
   {\ar@{->}^{ \chi_1^{-1} } "15";"5" };
  {\ar@{->}^{ \chi_0 } "1";"11" };
  {\ar@/^1.33pc/^{\chi_2} "11";"3"};
  {\ar@/_1.33pc/_{\chi_1} "11";"3"};
   {\ar@{=>}^{\scriptscriptstyle  \phi^{-1}} (32,4)*{}; (32,-4)*{}};
  {\ar@{->}^{ \chi_2^{-1} } "3";"4" };   {\ar@{->}^{ \chi_0^{-1} } "4";"14" };
  \endxy,
    $$
since $\chi_0^{-1}\#_0\chi_0 $ is equal to the identity $1_x$.

 (3)  Note that $ \phi\#_0\widetilde{\phi} =(\chi_1\#_0\widetilde{\phi }) \#_1(\phi\#_0\widetilde{\chi}_2 ) $ by using the  interchange  law (\ref{eq:interchanging-law}) .    We see that
 $$\begin{array}{ll}
   \left (\phi\#_0\widetilde{\phi}\right )^\dag &=\left(\chi_1\#_0\widetilde{\phi} \right)^\dag\#_1\left(\phi\#_0\widetilde{\chi}_2 \right)^\dag=\left(\widetilde{\phi}^{ \dag}\#_0\chi_1^{ -1}\right)\#_1\left(\widetilde{\chi}_2^{ -1}\#_0\phi^\dag\right)
    =\widetilde{\phi}^{ \dag}  \#_0\phi^\dag
\end{array} $$
by using (1), (2) and  the  interchange  law (\ref{eq:interchanging-law}) again.
     \end{proof}
\subsection{The  categorical action of the centralizer of $f$ on   $ {\rm \mathbb{T}r}_2\rho_f $} To construct a  categorical action of
the centralizer $C_G(f)$ of $f$   on  the category ${\rm \mathbb{T}r}_2\rho_f$,
let us write down the composition law for the functors $\psi_h$ and $  \psi_g$,
 $$
    \psi_h \circ\psi_g:{\rm \mathbb{T}r}_2\rho_f \longrightarrow{\rm \mathbb{T}r}_2\rho_f,$$
where $h ,g \in C_G(f)$. For a fixed $\chi\in ({\rm \mathbb{T}r}_2\rho_f)_0$ and $\Theta\in ({\rm \mathbb{T}r}_2\rho_f)_1$,   by using the definition (\ref{eq:psi-g})-(\ref{eq:psi-g-1}) of $\psi_*$ twice, we see that
$\psi_h \circ\psi_g(\chi)=\psi_h (\psi_g(\chi))$ is the composition of  $2$-arrows  in $\mathcal{C}$ in the following diagram:
\begin{equation}\label{eq:2-composition}
      \xy0;/r.22pc/:
 (-15,0)*+{ {x}}="1";
(15,0)*+{ {x} }="2";
(-30,0)*+{ x}="3";
(30,0)*+{ x}="4";
(-45,0)*+{ x}="5";
(45,0)*+{ x}="6";
{\ar@{->}^{\rho_g} "3"; "1"};
{\ar@{->}^{\rho_{g^*}} "2"; "4"};
{\ar@{->}^{\rho_h} "5"; "3"};
{\ar@{->}^{\rho_{h^*}} "4"; "6"};
{\ar@/^2.23pc/ "1";"2"_{1_x}};
{\ar@/_2.23pc/^{\rho_f} "1";"2"};
 {\ar@{=>}^{\chi} (0,5)*{}; (0,-5)*{}}; {\ar@{=>}^{u_h} (0,28)*{}; (0,22)*{}};
 {\ar@{=>}^{u_g} (0,16)*{}; (0,10)*{}};
 {\ar@/^4.33pc/ "3";"4"|-{1_x}}; {\ar@/^6.33pc/ "5";"6"^{1_x}};
 {\ar@/_3.33pc/^{\rho_{gf}} "3";"2"};
 {\ar@/_4.33pc/ "3";"4"^{\rho_{f}}};
  {\ar@/_5.33pc/^{\rho_{hf}} "5";"4"};
   {\ar@/_6.33pc/_{\rho_{f}} "5";"6"};
   {\ar@{=>}_{\scriptscriptstyle \phi_{g,f}} (-15,-1.5)*{}; (-15,-11)*{} };
    {\ar@{=>}^{} ( 15,-2)*{}; ( 15,-11.5)*{} };
     {\ar@{=>}_{\scriptscriptstyle \phi_{h,f}} (-30,-2)*{}; (-30,-12.5)*{} };
     {\ar@{=>}^{} ( 30,-2)*{}; ( 30,-13)*{} };
     (36,-5)*+{\scriptscriptstyle \phi_{h f,h^*}}="06";
     (21,-5)*+{\scriptscriptstyle \phi_{gf,g^*}}="05";
\endxy,
 \end{equation}and $\psi_h \circ\psi_g(\Theta)=\psi_h (\psi_g(\Theta))$ is  a $3$-arrow  in $\mathcal{C}$   defined similarly. Recall that we assume $\rho_{g1}=\rho_{g }1_x$ and $\rho_{h1}=\rho_{h }1_x$. The upper half part of (\ref{eq:2-composition}) is the same as the
 lower half with $f$ replaced by $1_x$ and $2$-isomorphisms inverted:
 \begin{equation}\label{eq:upper-half}
      \xy0;/r.22pc/:
 (-15,0)*+{ {x}}="1";
(15,0)*+{ {x} }="2";
(-30,0)*+{ x}="3";
(30,0)*+{ x}="4";
(-45,0)*+{ x}="5";
(45,0)*+{ x}="6";
{\ar@{->}_{\scriptscriptstyle  \rho_g} "3"; "1"};
{\ar@{->}_{\scriptscriptstyle   \rho_{g^*}} "2"; "4"};
{\ar@{->}_{\scriptscriptstyle  \rho_h} "5"; "3"};
{\ar@{->}_{\scriptscriptstyle  \rho_{h^*}} "4"; "6"};
{\ar@/^2.03pc/_{1_x} "1";"2"};
 {\ar@{=>}^{} (30,12.5)*{}; (30,2)*{}}; {\ar@{=>}_{\scriptscriptstyle  \phi_{h,1}^{-1}} (-30,13)*{}; (-30,1.5)*{}};
 {\ar@{=>}^{\scriptscriptstyle  \phi_{g1,g^*}^{-1}} (15,11.5)*{}; (15,1.5)*{}}; {\ar@{=>}_{\scriptscriptstyle  \phi_{g,1}^{-1}} (-15,10)*{}; (-15,1.5)*{}};
 {\ar@/^4.33pc/_{1_x} "3";"4"}; {\ar@/^6.33pc/^{1_x} "5";"6"};
  {\ar@{-->}@/^3.03pc/_{\scriptscriptstyle   \rho_{g1}} "3";"2"};
  {\ar@{-->}@/^5.33pc/_{\scriptscriptstyle  \rho_{h1}} "5";"4"};(36,5)*+{ \scriptscriptstyle \phi_{h1,h^*}^{-1}}="06";
  \endxy,
 \end{equation}namely, we have $u_h=\phi_{h1,h^*}^{-1}\#_1[\phi_{h,1}^{-1}\#_0\rho_{h^*}]$ and similar identity for $ u_g $. Note that $\phi_{h,1}$ and $\phi_{g,1}$ are identities by our assumptions in Remark 2.2 (1).

Now let us write down the
 natural isomorphism
 $$
    \Gamma_{h,g}:  \psi_h \circ \psi_g\longrightarrow\psi_{hg}
 $$
 between functors on the category   ${\rm \mathbb{T}r}_2\rho_f$.
 The lower half of diagram (\ref{eq:2-composition}) is
 \begin{equation}\label{eq:psi-g-h}
      \xy 0;/r.22pc/:
 (-15,0)*+{ {x}}="1";
(15,0)*+{ {x} }="2";
(-30,0)*+{ x}="3";
(30,0)*+{ x}="4";
(-45,0)*+{ x}="5";
(45,0)*+{ x}="6";
{\ar@{->}^{\rho_g} "3"; "1"};
{\ar@{-->}^{\rho_{g^*}} "2"; "4"};
{\ar@{-->}^{\rho_h} "5"; "3"};
{\ar@{->}^{\rho_{h^*}} "4"; "6"};
 {\ar@{->}^{\rho_f} "1";"2"};
 {\ar@{-->}@/_2.33pc/^{\rho_{gf}} "3";"2"};
 {\ar@{-->}@/_3.33pc/^{\rho_{\underline{gg^*}}} "3";"4"};
  {\ar@{-->}@/_4.33pc/^{\rho_{\underline{hg^*}}} "5";"4"};
   {\ar@/_5.33pc/_{\rho_{\underline{hh^*}}} "5";"6"};
  \endxy
 \end{equation} Here and in the following, for simplicity, we will use the notation
 $$
    \rho_{\underline{g_1g_2}}:=\rho_{ {g_1 \ldots g_2}} ,
 $$
i.e., we omit
 the group elements
 between $g_1$ and $g_2$ in the sequence $h,g,f,g^*,h^*$ in diagram (\ref{eq:psi-g-h}).

 Recall that the associator $3$-isomorphism  $\Phi_{g_3,g_2,g_1 }$ in (\ref{eq:3-cocycle})-(\ref{eq:diagonal}) can be drawn   in the  form (\ref{eq:associator}).
By definition, the $3$-isomorphism
  \begin{equation}\label{eq:hat-Lambda-1}
   \hat \Lambda_1= \gamma_1\#_1[ \Phi_{h,gf,g^*}\#_0 \rho_{h^*}]\#_1  \gamma_2, \end{equation}   is the associator $\Phi_{h,gf,g^*}\#_0 \rho_{h^*}$   whiskered by two $2$-isomorphisms
    \begin{equation}\label{eq:gamma-1-2} \begin{split}
 \gamma_1=[\rho_h\#_0\phi_{g,f}\#_0(\rho_{g^*}\rho_{h^*})]:&\qquad     \xy 0;/r.22pc/:
 (-15,0)*+{ {x}}="1";
(15,0)*+{ {x} }="2";
(-30,0)*+{ x}="3";
(30,0)*+{ x}="4";
(-45,0)*+{ x}="5";
(45,0)*+{ x}="6";
{\ar@{->}^{\rho_g} "3"; "1"};
{\ar@{->}^{\rho_{h^*}} "4"; "6"};
 {\ar@{->}^{\rho_f} "1";"2"};
 {\ar@{->}@/_2.33pc/_{\rho_{gf}} "3";"2"};{\ar@{->}^{\rho_{g^*}} "2"; "4"};
{\ar@{->}^{\rho_h} "5"; "3"};
  \endxy,\\
  \gamma_2= \phi_{\underline{hg^*},h^*} :& \qquad    \xy 0;/r.22pc/:
 (30,0)*+{ x}="4";
(-45,0)*+{ x}="5";
(45,0)*+{ x}="6";
{\ar@{->}^{\rho_{h^*}} "4"; "6"};
  {\ar@{->}@/_4.33pc/^{\rho_{\underline{hg^*}}} "5";"4"};
   {\ar@/_5.33pc/_{\rho_{\underline{hh^*}}} "5";"6"};
  \endxy,
       \end{split}  \end{equation}
 from above and below,
    respectively. This replaces the diagonal $\rho_{\underline{gg^*}}$ of the   dotted quadrilateral in
diagram (\ref{eq:psi-g-h}) by the wavy
diagonal $\rho_{\underline{hf} }$ of the   same quadrilateral in the following  diagram:
 \begin{equation}\label{eq:psi-g-h-2}
      \xy 0;/r.22pc/:
 (-15,0)*+{ {x}}="1";
(15,0)*+{ {x} }="2";
(-30,0)*+{ x}="3";
(30,0)*+{ x}="4";
(-45,0)*+{ x}="5";
(45,0)*+{ x}="6";
{\ar@{-->}^{\rho_g} "3"; "1"};
{\ar@{->}^{\rho_{g^*}} "2"; "4"};
{\ar@{-->}^{\rho_h} "5"; "3"};
{\ar@{->}^{\rho_{h^*}} "4"; "6"};
 {\ar@{-->}^{\rho_f} "1";"2"};
 {\ar@{-->}@/_2.33pc/^{\rho_{gf}} "3";"2"};
 {\ar@{~>}@/_3.33pc/^{\rho_{\underline{hf} }} "5";"2"};
  {\ar@/_4.33pc/^{\rho_{\underline{hg^*}}} "5";"4"};
   {\ar@/_5.33pc/_{\rho_{\underline{hh^*}}} "5";"6"};
  \endxy.
 \end{equation}
 $  \hat \Lambda_1 $ in
(\ref{eq:hat-Lambda-1}) is the following   $3$-isomorphism
 \begin{equation}\label{eq:lambda-1-3-dim}
       \xy 0;/r.22pc/:
(0,10)*{};
(0,-10)*{};
(-20,0)*+{x}="1";
(0,0)*+{x}="2";
{\ar@/^1.33pc/ "1";"2" };
{\ar@/_3.73pc/  "1";"2"};{\ar@/^3.73pc/  "1";"2"};
{\ar@{=>}^{\gamma_2} (-10,-8) ; (-10,-16)};{\ar@{<=}_{\gamma_1} (-10, 8) ; (-10, 16)};
{\ar@/_1.33pc/"1";"2"  };
(-10,4)*+{}="A";
(-10,-4)*+{}="B";
{\ar@{=>}@/_.5pc/_\chi "A"+(-1.33,0) ; "B"+(-.66,-.55)};
{\ar@{=}@/_.5pc/ "A"+(-1.33,0) ; "B"+(-1.33,0)};
{\ar@{=>}@/^.5pc/^{\chi'} "A"+(1.33,0) ; "B"+(.66,-.55)};
{\ar@{=}@/^.5pc/ "A"+(1.33,0) ; "B"+(1.33,0)};
{\ar@3{->} (-12,0)*{}; (-7,0)*{}};
\endxy
 \end{equation}
 where $\chi$ is the $2$-arrow corresponding to the   dotted quadrilateral in
diagram (\ref{eq:psi-g-h}), $\chi'$ is the $2$-arrow corresponding to the  same   quadrilateral in
diagram (\ref{eq:psi-g-h-2}) with the diagonal changed, and $2$-arrows $\gamma_1$ and  $\gamma_2$ are given by (\ref{eq:gamma-1-2}).

The $3$-isomorphism  \begin{equation}\label{eq:hat-Lambda-2}
  \hat  \Lambda_2=[\Phi_{h,g,f } \#_0(\rho_{g^*}\rho_{h^*})]\#_1 \{[ \phi_{\underline{hf},g^*}\#_0 \rho_{h^*} ] \#_1
  \phi_{\underline{hg^*},h^*}\},\end{equation} as a whiskered associator (\ref{eq:associator}),
then  changes the diagonal $\rho_{gf}$  of the   dotted-wavy quadrilateral in   diagram (\ref{eq:psi-g-h-2}) to the wavy
diagonal $\rho_{hg }$   of the   same quadrilateral in the following  diagram:

 \begin{equation}\label{eq:psi-g-h-3}
      \xy 0;/r.22pc/:
 (-15,0)*+{ {x}}="1";
(15,0)*+{ {x} }="2";
(-30,0)*+{ x}="3";
(30,0)*+{ x}="4";
(-45,0)*+{ x}="5";
(45,0)*+{ x}="6";
{\ar@{->}^{\rho_g} "3"; "1"};
{\ar@{-->}^{\rho_{g^*}} "2"; "4"};
{\ar@{->}^{\rho_h} "5"; "3"};
{\ar@{-->}^{\rho_{h^*}} "4"; "6"};
 {\ar@{->}^{\rho_f} "1";"2"};
 {\ar@{~>}@/_1.33pc/_{\rho_{hg }} "5";"1"};
 {\ar@{-->}@/_3.33pc/^{\rho_{\underline{hf} }} "5";"2"};
  {\ar@{-->}@/_4.33pc/^{\rho_{\underline{hg^*}}} "5";"4"};
   {\ar@{-->}@/_5.33pc/_{\rho_{\underline{hh^*}}} "5";"6"};
   \endxy.
 \end{equation}
 Similarly, the $3$-isomorphism
 \begin{equation}\label{eq:hat-Lambda-3'}
   \hat\Lambda_3=\{[\phi_{h,g }\#_0(\rho_f\rho_{g^*}\rho_{h^*})]\#_1[ \phi_{hg,f} \#_0( \rho_{g^*}\rho_{h^*})]\}\#_1 \Phi^{-1}_{\underline{hf},
   g^*,h^* } , \end{equation} which is the whiskered  associator $\Phi^{-1}_{\underline{hf},
   g^*,h^* } $,  changes the diagonal $\rho_{hg^* }$  of the  dotted  quadrilateral in
   diagram (\ref{eq:psi-g-h-3}) to the wavy
diagonal $\rho_{ g^* h^*  }$  of the   same quadrilateral  in the following  diagram:
 \begin{equation}\label{eq:psi-g-h-4}
      \xy 0;/r.22pc/:
 (-15,0)*+{ {x}}="1";
(15,0)*+{ {x} }="2";
(-30,0)*+{ x}="3";
(30,0)*+{ x}="4";
(-45,0)*+{ x}="5";
(45,0)*+{ x}="6";
{\ar@{->}^{\rho_g} "3"; "1"};
{\ar@{->}^{\rho_{g^*}} "2"; "4"};
{\ar@{->}^{\rho_h} "5"; "3"};
{\ar@{->}^{\rho_{h^*}} "4"; "6"};
 {\ar@{->}^{\rho_f} "1";"2"};
 {\ar@/_1.33pc/_{\rho_{hg }} "5";"1"};
 {\ar@{->}@/_3.33pc/^{\rho_{\underline{hf} }} "5";"2"};
  {\ar@{~>}@/_1.33pc/_{\rho_{ g^*h^*}} "2";"6"};
   {\ar@{->}@/_5.33pc/_{\rho_{\underline{hh^*}}} "5";"6"};
  \endxy.
 \end{equation}

Recall that the upper half  of diagram (\ref{eq:2-composition}) is the same as the  lower half with $f$ replaced by $1$ and $2$-isomorphisms inverted. So by the corresponding $3$-isomorphisms, denoted by   $ \hat\Lambda_1', \hat \Lambda_2', \hat \Lambda_3'$,
the upper half   of (\ref{eq:upper-half}) is changed to
 \begin{equation}\label{eq:psi-g-h-4-upper}
      \xy 0;/r.22pc/:
 (-15,0)*+{ {x}}="1";
(15,0)*+{ {x} }="2";
(-30,0)*+{ x}="3";
(30,0)*+{ x}="4";
(-45,0)*+{ x}="5";
(45,0)*+{ x}="6";
{\ar@{->}_{\rho_g} "3"; "1"};
{\ar@{->}_{\rho_{g^*}} "2"; "4"};
{\ar@{->}_{\rho_h} "5"; "3"};
{\ar@{->}_{\rho_{h^*}} "4"; "6"};
 {\ar@{->}^{\rho_1} "1";"2"};
 {\ar@/^1.33pc/^{\rho_{hg }} "5";"1"};
 {\ar@{->}@/^3.33pc/^{\rho_{gh }} "5";"2"};
  {\ar@{->}@/^1.33pc/^{\rho_{ g^*h^*}} "2";"6"};
   {\ar@{->}@/^5.33pc/^{\rho_{ 1}} "5";"6"};
  \endxy.
 \end{equation}  Note that
 \begin{equation}\label{eq:cancel}
      \xy 0;/r.22pc/:
 (-15,0)*+{ {x}}="1";
(-30,0)*+{ x}="3";
(-45,0)*+{ x}="5";
{\ar@{->}|-{\rho_g} "3"; "1"};
{\ar@{->}_{\rho_h} "5"; "3"};
 {\ar@/^2.33pc/^{\rho_{hg }} "5";"1"};
  {\ar@/_2.33pc/_{\rho_{hg }} "5";"1"};{\ar@{=>}^{\phi_{h,g}^{-1}} (-30,9); (-30,2)};{\ar@{=>}^{\phi_{h,g} } (-30,-2); (-30,-9)};
  \endxy\xy0;/r.22pc/:
  (0,0)*+{ }="1";
(10,0)*+{ }="2";
{\ar@{=}^{ }   "1"+(0,0);"2"+(0,0)};
   \endxy\xy 0;/r.22pc/:
 (-15,0)*+{ {x}}="1";
(-45,0)*+{ x}="5";
{\ar@{->}^{\rho_{hg}} "5"; "1"};
   \endxy
 \end{equation} and the part involving $\rho_{ g^*}\rho_{h^*}$ is also cancelled. As a result, the composition of (\ref{eq:psi-g-h-4-upper}) and (\ref{eq:psi-g-h-4}), together with $2$-arrow   $\chi:1_x\xy0;/r.22pc/:
  (0,0)*+{ }="1";
(7,0)*+{ }="2";
{\ar@{=>}^{ }   "1"+(0,0);"2"+(0,0)};
   \endxy\rho_f$, gives us the diagram (\ref{eq:psi-g-graph}) with $g$ replaced by $gh$. This is exactly
 $\psi_{gh}(\chi)$. Therefore, the composition of suitable whiskered $3$-isomorphisms   $ \hat \Lambda_1', \hat \Lambda_2', \hat \Lambda_3'$,  $  \hat\Lambda_1, \hat \Lambda_2$ and $
 \hat\Lambda_3$ gives a natural isomorphism  $\Gamma_{h,g}:  \psi_h \circ\psi_g\longrightarrow\psi_{hg} $ such that for $\chi\in (\mathbb{T}r_2\rho_f)_0 $$$
   \Gamma_{h,g}(\chi):\psi_h (\psi_g(\chi))\xy0;/r.22pc/:
  (0,0)*+{ }="1";
(7,0)*+{ }="2";
{\ar@3{->}^{ }   "1"+(0,0);"2"+(0,0)};
   \endxy\psi_{hg}(\chi)
$$
 is a $3$-isomorphism  in $\mathcal{C} $.

  It is not easy to draw $3$-arrows $ \hat\Lambda_j$'s in the $3$-category $\mathcal{C}$.
 But in the
$2$-category $ \mathcal{C}^{  +}$, the first $3$-arrow $\hat \Lambda_1$ in (\ref{eq:hat-Lambda-1}) can be drawn as the $2$-isomorphism corresponding to the following diagram:
 $$
      \xy 0;/r.22pc/:
 (-30,0)*+{\scriptscriptstyle {\rho_h\rho_{g  } \rho_f\rho_{g^*} \rho_{h^*}}}="1";
(20,0)*+{ \scriptscriptstyle {\rho_h \rho_{g f}\rho_{g^*} \rho_{h^*}} }="2";
(60,0)*+{\scriptscriptstyle {\rho_h  \rho_{\underline{gg^*}}  \rho_{h^*}}}="3";
(90,0)*+{ \scriptscriptstyle {\rho_{\underline{hg^*}}  \rho_{h^*}}}="4";
(120,0)*+{\scriptscriptstyle { \rho_{\underline{hh^*}}}}="5";
 (40,-15)*+{\scriptscriptstyle {\rho_{\underline{hf}}  \rho_{g^*} \rho_{h^*}}}="8";
{\ar@{->}^{\scriptscriptstyle \rho_h\#_0\phi_{g,f}\#_0\left(\rho_{g^*}\rho_{h^*}\right)} "1"; "2"};
{\ar@{->}^{\scriptscriptstyle\rho_h\#_0\phi_{g f,g^*}\#_0 \rho_{h^*} } "2"; "3"};
{\ar@{->}^{\scriptscriptstyle\phi_{h, \underline{gg^*}\#_0 \rho_{h^*} }} "3"; "4"};
{\ar@{->}^{\scriptscriptstyle\phi_{\underline{hg^*}, h^*}} "4"; "5"};
{\ar@{->}_{\scriptscriptstyle \phi_{h,gf }\#_0 (\rho_{g^*}\rho_{h^*} )} "2"; "8"};
{\ar@{->}_{\scriptscriptstyle \phi_{\underline{hf},g^* }\#_0 \rho_{h^*} } "8"; "4"};{\ar@{=>}^{ \hat \Lambda_1} (40,-1)*{}; (40,-13)*{}};
  \endxy.
 $$ Here
 the upper  path
 $$
      \xy 0;/r.22pc/:
 (-40,0)*+{  {\rho_h\rho_{g  } \rho_f\rho_{g^*} \rho_{h^*}}}="1";
(20,0)*+{ {\rho_h \rho_{g f}\rho_{g^*} \rho_{h^*}} }="2";
(70,0)*+{\cdots\cdots }="3";
{\ar@{->}^{  \rho_h\#_0\phi_{g,f}\#_0\left(\rho_{g^*}\rho_{h^*}\right)} "1"; "2"};
{\ar@{->}^{ \rho_h\#_0\phi_{g f,g^*}\#_0 \rho_{h^*} } "2"; "3"};
  \endxy
 $$
 corresponds to the $2$-isomorphisms in $\mathcal{C}$ in (\ref{eq:psi-g-h}) (the lower half   of $\psi_h (\psi_g(\chi))$), while
 the lower paths    corresponds to the $2$-isomorphisms in $\mathcal{C}$ in  (\ref{eq:psi-g-h-2}) (the lower half   of $\psi_{h g}(\chi))$). And the $2$-isomorphism $\hat \Lambda_1$ corresponds to the $3$-isomorphism in $\mathcal{C} $ in  (\ref{eq:hat-Lambda-1}). Since ${\rm \mathbb{T}r}_2\rho_f $ is a subcategory of $ \mathcal{C}^{ + +}$,   diagrams in the $2$-category  $ \mathcal{C}^{ +  }$ are sufficient for our purpose.
In the  sequel, to simplify  diagrams,
$$
   \rho_h \cdots \rho_{\underline{g_1g_2}} \cdots \rho_{h^*}\quad {\rm is}
\quad {\rm simply} \quad {\rm written}\quad {\rm as}\quad
 \rho_{\underline{g_1g_2}},
$$
 as an object  in the $2$-category $ \mathcal{C}^{  +}$.
 For simplicity, we also omit the whiskering part of  $1$-isomorphisms $\phi_{*,*}$'s in   diagrams.
The  $3$-isomorphisms  $  \hat\Lambda_1:$ (\ref{eq:psi-g-h})
$\xy0;/r.22pc/:
  (0,0)*+{ }="1";
(7,0)*+{ }="2";
{\ar@3{->}^{ }   "1"+(0,0);"2"+(0,0)};
   \endxy$(\ref{eq:psi-g-h-2}),  $ \hat \Lambda_2:$ (\ref{eq:psi-g-h-2})$\xy0;/r.22pc/:
  (0,0)*+{ }="1";
(7,0)*+{ }="2";
{\ar@3{->}^{ }   "1"+(0,0);"2"+(0,0)};
   \endxy$(\ref{eq:psi-g-h-3}) and $ \hat\Lambda_3:$ (\ref{eq:psi-g-h-3})$\xy0;/r.22pc/:
  (0,0)*+{ }="1";
(7,0)*+{ }="2";
{\ar@3{->}^{ }   "1"+(0,0);"2"+(0,0)};
   \endxy$(\ref{eq:psi-g-h-4}) in   the $3$-category $ \mathcal{C} $   correspond  to $2$-isomorphisms in   the $2$-category $ \mathcal{C}^{  +}$ in the following diagram:
\begin{equation}\label{eq:lower}
  \mathscr{ D}_f:\xy0;/r.22pc/:
  (0,0)*+{ }="1";
(7,0)*+{ }="2";
{\ar@{=}^{ }   "1"+(0,0);"2"+(0,0)};
   \endxy   \xy 0;/r.22pc/:
 (0,0)*+{ \rho_f }="1";
(30,0)*+{ \rho_{g f} }="2";
(60,0)*+{ \rho_{\underline{gg^*}}}="3";
(90,0)*+{ \rho_{\underline{hg^*}}}="4";
(120,0)*+{ \rho_{\underline{hh^*}}}="5";
 (40,-15)*+{\rho_{\underline{hf}}   }="8";
  (20,-35)*+{\rho_{hg  }   }="9";
  (90,-20)*+{\rho_{\underline{hf}}\rho_{g^*  h^* }   }="10";
{\ar@{->}^{\phi_{g,f}} "1"; "2"};
{\ar@{->}^{\phi_{g f,g^*}} "2"; "3"};
{\ar@{->}^{\phi_{h, \underline{gg^*}}} "3"; "4"};
{\ar@{->}^{\phi_{\underline{hg^*}, h^*}} "4"; "5"};
{\ar@{->}^{\scriptscriptstyle \phi_{h,gf }} "2"; "8"};
{\ar@{->}_{\scriptscriptstyle \phi_{\underline{hf},g^* }} "8"; "4"};
{\ar@{->}_{ \phi_{h,g  }} "1"; "9"};
{\ar@{->}_{ \phi_{h g,f  }} "9"; "8"};
{\ar@{->}_{ \phi_{g^*,  h^*}} "8"; "10"};
{\ar@{->}_{ \phi_{\underline{hf}, g^*  h^*  }} "10"; "5"};
{\ar@{=>}_{ \hat \Lambda_1} (55,-0.5)*{}; (55,-10)*{}};
{\ar@{=>}^{ \hat \Lambda_2} (20,-5)*{}; (20,-29)*{}};
{\ar@{=>}_{ \hat \Lambda_3} (90,-3)*{}; (90,-17)*{}};
  \endxy,
 \end{equation} respectively.
Just as for the upper half   of  diagram (\ref{eq:2-composition}),   diagram (\ref{eq:upper-half}) is changed to
  diagram (\ref{eq:psi-g-h-4-upper}). In $\mathcal{C}^+$, this is the composition of $2$-isomorphisms given by the following diagram
 \begin{equation}\label{eq:upper}
     \xy 0;/r.22pc/:
 (0,0)*+{ \rho_1 }="1";
(-30,0)*+{ \rho_{g 1} }="2";
(-60,0)*+{ \rho_{\underline{gg^*}}}="3";
(-90,0)*+{ \rho_{\underline{hg^*}}}="4";
(-120,0)*+{ \rho_{\underline{hh^*}}}="5";
 (-40,-15)*+{\rho_{\underline{h1}}   }="8";
  (-20,-35)*+{\rho_{hg  }   }="9";
  (-90,-20)*+{\rho_{\underline{h1}}\rho_{g^*  h^* }   }="10";
{\ar@{<-}_{\phi_{g,1}^{-1}} "1"; "2"};
{\ar@{<-}_{\phi_{g 1,g^*}^{-1}} "2"; "3"};
{\ar@{<-}_{\phi_{h, \underline{gg^*}}^{-1}} "3"; "4"};
{\ar@{<-}_{\phi_{\underline{hg^*}, h^*}^{-1}} "4"; "5"};
{\ar@{<-}^{ \scriptscriptstyle\phi_{h,g1 }^{-1}} "2"; "8"};
{\ar@{<-}^{\scriptscriptstyle \phi_{\underline{h1},g^* }^{-1}} "8"; "4"};
{\ar@{<-}^{ \phi_{h,g  }^{-1}} "1"; "9"};
{\ar@{<-}^{ \phi_{h g,1  }^{-1}} "9"; "8"};
{\ar@{<-}^{ \phi_{g^*,  h^*}^{-1} } "8"; "10"};
{\ar@{<-}^{ \phi_{\underline{h1}, g^*  h^*  }^{-1}} "10"; "5"};
{\ar@{=>}^{ \hat{\hat \Lambda}_1^\dag} (-55,-0.5)*{}; (-55,-10)*{}};
{\ar@{=>}_{ \hat{\hat \Lambda}_2^\dag} (-20,-5)*{}; (-20,-30)*{}};
{\ar@{=>}^{ \hat{\hat \Lambda}_3^\dag} (-90,-2)*{}; (-90,-18)*{}};
  \endxy \xy0;/r.22pc/:
  (0,0)*+{ }="1";
(7,0)*+{ }="2";
{\ar@{=}^{ }   "1"+(0,0);"2"+(0,0)};
   \endxy : \mathscr{ D}_1^\dag
 \end{equation}where $\hat{\hat \Lambda}_j$ is the $2$-isomorphism previously denoted by $\hat\Lambda_j$ (with $f$ replaced by $1$), and $\hat{\hat \Lambda}_j^\dag$ (previously   denoted by $ {\hat \Lambda}_j'$) is
 the $2$-isomorphism adjoint to $\hat{\hat \Lambda}_j$, defined in \S \ref{adjoint} . Recall that the adjoint $2$-isomorphism is the inverse one with $1$-source and $1$-target  inverted. We apply the adjoint operation  to  diagram (\ref{eq:lower}) to get  diagram (\ref{eq:upper}), the mirror-symmetric diagram of (\ref{eq:lower}),  by   using Proposition \ref{prop:adjoint}. Given $\chi:1_x\xy0;/r.22pc/:
  (0,0)*+{ }="1";
(7,0)*+{ }="2";
{\ar@{=>}^{ }   "1"+(0,0);"2"+(0,0)};
   \endxy\rho_f$, we connect the diagrams  (\ref{eq:upper}) and (\ref{eq:lower}) to get
 $\Gamma_{h,g}(\chi)$ as a  $2$-isomorphism
  in $ \mathcal{C}^{  +}$:
 \begin{equation}\label{eq:D-1-f}
   \mathscr{ D}_1^\dag \xy0;/r.22pc/:
  (0,0)*+{ }="1";
(7,0)*+{ }="2";
{\ar@{->}^{\chi }   "1"+(0,0);"2"+(0,0)};
   \endxy  \mathscr{ D}_f.
 \end{equation}

For objects $\chi,\chi'\in (\mathbb{T}r_2\rho_f)_0$ and a morphism  $\Theta:\chi\rightarrow\chi'$ in $(\mathbb{T}r_2\rho_f)_1$ (i.e., a  $3$-arrow
   in $ \mathcal{C} $),
$\Gamma_{h,g}(\Theta)$ is also a  $3$-arrow. We connect diagrams (\ref{eq:lower}) and (\ref{eq:upper}) to get $\Gamma_{h,g}(\Theta)$ as the following diagram in the $2$-category $  \mathcal{C}^{  +}$:
\begin{equation}\label{eq:connect}
      \xy 0;/r.22pc/:
      (-30,0)*+{ \rho_1 }="11";
(-60,0)*+{    \rho_{g 1}}="12"; (-70,0)*+{  \cdots  \cdots }="112";
 (-70,-15)*+{   \rho_{\underline{h1}} }="18";(-80,-15)*+{ \cdots \cdots  }="118";
  (-50,-35)*+{ \rho_{hg  }   }="19";
{\ar@{->}^{\phi_{g,1}^{-1}}"12"; "11"};
{\ar@{<-}_{ \phi_{h,g1 }^{-1}} "12"; "18"};
{\ar@{<-}^{ \phi_{h,g  }^{-1}} "11"; "19"};
{\ar@{<-}^{ \phi_{h g,1  }^{-1}} "19"; "18"};
{\ar@{=>}_{ \hat {\hat\Lambda}_2^\dag} (-50,-5)*{}; (-50,-25)*{}};
 (0,0)*+{ \rho_f }="1";
(30,0)*+{ \rho_{g f}   }="2";(40,0)*+{ \cdots\cdots  }="02";
 (40,-15)*+{\rho_{\underline{hf}}   }="8"; (50,-15)*+{\cdots \cdots  }="08";
  (20,-35)*+{\rho_{hg  }  }="9";
{\ar@{->}^{\phi_{g,f}} "1"; "2"};
{\ar@{->}^{ \phi_{h,gf }} "2"; "8"};
{\ar@{->}_{ \phi_{h,g  }} "1"; "9"};
{\ar@{->}_{ \phi_{h g,f  }} "9"; "8"};
{\ar@{=>}^{ \hat \Lambda_2} (20,-5)*{}; (20,-25)*{}};
{\ar@{->}^{ \chi} "11"; "1"};
{\ar@{->}@/_4.33pc/ "11";"1"_{\chi'}};
 {\ar@{=>}^{ \Theta} (-15,-3)*{}; (-15,-15)*{}};
  \endxy.
 \end{equation}Note that $\psi_h \circ\psi_g(\chi)$ in (\ref{eq:2-composition}) is the upper boundary of diagram (\ref{eq:connect})  and $\psi_{hg}(\chi')$ is the lower boundary of diagram (\ref{eq:connect}).
 $\Gamma_{h,g}(\chi)$ is  the diagram (\ref{eq:connect}) with the $2$-arrow $\Theta: \chi \Longrightarrow\chi'$ deleted, but $1$-arrow $\chi:\rho_1\longrightarrow \rho_f$ remains, whereas
 $\Gamma_{h,g}(\chi')$ is the  diagram (\ref{eq:connect}) with the $2$-arrow $\Theta: \chi \Longrightarrow\chi'$ deleted, but $1$-arrow $\chi':\rho_1\longrightarrow \rho_f$ remains.
 Applying the interchange  law
(\ref{eq:interchanging-law}) to the diagram (\ref{eq:connect}), we see that $\Gamma_{h,g}$ is a natural isomorphism  in the category ${\rm \mathbb{T}r}_2\rho_f \subset \mathcal{C}^{ + +}$, i.e.   the diagram
$$
    \xy0;/r.22pc/:
(-10,0)*+{   \psi_h \circ\psi_g(\chi)}="1";
(30,0)*+{   \psi_h \circ\psi_g(\chi' )}="2";
(-10,-30)*+{ \psi_{hg}(\chi)}="3";
(30,-30)*+{  \psi_{hg}(\chi')}="4";
{\ar@{->}^{  \psi_h \circ\psi_g(\Theta)} "1";"2" };
{\ar@{->}_{ \psi_{hg}(\Theta)} "3";"4" };
{\ar@{->}_{   \Gamma_{h,g}(\chi )} "1";"3" };
{\ar@{->}^{ \Gamma_{h,g}(\chi')} "2";"4" };
\endxy
$$
is commutative, where the $2$-arrows $ \psi_h \circ\psi_g(\Theta)$ and $ \psi_{hg}(\Theta)$ in $\mathcal{C}^{   +}$ are  $\Theta$ whiskered by $1$-isomorphisms corresponding to the upper and lower boundaries of   diagram  (\ref{eq:connect}), respectively.
\begin{thm} \label{thm:tr-cat}  $\left\{\psi_g,\Gamma_{h,g}\right\}_{g,h\in C_G(f)}$ is a categorical action of    the centralizer $C_G(f)$ on the category ${\rm \mathbb{T}r}_2 \rho_f$.
  \end{thm}
  This theorem will be proved in Section 6 by checking the associative law (\ref{eq:cat-tr-0}) for $\Gamma_{*,*}$,
which is an identity of natural transformations between functors on ${\rm \mathbb{T}r}_2 \rho_f$. Note that
in (\ref{eq:cat-tr-0}), $s_0(\Gamma_{k,h g})=\psi_k \circ\psi_{hg}=t_0(\psi_k \circ\Gamma_{h,g})$. So the composition of natural transformations   used in (\ref{eq:cat-tr-0})
is in the usual order, not in  the natural order which  we assumed in Remark 2.1 (1).

\subsection{$1$-dimensional $3$-representations}\label{sub:1-dim}
We fix a field $k$ of characteristic $0$
containing all roots of unity.
Let $\mathcal{A}$ be a $2$-category with only one object,   one $1$-arrow and $2$-arrows $\mathcal{A}_2\cong k^*$. Fix a  $3$-cocycle $c$ satisfying the condition (\ref{eq:3-cocycle-0}).  Let $\varrho^c$  be the strict $2$-categorical action of $G$ on $\mathcal{A}$ as follows: $\varrho^c_g$ is the identity functor for each $g\in G$;
$$
  \phi_{h,g}: 1_{\mathcal{A}}=\varrho^c_h\varrho^c_g\Longrightarrow\varrho^c_{hg}=1_{\mathcal{A}}
$$
is also the identity pseudonatural isomorphism for any $h, g\in G$; and
\begin{equation}\label{eq:3-cocycle-00}
     \Phi_{ g_3,g_2,g_1}:id=(\varrho^c_{g_3} \#_0\phi_{g_2,g_1} ) \#_1 \phi_{ g_3,g_2 g_1}\xy0;/r.22pc/:
  (0,0)*+{ }="1";
(10,0)*+{ }="2";
{\ar@3{->}^{ }  "1" ;"2" };
   \endxy (\phi_{g_3,g_2}\#_0 \varrho^c_{g_1})\#_1
   \phi_{g_3 g_2,g_1}=id ,
\end{equation}
is a modification determined by the element   $c(g_3,g_2,g_1) \in k^*$ for any $g_3,g_2,g_1\in G$. Then the $3$-cocycle condition (\ref{eq:Omega}) for $\Phi$   is reduced to the equation (\ref{eq:3-cocycle-0}).
The cohomology  classes of  $3$-cocycles are classified by $H^3(G,k^*)$.

For $f\in G$, it is easy to see that ${\rm \mathbb{T}r}_2\varrho^c_f$ is a category with a single object given by the identity pseudonatural isomorphism $\chi_0:1_{\mathcal{A}}\rightarrow \varrho^c_f= 1_{\mathcal{A}}$, and morphisms $({\rm \mathbb{T}r}_2\rho_f)_1\cong k^*$ (an element of  $ k^*$ provides a modification). For $g\in C_G(f)$, $\psi_g:{\rm \mathbb{T}r}_2\varrho^c_f\rightarrow{\rm \mathbb{T}r}_2\varrho^c_f$ is the identity functor by the definitions (\ref{eq:psi-g})-(\ref{eq:psi-g-1}). And
$$
   \Gamma_{h,g}:\chi_0=\psi_h\circ\psi_g(\chi_0)\longrightarrow \psi_{hg
}(\chi_0)=\chi_0
$$
  is a natural isomorphism  given by the element (also denoted by $ {  \Gamma}_{h,g}$ by abuse of notations)
\begin{equation}\label{eq:2-cocycle}
 {  \Gamma}_{h,g}=\frac {c(h,gf,g^*)c(h,g,f)c(\underline{hf},g^*,h^*)^{-1}}{c(h,g ,g^*)c(h,g,1) c(hg ,g^*,h^*)^{-1}}.
\end{equation} This element is obtained by replacing $\Phi_{ g_3,g_2,g_1}$ by the element $c(g_3,g_2,g_1) $  and all other isomorphisms by $1$ in
 $\hat \Lambda_j$'s in (\ref{eq:hat-Lambda-1}) (\ref{eq:hat-Lambda-2}) (\ref{eq:hat-Lambda-3'}),  and using the adjoint operation (\ref{eq:adjoint}), .

\begin{prop}\label{prop:2-cocycle}
   $  \Gamma$ given by (\ref{eq:2-cocycle}) is a $2$-cocycle on the   centralizer $C_G(f)$.
\end{prop}
This proposition will be proved in Section 6.1.

\begin{rem} There exists a  transgression map that maps a $3$-cocycle $c$ on a finite group $G$ to a $2$-cocycle on the inertia groupoid of $G$ \cite{Wi}. It is given by
  $$
  C_{h,g}:=   \frac { c(h, g, f)c(hgfg^{-1}h^{-1}, h, g)}{c(h, gfg^{-1}, g)}
  $$
for given  $f\in G$ (cf. Remark 3.17 in \cite{GU}). Note that for $h,g\in C_G(f)$ we have $C_{h,g}:=    { c(h, g, f)c( f , h, g)}/{c(h,  f , g)}$. So our $2$-cocycle $\Gamma_{h,g}$ in  (\ref{eq:2-cocycle}) is different from the transgressed one. On the other hand, our  $2$-cocycle is only defined for elements which commute with a given element $f$, not on the entire  inertia groupoid of $G$.
\end{rem}

Let $\varrho$ be a categorical action of  a finite group $G$  on a $ {k}$-linear category $\mathcal{W}$. For a  commuting pair of elements $g$ and $f$ in $G$, the {\it $2$-character $\chi_\varrho( f , g )$}  of a categorical action  $ \varrho $ is the joint trace   of functors $  \varrho_f $ and $  \varrho_g $, i.e., the
 trace of the linear transformation induced by the functor $  \varrho_g $   on the categorical trace $ \mathbb{{T}}r  \varrho_f $ (a  $ {k}$-vector space, which we assume to be finite dimensional).

Now let $\rho$ be a strict $2$-categorical action of  a finite group  $G$ on a $ {k}$-linear $2$-category $\mathcal{V}$. Then
$\mathbb{ {T}}r_2  \rho_f $ is a $ {k}$-linear category and $  \psi  $ defines a categorical action of the centralizer  of   $f$ in $G$ on it  by Theorem \ref{thm:tr-cat}.  If $k,g,f \in G$ are pairwise  commutative, we define {\it the $3$-character }  of the $2$-categorical action $\rho$ to be
 \begin{equation}\label{eq:3ch-def}
    \chi_\rho( f,g , k ):= \chi_\psi(g,k) ,
 \end{equation}
the
joint trace of functors $  \psi_g $ and $  \psi_k $ acting on the  $ {k}$-linear category  $ \mathbb{{T}}r_2   \rho_f $, i.e., the trace of  the linear transformation induced by the functor $  \psi_k $    on the  $ {k}$-vector space $ \mathbb{{T}}r     \psi_g $, which we assume to be finite dimensional.

By using the $2$-character formula for $1$-dimensional $2$-representation in proposition 5.1 of \cite{GK}, the   $3$-character of    the $3$-representation $\varrho^c$ for pairwise commutative   $k,g,f \in G$  is  given by
$$
   \chi_{\varrho^c}( f,g , k )= \frac { \Gamma_{k,g}\Gamma_{kg,k^{-1}}}{ \Gamma_{k,1}  \Gamma_{k ,k^{-1}}  },
$$where the expressions $\Gamma_{*,*}$'s are defined by (\ref{eq:2-cocycle}).
It can also be derived from   diagram (\ref{eq:psi-g-graph}).
\section{The induced strict $2$-categorical action on the induced $2$-category}
\subsection{The induced $2$-category}

Let $H\subset  G$  be a subgroup of a finite group $G$ and let $\rho: H\rightarrow \mathcal{V}^*  $ be a strict $2$-categorical action of $H$
on a strict $2$-category $\mathcal{V}$  (cf. definitions at the end of Section 2.4).
${\rm Ind}_H^G(\mathcal{V})$ is a strict $2$-category where

$\bullet$ objects are maps
$
     \vartheta:G\longrightarrow \mathcal{V}_0
$
together with a $1$-isomorphism
 $$
    u_{g,h}: \vartheta(gh)\longrightarrow \rho_{h^{ *}} \vartheta(g)
 $$
for each $g\in G ,h\in H$, satisfying the condition:

(1) $u_{g,1}: \vartheta(g )\longrightarrow \rho_1 \vartheta(g)$ coincides with $\phi_{1 }^{-1}[\vartheta(g)]$;

(2) for each $g\in G ,h_1,h_2\in H$, we have a $2$-isomorphism:
$$\label{eq:Phi-h-h}
    \xy 0;/r.17pc/:
( 15,15)*+{\rho_{(h_1 h_2)^{ *}} \vartheta(g) }="3";
(75,15)*+{\rho_{  h_2^{ *}}\rho_{h_1^{ *}} \vartheta(g) }="4";
(  15,55)*+{\vartheta(gh_1h_2)}="7";
(75,55)*+{\rho_{h^{ *}_2} \vartheta(gh_1) }="8";
{\ar@{->}_{u_{g,h_1 h_2}} "7";"3" };
{\ar@{->} "3";"4"_{ \phi^{-1}_{h_2^{ *},h_1^{ *}  } \vartheta(g)  } };
{\ar@{->}^{\rho_{  h_2^{ *}} u_{g,h_1 }  } "8";"4" };
{\ar@{->} "7";"8"^{ u_{gh_1,h_2}} };
{\ar@{=>}|-{   } (65,48)*{};(25,22)*{}} ;
 \endxy
$$

$\bullet$   $1$-arrows $F  :(\vartheta,u)  \rightarrow(\vartheta',u') $ between objects;

 $\bullet$   $2$-arrows $\gamma: F \rightarrow\widetilde{F}$.

For $k\in G$, the action $ ({\rm ind}_H^G\rho)_k$   on the $2$-category ${\rm Ind}_H^G(\mathcal{V})$  is given by
 $$
\left[ ({\rm ind}_H^G\rho)_k\vartheta\right](g)=\vartheta(k^{-1}g),\qquad\left [  ({\rm ind}_H^G\rho)_ku\right]_{g, h}=u_{k^{-1}g, h},
 $$
for an object $(\vartheta,u)$ in ${\rm Ind}_H^G(\mathcal{V})$.
And  $ ({\rm ind}_H^G\rho)_k(F )$ for a $1$-arrows $F:(\vartheta,u)\rightarrow(\vartheta',u')$ and
$ ({\rm ind}_H^G\rho)_k(\gamma) $ for a $2$-arrow $\gamma: F \rightarrow\widetilde{F}$ can be defined similarly. In general, each commutative diagram in the definition of
the induced category in section 7.1 of \cite{GK} is replaced by a $2$-isomorphism.

We will not write down the definition of the induced $2$-category ${\rm Ind}_H^G(\mathcal{V})$ explicitly. It is a little bit complicated. Since we only work
on finite groups, we can simply identify ${\rm Ind}_H^G(\mathcal{V})$ with $ \mathcal{V}^m$  as a $2$-category, where $m$ is the index of $H$ in $G$.
  For a strict $2$-category $\mathcal{V} $,
 $\mathcal{V}^m$ is also a strict $2$-category with
 $$\begin{array}{rl} {\rm objects}\hskip 3mm
   \mathcal{V}^m_0&:=\{(x_1,\ldots,x_m) ;x_j\in \mathcal{V}_0 \} ,\\  p-{\rm arrows}\hskip 3mm
  \mathcal{V}^m_p&:=\{(\gamma_1,\ldots,\gamma_m) :(x_1,\ldots,x_m)    \rightarrow (y_1,\ldots,y_m) ;\mathcal{V}_p\ni\gamma_j:x_j\rightarrow y_j\},
    \end{array}
     $$ $p=1,2$.
The compositions are defined as
  \begin{equation}\label{eq:V-m}  \begin{split}&
    (\ldots,
\gamma_j,\ldots  )\#_p(\ldots,\gamma_j',\ldots  ):=(\ldots,\gamma_j\#_p \gamma_j',\ldots  ), \end{split} \end{equation}
if $ \gamma_j$ and $\gamma_j'$ are $p$-composable. The axioms for functions $\#_p$ and identities of  $\mathcal{V}^m$ are obviously satisfied.
 The identification ${\rm Ind}_H^G(\mathcal{V})\cong \mathcal{V}^m$  can be obtained by choosing  a system of representatives
 $$
   \mathscr  R=\{r_1,\ldots,r_m\}
 $$
 of left cosets of $H$ in $G$, and associating to each map $\vartheta:G\rightarrow \mathcal{V}_0$
an  object $(\vartheta(r_1),\ldots, \vartheta(r_m))$ in $\mathcal{V}_0^m$.

Let $a_{jk}:\mathcal{V}\rightarrow\mathcal{V}$ be functors such that the $m\times m$ matrix  $F=(a_{jk})$ has only one nonvanishing  entry in each row or column. Then $F$ defines a strict functor from $\mathcal{V}^m $ to $\mathcal{V}^m  $ by
$$ F(\ldots,
\delta_j,\ldots  )=\left(\ldots,
\sum_k a_{jk}(\delta_k),\ldots \right ),
$$
where we write $ \sum_k a_{jk}(\delta_k)$ formally for $\delta_k\in \mathcal{V}_p$, $p=0,1,2$,  since there exists only one term in this sum. But when the $2$-category is $k$-linear, such sums exist. If $\widetilde{F}=(\widetilde{a}_{jk}):\mathcal{V}^m \rightarrow\mathcal{V}^m  $  is another such functor, then we have
$$
 (F\#_0\widetilde{F})_{jk}:= \sum_l a_{jl}\widetilde{a}_{l k} .
$$
Moreover, a  pseudonatural transformation $\phi: F\rightarrow\widetilde{F}$ is given by an $m\times m$ matrix  $\phi=(\phi_{jk})$ with $\phi_{jk}:a_{jk}\rightarrow\widetilde{a}_{jk}$   a pseudonatural transformation between functors on $\mathcal{V}$. Let $\widetilde{\phi}=(\widetilde{\phi}_{jk}): \widetilde{F}\rightarrow\widetilde{\widetilde{F}}$ be another      pseudonatural transformation. Then their composition is $\phi\#_1\widetilde{\phi}:=(\phi_{jk}\#_1\widetilde{\phi}_{jk})$.
\subsection{The induced strict $2$-categorical action  }
Suppose that $\rho$ is a  strict $2$-categorical  action of $H $ on the $2$-category $\mathcal{V}$. For $f\in G$, we define
 $({\rm ind}_H^G\rho)_f $ to be  a  functor from $ \mathcal{V}^m $ to $
\mathcal{V}^m $ as follows.
It is an $m\times m$ matrix whose entries are functors from $ \mathcal{V} $ to $
\mathcal{V} $, i.e., the $(j ,i )$-entry is
\begin{equation}\label{eq:ind-rho}
\left [  ({\rm ind}_H^G\rho)_f\right]_{j i}=\left\{\begin{array}{l}\rho_h, \qquad {\rm if}\quad fr_i=r_jh,\quad{\rm for} \quad  h\in H,\\0,\qquad\quad{\rm
otherwise}.
\end{array}
\right.
\end{equation}
This corresponds to the fact that for a map $\vartheta:G\rightarrow \mathcal{V}_0$,  we have $[({\rm ind}_H^G\rho)_f(\vartheta)](r_j)=\vartheta(f^{-1}r_j)$  and $\vartheta(f^{-1}r_j)=\vartheta( r_ih^{-1})\rightarrow\rho_h  \vartheta(r_i) $.
It is clear that only one entry in each row or column of the $m\times m$ matrix $ ({\rm ind}_H^G\rho)_f $ is nonvanishing.   Then,
\begin{equation}\label{eq:Ind-act}\begin{split}   ({\rm ind}_H^G\rho)_f (\ldots,\delta_j,\ldots)=\left(\ldots,\sum \left (  ({\rm ind}_H^G\rho)_f\right)_{j i}(\delta_{i}) ,\ldots\right),
\end{split}\end{equation}
where $\delta_j\in \mathcal{V}_p$, for $p=0,1,2$.

For simplicity, from now on  the induced object  will be denoted by the hatted one, e.g.  ${\rm ind}_H^G\rho $ is  denoted by $\widehat{\rho}$.  The composition functor  $ \widehat{\rho}_{g_2}$ and $ \widehat{ \rho}_{g_1}$  is defined as
 \begin{equation}\label{eq:hat-rho-1-2}
\left (\widehat{\rho}_{g_2} \widehat{ \rho}_{g_1}\right)_{k i}=\left\{\begin{array}{l}\rho_{h_2} \rho_{h_1} , \qquad\quad {\rm if}\quad
g_1r_i=r_jh_1,g_2r_j=r_kh_2,\,\quad{\rm for}\quad{\rm some} \quad   h_1,h_2\in H,\\0,\qquad\qquad\qquad{\rm  otherwise}.
\end{array}
\right.
\end{equation}Thus  $\widehat{\rho}_{g_2} \widehat{ \rho}_{g_1}$ can be viewed as the product of two $m\times m$ matrices of functors.
On the other hand,
 \begin{equation}\label{eq:hat-rho}
(\widehat{ \rho}_{g_2 g_1}  )_{k i}= \rho_{h_2 h_1}
 \end{equation}
 since $(g_2 g_1)r_i=r_k(h_2h_1)$ by (\ref{eq:hat-rho-1-2}). We define the pseudonatural transformation (as a  $2$-isomorphism in  $ (\mathcal{V}^m)^* $)
 $$
\widehat{ \phi}_{g_2, g_1} : \widehat{ \rho}_{g_2  } \widehat{ \rho}_{  g_1}\xy0;/r.22pc/:
  (0,0)*+{ }="1";
(10,0)*+{ }="2";
{\ar@{=>}   "1"+(0, 0);"2"+(0, 0)};   \endxy \widehat{ \rho}_{g_2 g_1},$$
as   the  $m\times m$ matrix   whose $(k,i)$-entry
 is the $2$-isomorphism
 \begin{equation}\label{eq:hat-phi}
(\widehat{\phi}_{g_2, g_1} )_{k i}= \phi_{h_2, h_1}: \rho_{h_2  }\rho_{  h_1}\longrightarrow  \rho_{h_2 h_1}  ,
 \end{equation} and all other entries vanish.
For $g_1, g_2,g_3\in G$, the   $3$-isomorphism  in  $ (\mathcal{V}^m)^* $
 $$
 \widehat{\Phi}_{g_3, g_2,g_1} :[ \widehat{\rho}_{g_3}\#_0  \widehat{\phi}_{g_2, g_1} ] \#_1  \widehat{\phi}_{g_3, g_2 g_1}
 \xy0;/r.22pc/:
  (0,0)*+{ }="1";
(10,0)*+{ }="2";
{\ar@3{->}   "1"+(0, 0);"2"+(0, 0)};   \endxy
[  \widehat{\phi}_{g_3, g_2  }\#_0 \widehat{\rho}_{g_1}] \#_1 \widehat{{\phi}}_{g_3g_2, g_1}
$$
is  a modification. Write
$$
  g_3r_k=r_lh_3
$$
for some $h_3\in H$. Then we have
\begin{equation}\label{eq:hat-phi-rho}
   [ \widehat{\rho}_{g_3}\#_0  \widehat{\phi}_{g_2, g_1} ]_{li}=\rho_{h_3}\#_0  {\phi}_{h_2, h_1} \quad{\rm and }\quad   [  \widehat{\phi}_{g_3, g_2  }\#_0 \widehat{\rho}_{g_1}]_{li}   =  {\phi}_{h_3, h_2  }\#_0  {\rho}_{h_1} ,
\end{equation}
 etc.. We define $\widehat{\Phi}_{g_3, g_2,g_1}$ as an  $m\times m$ matrix   whose $(l,i)$-entry
 is the modification (as a $3$-isomorphism  in  $  \mathcal{V}^* $)
 $$(\widehat{\Phi}_{g_3, g_2,g_1})_{li}=
\Phi_{h_3, h_2,h_1}: [  {\rho}_{h_3}\#_0   {\phi}_{h_2, h_1} ] \#_1   {\phi}_{h_3, h_2 h_1}\xy0;/r.22pc/:
  (0,0)*+{ }="1";
(10,0)*+{ }="2";
{\ar@3{->}   "1"+(0, 0);"2"+(0, 0)};
   \endxy  [  {\phi}_{h_3, h_2  }\#_0  {\rho}_{h_1}] \#_1 {\phi}_{h_3h_2, h_1},
 $$ and all other entries vanish.

 For $g_4\in G$, write
 $$    g_4r_{l }=r_th_4$$
for some $h_4\in H$. The $( t,i)$-entry of the $m\times m$ matrix
 $  [\widehat{\rho}_{g_4}\#_0\widehat{\Phi}_{g_3,g_2,g_1}]\#_1\widehat{ \phi}_{g_4,g_3g_2 g_1}  $ is the modification
$$
 [\rho_{h_4}\#_0\Phi_{h_3,h_2,h_1}]\#_1 \phi_{h_4,h_3h_2 h_1}
$$
of $  \mathcal{V}  $, and similarly we obtain  other terms in the   $3$-cocycle condition (\ref{eq:4-cocycle})
 for $\widehat{\Phi}$. So  the   $3$-cocycle condition (\ref{eq:4-cocycle})
 for $\widehat{\Phi}$ is reduced to the $3$-cocycle condition
 for $ {\Phi}$.  Note that functors or pseudonatural transformation or  modification we consider are matrices, of which entries are  in a strict $3$-subcategory $\mathcal{W}$ of $\mathcal{V}^*$. It follows from the strictness of $\mathcal{W}  $     that    $\widehat{\rho}$ is a strict $2$-categorical action of $G$ on $\mathcal{V}^m\approx{\rm Ind}_H^G(\mathcal{V})$.
\section{The $3$-character  of the induced strict $2$-categorical action }
\subsection{The $2$-categorical trace of the induced strict $2$-categorical action}\label{decomposition} As above
$\rho$ is a strict  $2$-categorical  action of $H $ on the $2$-category $\mathcal{V}$.
 Let $\mathcal{R}$ be a system of representatives of $G/H$. We have the decomposition
 $$
   \mathcal{R}=\mathcal{R}'\cup \mathcal{R}'',
 $$
where $\mathcal{R}':=\left \{r\in \mathcal{R}; r^{-1}fr\in H\right\}$,  $\mathcal{R}'':=\left \{r\in \mathcal{R}; r^{-1}fr\notin H\right\}$.
 For a fixed element  $f$ of $ G$, the decomposition
 $$
    [f]_G\cap H=[h_1]_H\cup \cdots [h_n]_H
 $$
 induces a decomposition
 $$
   \mathcal{R}'=\bigcup_{i=1}^n \mathcal{R}_i \qquad {\rm with}  \qquad\mathcal{R}_i=\left\{r\in \mathcal{R};
   r^{-1}fr\in [h_i]_H\right\}.
 $$
 For fixed $i$, we pick $r_i\in \mathcal{R}_i$ and write $h_i=r_i^{-1}fr_i$. For $r\in  \mathcal{R}_i$, we have $ r^{-1}fr=h^{-1}h_ih$ for some $h\in H$. From
 now on, by replacing $r$ by $rh^{-1}$ in the representatives of $\mathcal{R}_i\subset G/H$, we can assume
 \begin{equation}\label{eq:Ri}
  r^{-1}fr= h_i \qquad {\rm for} \quad {\rm all} \qquad r\in \mathcal{R}_i.
 \end{equation}
Denote
 $$
  m_i:=|\mathcal{R}_i| ,  \qquad m':=|\mathcal{R}'|=\sum_{i=1}^n m_i  ,  \qquad m'':=|\mathcal{R}''|,  \qquad m:=m'+m'' .$$
It follows from the definition (\ref{eq:ind-rho})-(\ref{eq:Ind-act}) of $\widehat{\rho}_f$ that
 \begin{equation}\label{eq:rho-f}
    \widehat{\rho}_f=\left(
    \begin{array}{ccccc} A_{00}&A_{01}&A_{02}&\cdots& A_{0n}\\A_{10}& A_{11} &0 &\cdots &0 \\ A_{20}&0&A_{22} &\cdots& 0 \\\vdots&\vdots&\vdots &\ddots& \vdots  \\A_{n0}&0 &0&\cdots& A_{nn}
\end{array}
\right),\qquad  A_{i i}=\left(
    \begin{array}{lll}  \rho_{h_i}&& \\ &\ddots& \\ &&\rho_{h_i}
\end{array}
\right)_{m_i\times m_i},
 \end{equation}
  where $i=1,\ldots, n  $,   and  $A_{00}$ is a off-diagonal $m''\times m''$ matrix. So an object  of $ \mathbb{T}r_2\widehat{\rho}_f$ is a pseudonatural transformation $\chi:1_{\mathcal{V}^m}\rightarrow  \widehat{\rho}_f$ of the form
    \begin{equation}\label{eq:Trf-0}
    \chi=\left(
    \begin{array}{llll} 0_{m''\times m''}&&& \\&\ddots && \\& &D_i& \\& &&\ddots
\end{array}
\right),\qquad D_i=\left(
    \begin{array}{lll}  \chi_{m_1+\cdots+m_{i-1}+1}&& \\ &\ddots& \\ &&\chi_{m_1+\cdots+m_{i }}
\end{array}
\right),
\end{equation}where $\chi_{m_1+\cdots+m_{i-1}+\alpha}:1_{\mathcal{V}}\rightarrow \rho_{h_i}$ is an object of $ \mathbb{T}r_2 \rho_{h_i}$,
  $\alpha=1,\ldots,m_i $. Also morphisms
   in $ \mathbb{T}r_2\widehat{\rho}_f$  are   diagonal. So we have
  $$
 \mathbb{T}r_2\widehat{\rho}_f=\bigoplus_{i=1}^n  (\mathbb{T}r_2 {\rho}_{h_i})^{m_i} .$$

 \begin{lem}\label{lem:Ri} {\rm (\cite{GK}, Lemma 7.7)} Left multiplication with $r_i^{-1}$ maps $\mathcal{R}_i$ into a system of representatives of
 $C_G(h_i)/C_H(h_i)$.
   \end{lem}

      For $g\in C_G(f)$ and $r\in \mathcal{R}_i$, we write
      \begin{equation}\label{eq:g-r}
         gr=\widetilde{r}h,
      \end{equation}
     for some $\widetilde{r}\in \mathcal{R}$ and $h\in H$. Also,  $r$ is uniquely determined by $\widetilde{r}$ for fixed $g$. Then
                      $$
              \widetilde{r}^{-1} f\widetilde{r}=hr^{-1}g^{-1}fgrh^{-1}=hh_ih^{-1}
         $$
 by (\ref{eq:Ri}). Hence $\widetilde{r}\in \mathcal{R}_i$ and
    so  $ \widetilde{r}^{-1}f\widetilde{r}=h_i$ by   assumption (\ref{eq:Ri}). It follows that $h\in C_H(h_i)$. Then
      $$\begin{array}{rll}
         gr&=\widetilde{r}h\qquad\qquad  &{\rm gives}\qquad (\widehat{\rho}_g )_{ \widetilde{r}r}=\rho_h ,\\
         f {r}&= {r}h_i\qquad\qquad  &{\rm gives}\qquad (\widehat{\rho}_f )_{  {r} {r}}=\rho_{h_i},\\
          g^{-1}\widetilde r&={r}h^{-1}\qquad &{\rm gives}\qquad (\widehat{\rho}_{g^{-1}})_{r \widetilde{r}}=\rho_{h^*},
   \end{array}  $$
   and all other entries vanish. Thus
   $$
    \left (\widehat{\rho}_g\widehat{\rho}_f  \widehat{\rho}_{g^*} \right)_{\widetilde{r}\widetilde{r}}=\rho_h\rho_{h_i}\rho_{h^*} £¬
   $$
and  all other entries in the last $(m'\times m')$-block vanish  (see  (\ref{eq:rho-f})).

We denote by $ \widetilde{ \psi}$ the categorical action of the centralizer $  C_G(f)$ of $f$  on the category $ \mathbb{T}r_2\widehat{\rho}_f$. By definition (\ref{eq:psi-g}), $ \widetilde{ \psi}_g$ for $g\in C_G(f)$ is an invertible functor as follows. For      a pseudonatural transformation diag$(\ldots,\chi_r,\ldots)=\chi :1_{\mathcal{V}^m}\rightarrow\widehat{\rho}_f$
in (\ref{eq:Trf-0}),  $ \widetilde{ \psi}_g( {\chi})$ is a pseudonatural transformation  given by
$$
   {\rm diag}(1_{\mathcal{V}}, \ldots,1_{\mathcal{V}}  ) \xrightarrow{\widehat{\phi}_{g ,g^*}^{-1}} \widehat{\rho}_g  \widehat{\rho}_{g^*}\xrightarrow{ \widehat{\rho}_{g } \#_0  {\chi}\#_0\widehat{\rho}_{g^* }} \widehat{\rho}_g\widehat{\rho}_f \widehat{\rho}_{g^*}\xrightarrow{\widehat{\phi}_{g ,f}\#_0 \widehat{\rho}_{g^* } } \widehat{\rho}_{g f } \widehat{\rho}_{g^*} \xrightarrow{\widehat{\phi}_{g f,g^*}  } \widehat{\rho}_{g  f  g^*}=\widehat{\rho}_{ f   },
$$
 where the first $m''$ diagonal terms  of $ \widetilde{ \psi}_g( {\chi})$ must vanish, and other diagonal terms are
$$\begin{array}{rl}
  \left( \widehat{\phi}_{g ,g^*}^{-1}\right)_{\widetilde{r}\widetilde{r}}&=\phi_{h ,h^*}^{-1} : 1_{\mathcal{V}}\rightarrow \rho_h\rho_{h^*} ,\\
 \left (\widehat{\rho}_{g  }\#_0  {\chi}\#_0 \widehat{\rho}_{g^* }\right)_{\widetilde{r}\widetilde{r}}&= (\widehat{\rho}_{g })_{\widetilde  r{r}}\#_0 {\chi}_{ {r} {r}}\#_0 (\widehat{\rho}_{g^* } )_{{r}\widetilde r}=  \rho_h\#_0 \chi_{ {r}}\#_0\rho_{h^*}: \rho_h\rho_{h^*} \rightarrow \rho_h\rho_{h_i}\rho_{h^*} ,\\
 \left (\widehat{\phi}_{g ,f}\#_0 \widehat{\rho}_{g^* }\right)_{\widetilde{r}\widetilde{r}}&=\phi_{h ,h_i }\#_0  \rho_{h^*}: \rho_h\rho_{h_i}\rho_{h^*}\rightarrow \rho_{h h_i} \rho_{h^*}\\
\left ( \widehat{\phi}_{g f,g^*}  \right)_{\widetilde{r}\widetilde{r}}&=\phi_{h  h_i,h^*}:\rho_{h  h_i} \rho_{h^*}\rightarrow\rho_{h  h_i h^*}=\rho_{  h_i}.
\end{array}  $$
All other entries vanish by definitions (\ref{eq:hat-rho})-(\ref{eq:hat-phi}). Therefore, $  \widetilde{\psi}_g( {\chi)}$ is a diagonal $m\times m$ matrix of pseudonatural transformations, whose $(\widetilde{r},\widetilde{r})$-entry for $\widetilde{r}\in \mathcal{R}'$ is
\begin{equation}\label{eq:r-r-entry}
 \left(   \widetilde{ {\psi}}_g( {\chi})\right)_{\widetilde{r}\widetilde{r}}=\phi_{h ,h^*}^{-1}  \#_1
  [\rho_{h }\#_0 \chi_{ {r}}\#_0 \rho_{h^*} ]\#_1
  [\phi_{h ,h_i }\#_0  \rho_{h^*}] \#_1
  \phi_{h h_i,h^*}: 1_{\mathcal{V}}\rightarrow\rho_{   h_i  },
\end{equation}
and vanishes for all $\widetilde{r}\in \mathcal{R}''$.

Now denote by $\psi^{(i)}$   the categorical action of the centralizer $  C_H( {h_i})$     on the category $ \mathbb{T}r_2 {\rho}_{h_i}$, which is  constructed from the strict $2$-categorical action $\rho$ of $H$  on $\mathcal{V}$. Recall that by definition (\ref{eq:psi-g}), we have a functor  $ {\psi}_h^{(i)}$ for each $h\in C_H( {h_i})$.
For $h\in C_H(h_i)$ and a pseudonatural transformation $\omega:1_{\mathcal{V}}\rightarrow\rho_{h_i}$ , the pseudonatural transformation $\psi_h^{(i)}(\omega)  $ is   again by   definition (\ref{eq:psi-g})
the composition   of the following   pseudonatural transformations between functors:
$$
  1_{\mathcal{V}}\xrightarrow{\phi_{h ,h^*}^{-1}}\rho_{h  }  \rho_{h^* } \xrightarrow{\rho_{h  }\#_0 \omega\#_0 \rho_{h^* } }\rho_{h  }  \rho_{h_i} \rho_{h^* }\xrightarrow{\phi_{h ,h_i}\#_0 \rho_{h^* } }\rho_{h   h_i}  \rho_{h^* }\xrightarrow{\phi_{h  h_i,h^*}  } \rho_{h   h_i h^* }=\rho_{  h_i}.
$$
Then we see that (\ref{eq:r-r-entry}) can be written as
\begin{equation}\label{eq:hat-psi-rr}
   \left (   \widetilde{ \psi }_g\left( {\chi}\right)\right)_{\widetilde{r}\widetilde{r}}=\psi_h^{(i)}(\chi_{ {r}}): 1_{\mathcal{V}}\rightarrow\rho_{   h_i  },
\end{equation}
with $r,\widetilde{r}\in R_i$ and $h$ determined by (\ref{eq:g-r}). Namely, the resulting $\widetilde{r}$-th diagonal term is   the image of the $ {r}$-th diagonal term under the action of the functor $\psi_h^{(i)}$.

Note that we have the identification
   \begin{equation}\label{eq:Ind-C-G}
  {\rm Ind}_{C_H(h_i)}^{C_G(h_i)}    \mathbb{T}r_2 \rho_{h_i}\cong (\mathbb{T}r_2 \rho_{h_i})^{m_i},
   \end{equation}
   since $|C_G(h_i)/C_H(h_i)|=m_i$ by Lemma \ref{lem:Ri}, and that (\ref{eq:g-r}) is equivalent to
   \begin{equation}\label{eq:h-C-G}
     ( r_i^{-1}gr_i)(r_i^{-1}r)=(r_i^{-1}\widetilde{r})h, \qquad h\in C_H(h_i).
   \end{equation} The coset $  C_G(h_i)/C_H(h_i)$ are represented by $r_i^{-1}r$ for $r\in \mathcal{R}_i$ by Lemma \ref{lem:Ri} again, and an element of $  C_G(h_i)$ can always be
written as $ r_i^{-1}gr_i$ for some $g\in C_G(f)$. As above we denote by $\widehat{\psi^{(i)}}$   the induced   action of the centralizer $  C_G( {h_i})$  of $h_i$   on the category $ {\rm Ind}_{C_H(h_i)}^{C_G(h_i)}\mathbb{T}r_2 \rho_{h_i}$. Recall    the definition (\ref{eq:ind-rho})-(\ref{eq:Ind-act}) of the induced   action. So the action of $ r_i^{-1}gr_i\in C_G(h_i)$  on the induced category (\ref{eq:Ind-C-G})
 is  given by  the functor $\widehat{\psi^{(i)}}_{r_i^{-1}gr_i} $ on $(\mathbb{T}r_2 \rho_{h_i})^{m_i}$ with
   \begin{equation}\label{eq:psi-rr}
 \left(\widehat{\psi^{(i)}}_{r_i^{-1}gr_i} (\chi) \right )_{r_i^{-1}\widetilde{r},r_i^{-1} {r} }=\psi_h^{(i)}\left(\chi_{{r_i^{-1} r}}\right) : 1_{\mathcal{V}}\rightarrow\rho_{   h_i  }.
   \end{equation} for $\chi\in(\mathbb{T}r_2 \rho_{h_i})^{m_i} $, where $h$ is given by (\ref{eq:h-C-G}), and all other entries vanish. Here we use the expressions $ r_i^{-1} r$ as indices of the components of $(\mathbb{T}r_2 \rho_{h_i})^{m_i}$.
Comparing (\ref{eq:hat-psi-rr}) with (\ref{eq:psi-rr}), we find that the action of $g\in C_G(f)$ on $ (\mathbb{T}r_2 \rho_{h_i})^{m_i}$ coincides with the induced action of $ r_i^{-1}gr_i\in C_G(h_i)$ on it, and so
    the action of the centralizer $C_G(f)$ on $\mathbb{T}r_2  \widehat{\rho}_f$ decomposes into actions on
   \begin{equation}\label{eq:decomposition}
      \bigoplus_{ i }  (\mathbb{T}r_2 \rho_{h_i})^{m_i}= \bigoplus_{ i } {\rm Ind}_{C_H(h_i)}^{C_G(h_i)}    \mathbb{T}r_2 \rho_{h_i}.
   \end{equation}

Recall that the {\it initia groupoid} $\Lambda(G)$ of a group $G$ has as objects, the elements of $G$, and for two such elements $u$ and $v$,
   there is one morphism in $\Lambda(G)$ from $u$ to $v$ for every   $g\in G$ such that $v=gug^{-1}$.  Note that the   initia groupoid  $\Lambda(G)$ is equivalent to the groupoid with the set of objects consisting of the conjugacy classes $[g_i]$  and the set of morphisms consisting of $g:[g_i]\rightarrow [g_i]$ for   $g\in C_G(g_i)$. Therefore the above result can be summarized as follows.
\begin{thm}  Let $\mathcal{V}$ be a $k$-linear   $2$-category. The $2$-categorical trace $ \mathbb{T}r_2$ takes induced strict $2$-categorical action into the induced
 categorical action of the associated initia groupoids, i.e.  (\ref{eq:3-trace}) holds.
   \end{thm}

   \begin{rem} Even for the  categorical action, Section 4 above and the present subsection provide some details not written down explicitly in   section 7.2 of \cite{GK}.      \end{rem}
\subsection{The $3$-character formula}

Recall the  $2$-character formula for an  induced categorical action.
\begin{thm} \label{thm:2-ch}  {\rm (\cite{KV}, Corollary 7.6)} Let $\varrho$ be a categorical action of a subgroup $H$ of a finite group $G$ on a  $ {k}$-linear category $\mathcal{W}$.
Suppose that $\mathbb{T}r  \varrho_h $
is finite dimensional for each $h\in H$. Then the $2$-character of the induced categorical action of $G$ is given by
\begin{equation}
   \chi_{\rm ind}(f,g )=\frac 1{|H|}\sum_{\substack {s\in G\\ s^{-1}(f,g)s\in H\times H}} \chi_\varrho(s^{-1}fs,s^{-1}gs)
\end{equation}
for $g\in C_G(f)$.
   \end{thm}
We now state:
 \begin{thm}  \label{thm:3-ch}  Let $H$ be a subgroup of a finite group $G$  and let $\rho$ be a strict $2$-categorical action of  $H$ on the $2$-category $\mathcal{V}$. Let $\psi$ be the categorical actions of the centralizers   on  the $2$-categorical trace. Suppose that $\mathbb{T}r  \psi_h $
is finite dimensional for each $h\in H$. Then the $3$-character of the induced  strict  $2$-categorical action of $G$ is given by
\begin{equation}\label{eq:3-character}
   \chi_{\rm ind}(f,g,k)=\frac 1{|H|}\sum_{\substack {s\in G\\s^{-1}(f,g,k)s\in H\times H\times H}} \chi_\rho(s^{-1}fs,s^{-1}gs,s^{-1}ks)
\end{equation}for $f,g$ and $k $ pairwise commutative.
   \end{thm}
 \begin{proof}
By the decomposition (\ref{eq:decomposition}) of the action of $C_G(f)$ on $\mathbb{T}r_2  \widehat{\rho}_f$ and (\ref{eq:hat-psi-rr})-(\ref{eq:psi-rr}), we have
 $$
        \chi_{\rm ind}(f,g,k) =\sum_{i=1}^{n}\chi_{\widehat{\psi^{(i)}}  }   (r_i^{-1}gr_i,r_i^{-1}kr_i).  $$
 Now apply Theorem \ref{thm:2-ch} to the categorical action $\widehat{\psi^{(i)}} $ (\ref{eq:psi-rr}) of $C_G(h_i)$,  which is  induced  from the categorical action $\psi^{(i)}  $ of $C_H(h_i)$ on $\mathbb{T}r_2 \rho_{h_i}$,  to get
$$
  \chi_{\rm ind}(f,g,k)
  =\sum_{i=1}^{n}\frac 1{|C_H(h_i)|}\sum_{\substack {t\in C_G(h_i)\\ t^{-1}r_i^{-1}(g , k)r_it\in C_H(h_i)\times C_H(h_i)}} \chi_{\psi^{(i)}   }( t^{-1} r_i^{-1}gr_it,t^{-1}r_i^{-1}kr_it).
  $$ Recall that $\psi^{(i)}  $ is the  categorical action of $C_H(h_i)$ on $\mathbb{T}r_2 \rho_{h_i}$ constructed  from the strict $2$-categorical action $\rho$ of $H$. So we have
  $$
     \chi_{\psi^{(i)}   }( t^{-1} r_i^{-1}gr_it,t^{-1}r_i^{-1}kr_it)=\chi_\rho( h_i, t^{-1} r_i^{-1}gr_it,t^{-1}r_i^{-1}kr_it)  $$by the definition of the $3$-character  (\ref{eq:3ch-def}) for the strict $2$-categorical action $\rho$  of group $H$.
  Moreover,  the decomposition of the action of $C_G(f)$ on $\mathbb{T}r_2  \widehat{\rho}_f$ in Section \ref{decomposition} is independent of the choice of $h_i\in [h_i]_H$,    conjugacy class of $h_i$ in $H$. Therefore,
 $$ \chi_{\rm ind}(f,g,k) =\sum_{h\in H } \frac 1{|[h]_H|}\frac 1{|C_H(h)|}\sum_{\substack { s^{-1}fs=h,s\in G,\\ s^{-1}(g,k)s\in C_H(h)\times C_H(h)} } \chi_\rho(
  h, s^{-1}gs,s^{-1}ks).
   $$
Here  we have used the fact that
$
    h_i= s^{-1}fs=s^{-1}r_ih_i r_i^{-1} s
$
if and only if  $r_i^{-1} s\in C_G(h_i)$.
Note that for $s\in G$, we have $s^{-1} g s$ (resp. $ s^{-1} ks$) $\in H $ if and only if  $s^{-1} g s$ (resp. $ s^{-1} k s$) $\in C_H(h) $ since
$g$ and $k$ commute with $f= shs^{-1}$.  The $3$-character formula (\ref{eq:3-character}) follows.
 \end{proof}
\section{The categorical action of the   centralizer of $ f $ on ${\rm \mathbb{T}r}_2 \rho_f$}
\subsection{A  model: the $1$-dimensional case}
Let us prove by using  the condition (\ref{eq:3-cocycle-0}) for $3$-cocycles repeatedly  that  the expression $  \Gamma$ given in (\ref{eq:2-cocycle}) is a $2$-cocycle on the   centralizer $C_G(f)$. This proof corresponds step by step to that of the general case  carried out in Section  6.4.

{\it Proof of Proposition \ref{prop:2-cocycle}}.   By the definition of $\Gamma_{*,*}$ in (\ref{eq:2-cocycle}), we see that
$$
   \Gamma_{ h,g} \Gamma_{k,hg }= \frac {\Pi_f} {\Pi_1},
$$
where
\begin{equation}\label{eq:1-cocycle-0}\begin{split}
\Pi_f:= c( h,gf,g^*)   c( h,g,f ) c( hgf ,g^*,h^* )^{-1}
 \cdot     c(k,\underline{hf},g^*h^*  )
c(k,h g,f ) c(\underline{kf},g^*h^*,k^*)^{-1} ,
\end{split}\end{equation}
and $ \Pi_1$ is just $ \Pi_f$ with $f$ replaced by $1$. Similarly, we have
$$
   \Gamma_{k,h }  \Gamma_{kh,g}=  \frac   {\Pi_f'} {\Pi_1'},
$$where
\begin{equation}\label{eq:1-cocycle-1}\begin{split}
\Pi_f'=c(k,hf,h^*)\mathbf{c(k,h,f)}c(khf,h^*,k^*)^{-1} \cdot \mathbf{c(kh,gf,g^*)}c(kh,g,f)c(\underline{k f},g^*,(kh)^*)^{-1},
\end{split}\end{equation}and $ \Pi_1'$ is just $ \Pi_f'$ with $f$ replaced by $1$.

 Apply the $3$-cocycle condition (\ref{eq:3-cocycle-0}) to the product of the two boldface terms in (\ref{eq:1-cocycle-1})  with $g_4=k,g_3=h,g_2=gf,g_1=g^*$ to get
\begin{equation}\label{eq:1-cocycle-2}\begin{array}{ll}\Pi_f'=&
c( h,gf,g^*)c(k,h gf,g^*)c(k,h,gf ) \\ &\cdot c(k,hf,h^*)\mathbf{c(khf,h^*,k^*)}^{-1}  c(kh,g,f)\mathbf{c(\underline{k f},g^*,(kh)^*)}^{-1} .\end{array}
\end{equation}
Here the second line above is the right-hand side of (\ref{eq:1-cocycle-1}) with the two boldface terms deleted. Note that $\underline{kf} g^*=khf$.
 Apply the $3$-cocycle condition (\ref{eq:3-cocycle-0}) to the product of the two boldface  terms in (\ref{eq:1-cocycle-2}) with  $g_4= \underline{kf} ,g_3=g^*,g_2=h^*,g_1=k^*$ to get
\begin{equation}\label{eq:1-cocycle-3}\begin{array}{ll}\Pi_f'=& c(g^*,h^*,k^*)^{-1}c(\underline{kf},g^*h^*,k^*)^{-1}\mathbf{c(\underline{kf},g^*,h^* )^{-1}}
\\&\cdot
 {c( h,gf,g^*)}\mathbf{c(k,\underline{hf},g^*)}c(k,h,gf ) \cdot \mathbf{c(k,hf,h^*)}  \cdot  { c(kh,g,f)} ,
\end{array} \end{equation}Here the second line above is the right-hand side of (\ref{eq:1-cocycle-2}) with the two boldface terms deleted.
 Apply the $3$-cocycle condition (\ref{eq:3-cocycle-0}) to the product of the three boldface terms in (\ref{eq:1-cocycle-3})
with $g_4=k$, $g_3=\underline{hf}$, $g_2=g^*$,$g_1=h^*$  to get
\begin{equation}\label{eq:1-cocycle-4'}\begin{array}{ll}\Pi_f'=& c(k,\underline{hf},g^*h^*  ) c(\underline{hf},g^*,h^* )^{-1}
\\ &\cdot c(g^*,h^*,k^*)^{-1}c(\underline{kf},g^*h^*,k^*)^{-1}
c( h,gf,g^*)    \mathbf{c(k,h,gf )  }  \mathbf{ c(kh,g,f)}  ,
\end{array} \end{equation} by $k\underline{hf}=\underline{kf}$ and $\underline{hf}g^*=hf$. Here the second line above is the right-hand side of (\ref{eq:1-cocycle-3}) with the three boldface terms deleted.
Apply the $3$-cocycle condition (\ref{eq:3-cocycle-0}) to the product of the two boldface terms in (\ref{eq:1-cocycle-4'})
with $g_4=k$, $g_3=h$, $g_2=g $, $g_1=f$  to get
\begin{equation}\label{eq:1-cocycle-4}\begin{array} {ll}\Pi_f'=&c( h,g,f ) \cdot c(k,h g,f ) \mathbf{c(k,h,g ) }
\\ &  \cdot c(k,\underline{hf},g^*h^*  ) c(\underline{hf},g^*,h^* )^{-1}\mathbf{c(g^*,h^*,k^*)^{-1}}c(\underline{kf},g^*h^*,k^*)^{-1}
c( h,gf,g^*).
\end{array} \end{equation}
For $f=1$ in (\ref{eq:1-cocycle-4}), we see that $\Pi_1'$ also has the product $c(k,h,g )c(g^*,h^*,k^*)^{-1}$ of the two boldface terms, which is independent of $f$. They are cancelled in $  {\Pi_f'}/{\Pi_1'}$. So we get
$$
\frac   {\Pi_f'} {\Pi_1'}=  \frac {\Pi_f} {\Pi_1}
$$
by comparing (\ref{eq:1-cocycle-4}) and (\ref{eq:1-cocycle-0}). Proposition \ref{prop:2-cocycle} is proved.

\subsection{The natural isomorphism    $  \Gamma_{k,h g} {\#} (\psi_k \circ\Gamma_{h,g}) $ } Let us  write down the natural isomorphism     $  \Gamma_{k,h g} {\#} (\psi_k \circ\Gamma_{h,g}) $.
For a fixed $\chi\in  ({\rm \mathbb{T}r}_2 \rho_f)_0$, by using the definition of   compositions in (\ref{eq:2-composition}) twice, we see that $\psi_k \circ\psi_h
\circ\psi_g(\chi)$ is the composition of the following $2$-arrows:
 \begin{equation}\label{eq:3-composition}
      \xy 0;/r.22pc/:
 (-15,0)*+{ {x}}="1";
(15,0)*+{ {x} }="2";
(-30,0)*+{ x}="3";
(30,0)*+{ x}="4";
(-45,0)*+{ x}="5";
(45,0)*+{ x}="6";
(-60,0)*+{ x}="7";
(60,0)*+{ x}="8";
{\ar@{->}^{\rho_g} "3"; "1"};
{\ar@{-->}^{\rho_{g^*}} "2"; "4"};
{\ar@{-->}^{\rho_h} "5"; "3"};
{\ar@{->}^{\rho_{h^*}} "4"; "6"};
{\ar@{->}^{\rho_k} "7"; "5"};
{\ar@{->}^{\rho_{k^*}} "6"; "8"};
{\ar@/^1.33pc/|-{1_x} "1";"2"};
{\ar@/_1.33pc/|-{\rho_f} "1";"2"};
 {\ar@{=>}^{\chi} (0,3)*{}; (0,-3)*{}};
 {\ar@{=>}^{u_g} (0,13)*{}; (0,7)*{}};{\ar@{=>}^{u_h} (0,22)*{}; (0,16)*{}};{\ar@{=>}^{u_k} (0,32)*{}; (0,26)*{}};
 {\ar@/^3.33pc/|-{1_x} "3";"4"}; {\ar@/^5.33pc/|-{1_x} "5";"6"};{\ar@/^7.33pc/^{1_x} "7";"8"};
 {\ar@{-->}@/_2.33pc/^{\rho_{gf}} "3";"2"};
 {\ar@{-->}@/_3.33pc/^{\rho_{\underline{gg^*}}} "3";"4"};
  {\ar@{-->}@/_4.33pc/^{\rho_{\underline{hg^*}}} "5";"4"};
   {\ar@/_5.33pc/^{\rho_{\underline{hh^*}}} "5";"6"};
   {\ar@/_6.33pc/^{\rho_{\underline{kh^*}}} "7";"6"};
    {\ar@/_7.33pc/_{\rho_{\underline{kk^*}}} "7";"8"};
\endxy.
 \end{equation}
  Let us calculate the $3$-isomorphism
  \begin{equation}\label{eq:left-hand side-chi}
     [\Gamma_{k,h g} {\#}( \psi_k \circ\Gamma_{h,g})](\chi)
:\psi_k \circ\psi_h \circ\psi_g(\chi)\xy0;/r.22pc/:
  (0,0)*+{ }="1";
(6,0)*+{ }="2";
{\ar@3{->}    "1"+(0, 0);"2"+(0, 0)};
   \endxy \psi_{khg}(\chi)
  \end{equation}
  for a fixed $2$-arrow $\chi\in \mathbb{T}r_2   {\rho}_f\subset \mathcal{C}^{++}$. We consider the lower half part of (\ref{eq:3-composition}) first.
The $3$-isomorphism
  \begin{equation}\label{eq:Lambda-1}
   \Lambda_1=\diamondsuit\#_1[\rho_k\#_0 \Phi_{h,gf,g^*}\#_0 (\rho_{h^*}\rho_{k^*})]\#_1 \diamondsuit,\end{equation}
   the associator $\Phi_{h,gf,g^*}$  (\ref{eq:associator})  whiskered by $2$-isomorphisms $\diamondsuit $ which we do not write down explicitly, changes the diagonal $\rho_{\underline{gg^*}}$ of the  dotted
  quadrilateral in (\ref{eq:3-composition}) to the wavy diagonal
$\rho_{\underline{hf}}$ of the same   quadrilateral  in the following  diagram:
 \begin{equation}\label{eq:Lambda-1-Diagram}
      \xy 0;/r.22pc/:
 (-15,0)*+{ {x}}="1";
(15,0)*+{ {x} }="2";
(-30,0)*+{ x}="3";
(30,0)*+{ x}="4";
(-45,0)*+{ x}="5";
(45,0)*+{ x}="6";
(-60,0)*+{ x}="7";
(60,0)*+{ x}="8";
{\ar@{-->}^{\rho_g} "3"; "1"};
{\ar@{->}^{\rho_{g^*}} "2"; "4"};
{\ar@{-->}^{\rho_h} "5"; "3"};
{\ar@{->}^{\rho_{h^*}} "4"; "6"};
{\ar@{->}^{\rho_k} "7"; "5"};
{\ar@{->}^{\rho_{k^*}} "6"; "8"};
 {\ar@{-->}^{\rho_f} "1";"2"};
 {\ar@{-->}@/_2.33pc/^{\rho_{gf}} "3";"2"};
 {\ar@{~>}@/_3.33pc/^{\rho_{\underline{hf} }} "5";"2"};
  {\ar@/_4.33pc/^{\rho_{\underline{hg^*}}} "5";"4"};
   {\ar@/_5.33pc/^{\rho_{\underline{hh^*}}} "5";"6"};
   {\ar@/_6.33pc/^{\rho_{\underline{kh^*}}} "7";"6"};
    {\ar@/_7.33pc/_{\rho_{\underline{kk^*}}} "7";"8"};
\endxy.
 \end{equation}This is a $3$-arrow as (\ref{eq:lambda-1-3-dim}). The $2$-arrows outside the quadrilateral are fixed as the whiskering parts.
The $3$-isomorphism  \begin{equation}\label{eq:Lambda-2}
   \Lambda_2=\diamondsuit\#_1[\rho_k\#_0 \Phi_{h,g,f   }\#_0 (\rho_{g^*}\rho_{h^*}\rho_{k^*})]\#_1 \diamondsuit,     \end{equation}
as a whiskered associator (\ref{eq:associator}),  changes the diagonal $\rho_{ {gf}}$ of the above dotted-wavy quadrilateral to the wavy
diagonal $\rho_{ {hg}}$ of the same   quadrilateral   in the following  diagram:
 \begin{equation}\label{eq:Lambda-2-Diagram}
      \xy 0;/r.22pc/:
 (-15,0)*+{ {x}}="1";
(15,0)*+{ {x} }="2";
(-30,0)*+{ x}="3";
(30,0)*+{ x}="4";
(-45,0)*+{ x}="5";
(45,0)*+{ x}="6";
(-60,0)*+{ x}="7";
(60,0)*+{ x}="8";
{\ar@{->}^{\rho_g} "3"; "1"};
{\ar@{-->}^{\rho_{g^*}} "2"; "4"};
{\ar@{->}^{\rho_h} "5"; "3"};
{\ar@{-->}^{\rho_{h^*}} "4"; "6"};
{\ar@{->}^{\rho_k} "7"; "5"};
{\ar@{->}^{\rho_{k^*}} "6"; "8"};
 {\ar@{->}^{\rho_f} "1";"2"};
 {\ar@{~>}@/_1.33pc/^{\rho_{hg }} "5";"1"};
 {\ar@{-->}@/_3.33pc/^{\rho_{\underline{hf} }} "5";"2"};
  {\ar@{-->}@/_4.33pc/^{\rho_{\underline{hg^*}}} "5";"4"};
   {\ar@{-->}@/_5.33pc/^{\rho_{\underline{hh^*}}} "5";"6"};
   {\ar@/_6.33pc/^{\rho_{\underline{kh^*}}} "7";"6"};
    {\ar@/_7.33pc/_{\rho_{\underline{kk^*}}} "7";"8"};
\endxy.
 \end{equation}
The $3$-isomorphism
 \begin{equation}\label{eq:hat-Lambda-3}
   \Lambda_3=\diamondsuit\#_1[\rho_k\#_0 \Phi_{\underline{h f},g^*,h^*  }^{-1}\#_0  \rho_{k^*} ]\#_1 \diamondsuit,   \end{equation}as a whiskered associator (\ref{eq:associator}), changes the diagonal
   $\rho_{\underline{hg^*}}$ of the above dotted  quadrilateral to the wavy
diagonal $\rho_{ {g^*h^*}}$ of the same   quadrilateral   in the following  diagram:
 \begin{equation}\label{eq:Lambda-3-Diagram}
      \xy 0;/r.22pc/:
 (-15,0)*+{ {x}}="1";
(15,0)*+{ {x} }="2";
(-30,0)*+{ x}="3";
(30,0)*+{ x}="4";
(-45,0)*+{ x}="5";
(45,0)*+{ x}="6";
(-60,0)*+{ x}="7";
(60,0)*+{ x}="8";
{\ar@{->}^{\rho_g} "3"; "1"};
{\ar@{->}^{\rho_{g^*}} "2"; "4"};
{\ar@{->}^{\rho_h} "5"; "3"};
{\ar@{->}^{\rho_{h^*}} "4"; "6"};
{\ar@{-->}^{\rho_k} "7"; "5"};
{\ar@{->}^{\rho_{k^*}} "6"; "8"};
 {\ar@{->}^{\rho_f} "1";"2"};
 {\ar@/_1.33pc/_{\rho_{hg }} "5";"1"};
 {\ar@{-->}@/_3.33pc/^{\rho_{\underline{hf} }} "5";"2"};
  {\ar@{~>}@/_1.33pc/_{\rho_{ g^*h^*}} "2";"6"};
   {\ar@{-->}@/_5.33pc/^{\rho_{\underline{hh^*}}} "5";"6"};
   {\ar@{-->}@/_6.33pc/^{\rho_{\underline{kh^*}}} "7";"6"};
    {\ar@/_7.33pc/_{\rho_{\underline{kk^*}}} "7";"8"};
\endxy.
 \end{equation} Note that the diagrams (\ref{eq:Lambda-1-Diagram}), (\ref{eq:Lambda-2-Diagram}) and (\ref{eq:Lambda-3-Diagram}) are exactly the diagrams (\ref{eq:psi-g-h-2}),
 (\ref{eq:psi-g-h-3}) and (\ref{eq:psi-g-h-4}) by adding from below to each of these the arrows:
 $$
      
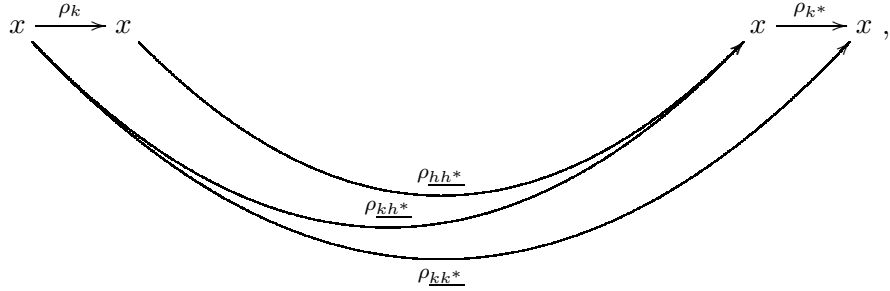
\xy 0;/r.22pc/:
 (-45,0)*+{ x}="5";
(45,0)*+{ x}="6";
(-60,0)*+{ x}="7";
(60,0)*+{ x}="8";{\ar@{->}@/_5.33pc/^{\rho_{\underline{hh^*}}} "5";"6"};
{\ar@{->}^{\rho_k} "7"; "5"};
{\ar@{->}^{\rho_{k^*}} "6"; "8"};
    {\ar@{->}@/_6.33pc/^{\rho_{\underline{kh^*}}} "7";"6"};
    {\ar@/_7.33pc/_{\rho_{\underline{kk^*}}} "7";"8"};
\endxy,
 $$
 By definition,
 the composition $\Lambda_1\#_2\Lambda_2\#_2\Lambda_3$ is the $3$-isomorphism
 $$
    [ \psi_k  \circ \Gamma_{h,g}](\chi):\psi_k \circ \psi_h \circ\psi_g(\chi)\xy0;/r.22pc/:
  (0,0)*+{ }="1";
(10,0)*+{ }="2";
{\ar@3{->}^{ }  "1" ;"2" };
   \endxy \psi_k
\circ\psi_{ hg}(\chi)
 $$
 corresponding to the lower half   of (\ref{eq:3-composition}).

The $3$-isomorphism
 $$
   \Lambda_4=\diamondsuit\#_1[  \Phi_{k,\underline{hf}, g^* h^* } \#_0  \rho_{k^*} ]\#_1 \diamondsuit,
     $$as a whiskered associator (\ref{eq:associator}), changes the diagonal $\rho_{\underline{hh^*}}$ of the   dotted-wavy  quadrilateral in (\ref{eq:Lambda-3-Diagram}) to the wavy
diagonal  $\rho_{\underline{kf}}$ of the same   quadrilateral   in the following  diagram:
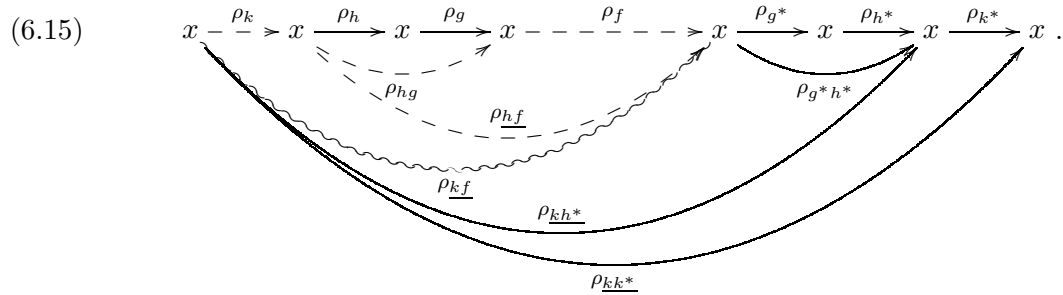
\begin{equation}\label{eq:Lambda-4}
      \xy 0;/r.22pc/:
 (-15,0)*+{ {x}}="1";
(15,0)*+{ {x} }="2";
(-30,0)*+{ x}="3";
(30,0)*+{ x}="4";
(-45,0)*+{ x}="5";
(45,0)*+{ x}="6";
(-60,0)*+{ x}="7";
(60,0)*+{ x}="8";
{\ar@{->}^{\rho_g} "3"; "1"};
{\ar@{->}^{\rho_{g^*}} "2"; "4"};
{\ar@{->}^{\rho_h} "5"; "3"};
{\ar@{->}^{\rho_{h^*}} "4"; "6"};
{\ar@{-->}^{\rho_k} "7"; "5"};
{\ar@{->}^{\rho_{k^*}} "6"; "8"};
 {\ar@{-->}^{\rho_f} "1";"2"};
 {\ar@{-->}@/_1.33pc/_{\rho_{hg }} "5";"1"};
 {\ar@{-->}@/_3.33pc/^{\rho_{\underline{hf} }} "5";"2"};
  {\ar@{->}@/_1.33pc/_{\rho_{ g^*h^*}} "2";"6"};
   {\ar@{~>}@/_4.33pc/_{\rho_{\underline{kf} }} "7";"2"};
   {\ar@{->}@/_6.33pc/^{\rho_{\underline{kh^*}}} "7";"6"};
    {\ar@{->}@/_7.33pc/_{\rho_{\underline{kk^*}}} "7";"8"};
\endxy.
 \end{equation}
The $3$-isomorphism
 $$
   \Lambda_5=\diamondsuit\#_1   \left[\Phi_{k , hg,f } \#_0 \left(\rho_{g^*} \rho_{h^*}  \rho_{k^*}\right)\right ]\#_1 \diamondsuit,   $$ as a whiskered associator (\ref{eq:associator}),  changes the diagonal $\rho_{\underline{hf}}$ of the above dotted-wavy  quadrilateral to
   the wavy
diagonal $\rho_{\underline{kg}}$  of the same   quadrilateral   in the following  diagram:
 
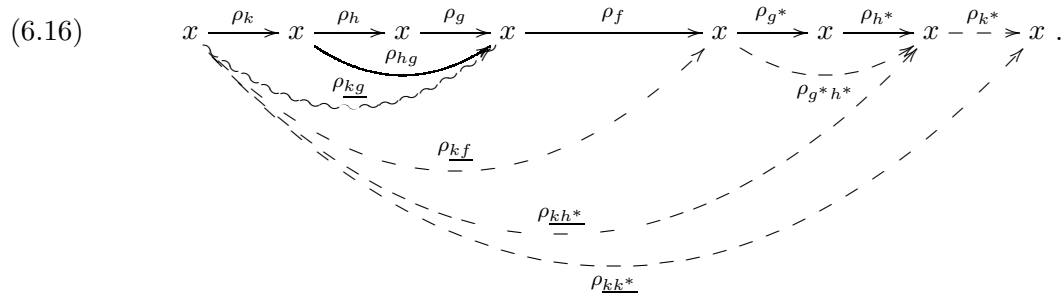
\begin{equation}\label{eq:Lambda-5}
      \xy 0;/r.22pc/:
 (-15,0)*+{ {x}}="1";
(15,0)*+{ {x} }="2";
(-30,0)*+{ x}="3";
(30,0)*+{ x}="4";
(-45,0)*+{ x}="5";
(45,0)*+{ x}="6";
(-60,0)*+{ x}="7";
(60,0)*+{ x}="8";
{\ar@{->}^{\rho_g} "3"; "1"};
{\ar@{->}^{\rho_{g^*}} "2"; "4"};
{\ar@{->}^{\rho_h} "5"; "3"};
{\ar@{->}^{\rho_{h^*}} "4"; "6"};
{\ar@{->}^{\rho_k} "7"; "5"};
{\ar@{-->}^{\rho_{k^*}} "6"; "8"};
 {\ar@{->}^{\rho_f} "1";"2"};
 {\ar@/_1.33pc/^{\rho_{hg }} "5";"1"};
 {\ar@{~>}@/_2.33pc/^{\rho_{\underline{kg} }} "7";"1"};
  {\ar@{-->}@/_1.33pc/_{\rho_{ g^*h^*}} "2";"6"};
   {\ar@{-->}@/_4.33pc/^{\rho_{\underline{kf} }} "7";"2"};
   {\ar@{-->}@/_6.33pc/^{\rho_{\underline{kh^*}}} "7";"6"};
    {\ar@{-->}@/_7.33pc/_{\rho_{\underline{kk^*}}} "7";"8"};
\endxy.
 \end{equation}
The $3$-isomorphism
 $$
   \Lambda_6=\diamondsuit\#_1   \Phi_{\underline{k  f} , g^*h^*,k^* }^{-1},   $$as a whiskered associator (\ref{eq:associator}), changes the diagonal $\rho_{\underline{kh^*}}$ of the above
   dotted  quadrilateral to the wavy
diagonal $\rho_{\underline{g^*k^*}}$  of the same   quadrilateral   in the following  diagram:
  \begin{equation}\label{eq:Lambda-6}
      \xy 0;/r.22pc/:
 (-15,0)*+{ {x}}="1";
(15,0)*+{ {x} }="2";
(-30,0)*+{ x}="3";
(30,0)*+{ x}="4";
(-45,0)*+{ x}="5";
(45,0)*+{ x}="6";
(-60,0)*+{ x}="7";
(60,0)*+{ x}="8";
{\ar@{->}^{\rho_g} "3"; "1"};
{\ar@{->}^{\rho_{g^*}} "2"; "4"};
{\ar@{->}^{\rho_h} "5"; "3"};
{\ar@{->}^{\rho_{h^*}} "4"; "6"};
{\ar@{->}^{\rho_k} "7"; "5"};
{\ar@{->}^{\rho_{k^*}} "6"; "8"};
 {\ar@{->}^{\rho_f} "1";"2"};
 {\ar@{->}@/_1.33pc/^{\rho_{hg }} "5";"1"};
 {\ar@{->}@/_2.33pc/^{\rho_{\underline{kg}  }} "7";"1"};
  {\ar@{->}@/_1.33pc/^{\rho_{ g^*h^*}} "2";"6"};
   {\ar@/_4.33pc/^{\rho_{\underline{kf} }} "7";"2"};
   {\ar@{~>}@/_2.33pc/_{\rho_{ \underline{g^*k^*}}} "2";"8"};
    {\ar@/_7.33pc/^{\rho_{\underline{kk^*}}} "7";"8"};
\endxy.
 \end{equation}
The composition $\Lambda_4\#_2\Lambda_5\#_2\Lambda_6$ is the $3$-isomorphism
$$
   \Gamma_{k,h g}(\chi):\psi_k \circ\psi_{h  g}(\chi)\xy0;/r.22pc/:
  (0,0)*+{ }="1";
(6,0)*+{ }="2";
{\ar@3{->}    "1"+(0, 0);"2"+(0, 0)};
   \endxy  \psi_{ khg}(\chi)
$$
 corresponding to the lower half  of (\ref{eq:3-composition}).

In   the $2$-category $ \mathcal{C}^{ +}$,
the composition $   \Lambda_1\#_2\cdots \#_2\Lambda_6$ of  $3$-isomorphisms
  corresponds to the following diagram   $\mathscr D_f^l:=$
  \begin{equation}\label{eq:diagram-1}
      \xy 0;/r.22pc/:
 (0,0)*+{ \rho_f }="1";
(30,0)*+{ \rho_{g f} }="2";
(60,0)*+{ \rho_{\underline{gg^*}}}="3";
(90,0)*+{ \rho_{\underline{hg^*}}}="4";
(120,0)*+{ \rho_{\underline{hh^*}}}="5";
(150,0)*+{ \rho_{\underline{kh^*}}}="6";
(180,0)*+{\rho_{\underline{kk^*}}}="7";
 (50,-25)*+{\rho_{\underline{hf}}   }="8";
  (20,-35)*+{\rho_{hg  }   }="9";
  (90,-20)*+{\rho_{\underline{hf}}\rho_{g^*  h^* }   }="10";          (50,-40)*+{\rho_{\underline{kg}   }    }="11";
         (120,-20)*+{\rho_{\underline{kf}  }\rho_{g^*  h^* }   }="12";
          (90,-35)*+{\rho_{\underline{kf}   }  }="13";
          (150,-20)*+{\rho_{\underline{kf}   }\rho_{g^*  h^*k^* }   }="14";
{\ar@{->}^{\phi_{g,f}} "1"; "2"};
{\ar@{->}^{\phi_{g f,g^*}} "2"; "3"};
{\ar@{->}^{\phi_{h, \underline{gg^*}}} "3"; "4"};
{\ar@{->}^{\phi_{\underline{hg^*}, h^*}} "4"; "5"};
{\ar@{->}^{\phi_{k,\underline{hh^*}}} "5"; "6"};
{\ar@{->}^{\phi_{\underline{kh^*},k^*}} "6"; "7"};
{\ar@{->}_{ \scriptscriptstyle\phi_{h,gf }} "2"; "8"};
{\ar@{->}|-{ \phi_{\underline{hf},g^* }} "8"; "4"};
{\ar@{->}_{ \phi_{h,g  }} "1"; "9"};
{\ar@{->}^{\scriptscriptstyle\phi_{h g,f  }} "9"; "8"};
{\ar@{->}^{\scriptscriptstyle \phi_{g^* ,  h^*}} "8"; "10"};
{\ar@{->}|-{ \phi_{\underline{hf}, g^*  h^*  }} "10"; "5"};
{\ar@{=}^{}(75,-28)*^{} ;(95,-28)*^{}   }; {\ar@{->}_{ \scriptscriptstyle\phi_{k,  \underline{hf}}} "8"; "13"};
{\ar@{->}|-{ \phi_{\underline{kf}, g^*  h^*  }} "12"; "6"};
{\ar@{->}^{\scriptscriptstyle \phi_{k,  \underline{hf}}} "10"; "12"};
{\ar@{->}_{ \phi_{  g^* , h^*  }} "13"; "12"};
{\ar@{->}_{ \phi_{g^*  h^*,k^*}} "12"; "14"};
{\ar@{->}_{ \phi_{\underline{kf}, \underline{g^*   k^*} }} "14"; "7"};
{\ar@{->}_{ \phi_{\underline{kg}, f  }} "11"; "13"};{\ar@{->}_{ \phi_{k,hg }} "9"; "11"};
{\ar@{=>}_{   \Lambda_1} (55,-3)*{}; (55,-20)*{}};
{\ar@{=>}_{   \Lambda_2} (25,-5)*{}; (25,-25)*{}};
{\ar@{=>}_{   \Lambda_3} (90,-5)*{}; (90,-15)*{}};
{\ar@{=>}_{   \Lambda_4} (120,-5)*{}; (120,-15)*{}};
{\ar@{=>}_{   \Lambda_5} (50,-27)*{}; (50,-37)*{}};
{\ar@{=>}^{   \Lambda_6} (150,-5)*{}; (150,-15)*{}};
  \endxy
 \end{equation}
where the symbol  $=$ in this diagram   follows from the interchange  law
(\ref{eq:interchanging-law}) for a horizontal composition: the commutativity of $ \phi_{k,  \underline{hf}}$ and $\phi_{g^*,  h^*}$. Note that the part involving $\Lambda_1,\Lambda_2,\Lambda_3$ is just the diagram (\ref{eq:lower}).
Let $\mathscr D_1^l $  be the corresponding  diagram in $\mathcal{C}^+$ with   $f$ replaced by $1$, by using adjoint operations as in
 (\ref{eq:upper}). Then as in (\ref{eq:D-1-f}), the $2$-isomorphism in $ \mathcal{C}^{  +}$ corresponding to the morphism
$  [\Gamma_{k,h g} {\#} (\psi_k \circ\Gamma_{h,g}) ](\chi)$ in $\mathbb{ {T}}r_2  \rho_f $
  is
 \begin{equation}\label{eq:D-l}
   \mathscr{ D}_1^l \xy0;/r.22pc/:
  (0,0)*+{ }="1";
(7,0)*+{ }="2";
{\ar@{->}^{\chi }   "1"+(0,0);"2"+(0,0)};
   \endxy  \mathscr{ D}_f^l.
 \end{equation}

\subsection{The natural isomorphism   $\Gamma_{kh ,g} {\#} (\Gamma_{k,h}\circ \psi_g) $ }
To calculate $  \Gamma_{k,h }\circ\psi_g$, we fix the part
$$
      \xy 0;/r.22pc/:
 (-15,0)*+{ {x}}="1";
(15,0)*+{ {x} }="2";
(-30,0)*+{ x}="3";
(30,0)*+{ x}="4";
{\ar@{->}^{\rho_g} "3"; "1"};
{\ar@{->}^{\rho_{g^*}} "2"; "4"};
 {\ar@{->}^{\rho_f} "1";"2"};
 {\ar@/_2.33pc/^{\rho_{gf}} "3";"2"};
 {\ar@{->}@/_3.33pc/_{\rho_{\underline{gg^*}}} "3";"4"};
 \endxy
 $$in the lower half   of (\ref{eq:3-composition}), which corresponds to $\psi_g$.
The $3$-isomorphism
 \begin{equation}\label{eq:tPhi-1}
   \widetilde{\Lambda}_1=\diamondsuit\#_1 \Phi_{k,\underline{hg^*},h^*} \#_1\diamondsuit,\end{equation}as a whiskered associator (\ref{eq:associator}),
   changes the $1$-isomorphism $\rho_{\underline{hh^*}}$ in the lower part of (\ref{eq:3-composition}) to the wavy
diagonal $\rho_{\underline{kg^*}}$ of the same   quadrilateral   in the following  diagram:
 $$
      
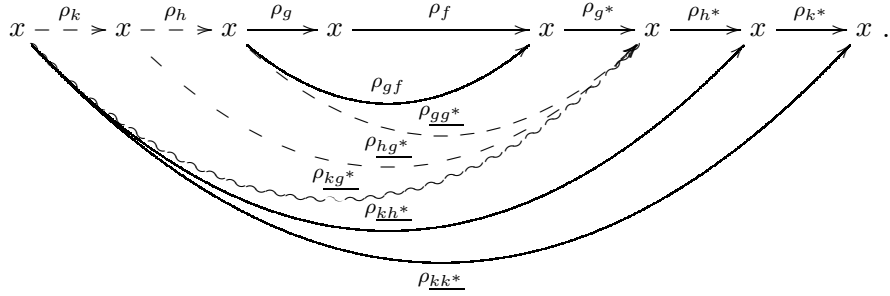
\xy 0;/r.22pc/:
 (-15,0)*+{ {x}}="1";
(15,0)*+{ {x} }="2";
(-30,0)*+{ x}="3";
(30,0)*+{ x}="4";
(-45,0)*+{ x}="5";
(45,0)*+{ x}="6";
(-60,0)*+{ x}="7";
(60,0)*+{ x}="8";
{\ar@{->}^{\rho_g} "3"; "1"};
{\ar@{->}^{\rho_{g^*}} "2"; "4"};
{\ar@{-->}^{\rho_h} "5"; "3"};
{\ar@{->}^{\rho_{h^*}} "4"; "6"};
{\ar@{-->}^{\rho_k} "7"; "5"};
{\ar@{->}^{\rho_{k^*}} "6"; "8"};
 {\ar@{->}^{\rho_f} "1";"2"};
 {\ar@/_2.33pc/^{\rho_{gf}} "3";"2"};
 {\ar@{-->}@/_3.33pc/^{\rho_{\underline{gg^*}}} "3";"4"};
  {\ar@{-->}@/_4.33pc/^{\rho_{\underline{hg^*}}} "5";"4"};
   {\ar@{~>}@/_5.33pc/^{\rho_{\underline{kg^*} }} "7";"4"};
   {\ar@/_6.33pc/^{\rho_{\underline{kh^*}}} "7";"6"};
    {\ar@/_7.33pc/_{\rho_{\underline{kk^*}}} "7";"8"};
\endxy.
 $$
The $3$-isomorphism
 \begin{equation}\label{eq:tPhi-2}
   \widetilde{\Lambda}_2= \diamondsuit\#_1[  \Phi_{k,h, \underline{g  g^* }}   \#_0 ( \rho_{h^*} \rho_{k^*}) ]\#_1 \diamondsuit,  \end{equation}as a whiskered associator (\ref{eq:associator}),
   changes the diagonal $\rho_{\underline{hg^*}}$ of the above dotted-wavy  quadrilateral to the wavy
diagonal $\rho_{ {kh}}$ of the same   quadrilateral   in the following  diagram:
 $$
      \xy 0;/r.22pc/:
 (-15,0)*+{ {x}}="1";
(15,0)*+{ {x} }="2";
(-30,0)*+{ x}="3";
(30,0)*+{ x}="4";
(-45,0)*+{ x}="5";
(45,0)*+{ x}="6";
(-60,0)*+{ x}="7";
(60,0)*+{ x}="8";
{\ar@{->}^{\rho_g} "3"; "1"};
{\ar@{->}^{\rho_{g^*}} "2"; "4"};
{\ar@{->}^{\rho_h} "5"; "3"};
{\ar@{-->}^{\rho_{h^*}} "4"; "6"};
{\ar@{->}^{\rho_k} "7"; "5"};
{\ar@{-->}^{\rho_{k^*}} "6"; "8"};
 {\ar@{->}^{\rho_f} "1";"2"};
 {\ar@/_2.33pc/^{\rho_{gf}} "3";"2"};
 {\ar@{->}@/_3.33pc/^{\rho_{\underline{gg^*}}} "3";"4"};
  {\ar@{~>}@/_1.33pc/_{\rho_{kh }} "7";"3"};
   {\ar@{-->}@/_5.33pc/^{\rho_{\underline{kg^*} }} "7";"4"};
   {\ar@{-->}@/_6.33pc/^{\rho_{\underline{kh^*}}} "7";"6"};
    {\ar@{-->}@/_7.33pc/_{\rho_{\underline{kk^*}}} "7";"8"};
\endxy.
 $$
   The $3$-isomorphism
 \begin{equation}\label{eq:tPhi-3}
   \widetilde{\Lambda}_3=\diamondsuit\#_1   \Phi_{\underline{k  g^*},h^*,k^* }^{-1},   \end{equation}as a whiskered associator (\ref{eq:associator}),
   changes the diagonal $\rho_{\underline{kh^*}}$ of the above dotted  quadrilateral to the wavy
diagonal $\rho_{ {h^*k^*}}$ of the same   quadrilateral   in the following  diagram:
  
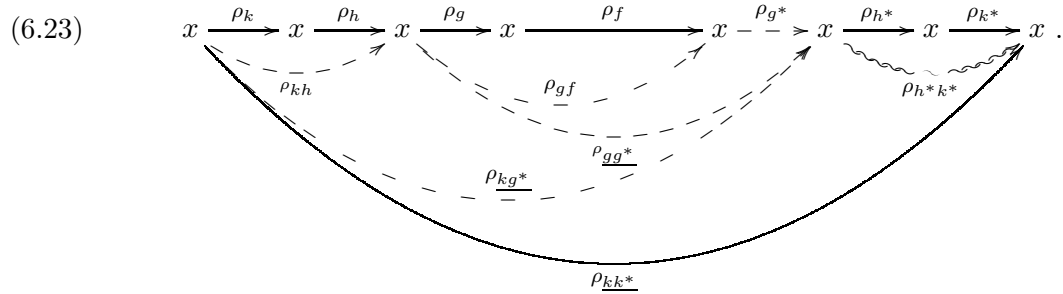
\begin{equation}\label{eq:tdiagram-3}
      \xy 0;/r.22pc/:
 (-15,0)*+{ {x}}="1";
(15,0)*+{ {x} }="2";
(-30,0)*+{ x}="3";
(30,0)*+{ x}="4";
(-45,0)*+{ x}="5";
(45,0)*+{ x}="6";
(-60,0)*+{ x}="7";
(60,0)*+{ x}="8";
{\ar@{->}^{\rho_g} "3"; "1"};
{\ar@{-->}^{\rho_{g^*}} "2"; "4"};
{\ar@{->}^{\rho_h} "5"; "3"};
{\ar@{->}^{\rho_{h^*}} "4"; "6"};
{\ar@{->}^{\rho_k} "7"; "5"};
{\ar@{->}^{\rho_{k^*}} "6"; "8"};
 {\ar@{->}^{\rho_f} "1";"2"};
 {\ar@{-->}@/_2.33pc/^{\rho_{gf}} "3";"2"};
 {\ar@{-->}@/_3.33pc/_{\scriptscriptstyle\rho_{\underline{gg^*}}} "3";"4"};
  {\ar@{-->}@/_1.33pc/_{\scriptscriptstyle\rho_{kh }} "7";"3"};
   {\ar@{-->}@/_5.33pc/^{\rho_{\underline{kg^*} }} "7";"4"};
   {\ar@{~>}@/_1.33pc/_{\rho_{h^*k^*}} "4";"8"};
    {\ar@/_7.33pc/_{\rho_{\underline{kk^*}}} "7";"8"};
\endxy.
 \end{equation}
 The composition $\widetilde{\Lambda}_1\#_2\widetilde{\Lambda}_2\#_2\widetilde{\Lambda}_3$ is the $3$-isomorphism
  $$
   \Gamma_{k,h }\circ\psi_g:\psi_k \circ\psi_h\circ\psi_g(\chi)\xy0;/r.22pc/:
  (0,0)*+{ }="1";
(6,0)*+{ }="2";
{\ar@3{->}    "1"+(0, 0);"2"+(0, 0)};
   \endxy \psi_{k h}
\circ\psi_{g}(\chi),$$  corresponding to the lower half   of (\ref{eq:3-composition}).

The $3$-isomorphism
 \begin{equation}\label{eq:tPhi-4}
   \widetilde{\Lambda}_4=   \diamondsuit\#_1[  \Phi_{k h, gf, g^*   }   \#_0 ( \rho_{h^*} \rho_{k^*}) ]\#_1\diamondsuit,  \end{equation}as a whiskered associator (\ref{eq:associator}),
   changes the diagonal $\rho_{\underline{gg^*}}$ of the   dotted  quadrilateral in (\ref{eq:tdiagram-3}) to the wavy
diagonal $\rho_{\underline{kf}}$ of the same   quadrilateral   in the following  diagram:
 $$
      
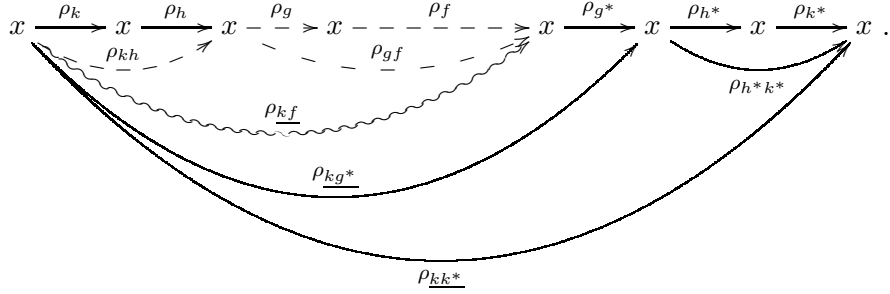
\xy 0;/r.22pc/:
 (-15,0)*+{ {x}}="1";
(15,0)*+{ {x} }="2";
(-30,0)*+{ x}="3";
(30,0)*+{ x}="4";
(-45,0)*+{ x}="5";
(45,0)*+{ x}="6";
(-60,0)*+{ x}="7";
(60,0)*+{ x}="8";
{\ar@{-->}^{\rho_g} "3"; "1"};
{\ar@{->}^{\rho_{g^*}} "2"; "4"};
{\ar@{->}^{\rho_h} "5"; "3"};
{\ar@{->}^{\rho_{h^*}} "4"; "6"};
{\ar@{->}^{\rho_k} "7"; "5"};
{\ar@{->}^{\rho_{k^*}} "6"; "8"};
 {\ar@{-->}^{\rho_f} "1";"2"};
 {\ar@{-->}@/_1.33pc/^{\rho_{gf}} "3";"2"};
 {\ar@{~>}@/_3.33pc/^{\rho_{\underline{kf}}} "7";"2"};
  {\ar@{-->}@/_1.33pc/^{\rho_{kh }} "7";"3"};
   {\ar@/_5.33pc/^{\rho_{\underline{kg^*} }} "7";"4"};
   {\ar@/_1.33pc/_{\rho_{h^*k^*}} "4";"8"};
    {\ar@/_7.33pc/_{\rho_{\underline{kk^*}}} "7";"8"};
\endxy.
 $$
 The $3$-isomorphism
 \begin{equation}\label{eq:tPhi-5}
   \widetilde{\Lambda}_5=  \diamondsuit\#_1[  \Phi_{   {kh}, g  ,f }    \#_0 ( \rho_{g^*}  \rho_{h^*}\rho_{k^*}) ]\#_1 \diamondsuit , \end{equation}
 as a whiskered associator (\ref{eq:associator}),  changes the diagonal $\rho_{ {gf}}$ of the above dotted  quadrilateral to the wavy
diagonal $\rho_{\underline{kg}}$ of the same   quadrilateral   in the following  diagram:
 $$
      
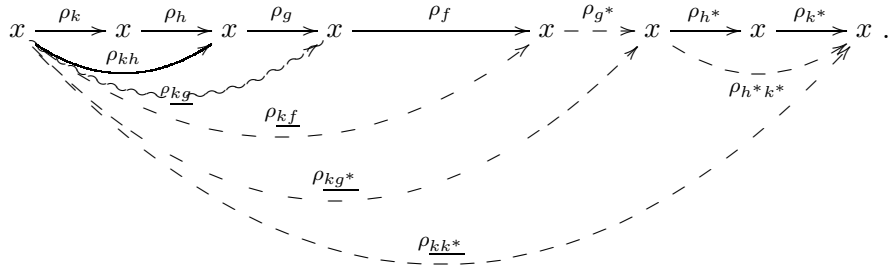
\xy 0;/r.22pc/:
 (-15,0)*+{ {x}}="1";
(15,0)*+{ {x} }="2";
(-30,0)*+{ x}="3";
(30,0)*+{ x}="4";
(-45,0)*+{ x}="5";
(45,0)*+{ x}="6";
(-60,0)*+{ x}="7";
(60,0)*+{ x}="8";
{\ar@{->}^{\rho_g} "3"; "1"};
{\ar@{-->}^{\rho_{g^*}} "2"; "4"};
{\ar@{->}^{\rho_h} "5"; "3"};
{\ar@{->}^{\rho_{h^*}} "4"; "6"};
{\ar@{->}^{\rho_k} "7"; "5"};
{\ar@{->}^{\rho_{k^*}} "6"; "8"};
 {\ar@{->}^{\rho_f} "1";"2"};
 {\ar@{~>}@/_2.03pc/|-{\scriptscriptstyle\rho_{\underline{kg}}} "7";"1"};
 {\ar@{-->}@/_3.33pc/^{\rho_{\underline{kf}}} "7";"2"};
  {\ar@/_1.33pc/^{\rho_{kh }} "7";"3"};
   {\ar@{-->}@/_5.33pc/^{\rho_{\underline{kg^*} }} "7";"4"};
   {\ar@{-->}@/_1.33pc/_{\rho_{h^*k^*}} "4";"8"};
    {\ar@{-->}@/_7.33pc/^{\rho_{\underline{kk^*}}} "7";"8"};
\endxy.
 $$
At last, the $3$-isomorphism
 \begin{equation}\label{eq:tPhi-6}
   \widetilde{\Lambda}_6=  \diamondsuit\#_1  \Phi_{ \underline{kf}, g^*,h^* k^* }^{-1} ,   \end{equation}as a whiskered associator (\ref{eq:associator}),
   changes the diagonal $\rho_{\underline{kg^*}}$ of the above dotted  quadrilateral to the wavy
diagonal $\rho_{\underline{g^*k^*}}$ of the same   quadrilateral   in the following  diagram:
  $$
      
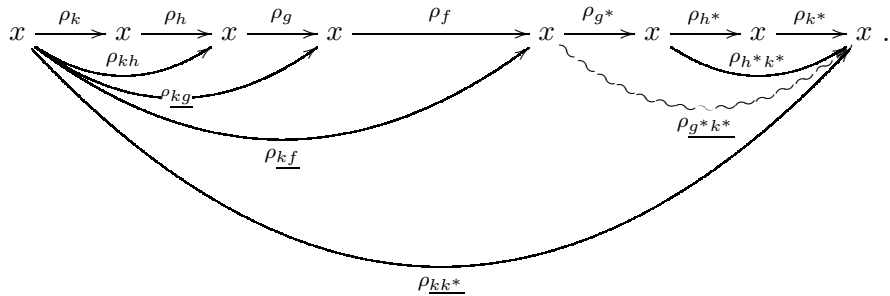
\xy 0;/r.22pc/:
 (-15,0)*+{ {x}}="1";
(15,0)*+{ {x} }="2";
(-30,0)*+{ x}="3";
(30,0)*+{ x}="4";
(-45,0)*+{ x}="5";
(45,0)*+{ x}="6";
(-60,0)*+{ x}="7";
(60,0)*+{ x}="8";
{\ar@{->}^{\rho_g} "3"; "1"};
{\ar@{->}^{\rho_{g^*}} "2"; "4"};
{\ar@{->}^{\rho_h} "5"; "3"};
{\ar@{->}^{\rho_{h^*}} "4"; "6"};
{\ar@{->}^{\rho_k} "7"; "5"};
{\ar@{->}^{\rho_{k^*}} "6"; "8"};
 {\ar@{->}^{\rho_f} "1";"2"};
 {\ar@/_2.03pc/|-{\scriptscriptstyle\rho_{\underline{kg} }} "7";"1"};
 {\ar@/_3.33pc/_{\rho_{\underline{kf}}} "7";"2"};
  {\ar@/_1.33pc/^{\rho_{kh }} "7";"3"};
   {\ar@{~>}@/_2.33pc/_{\rho_{ \underline{g^*k^*} }} "2";"8"};
   {\ar@/_1.33pc/^{\rho_{h^*k^*}} "4";"8"};
    {\ar@/_7.33pc/_{\rho_{\underline{kk^*}}} "7";"8"};
\endxy.
 $$
  The composition $\widetilde{\Lambda}_4\#_2\widetilde{\Lambda}_5\#_2\widetilde{\Lambda}_6$ is the $3$-isomorphism $  \Gamma_{kh ,g}(\chi):\psi_{k h}
\circ\psi_{g}(\chi)\xy0;/r.22pc/:
  (0,0)*+{ }="1";
(6,0)*+{ }="2";
{\ar@3{->}    "1"+(0, 0);"2"+(0, 0)};
   \endxy \psi_{k  h  g}(\chi)$ corresponding to the lower half    of (\ref{eq:3-composition}).

The composition of  $   \widetilde{\Lambda}_1\#_2\cdots\#_2\widetilde{\Lambda}_6$  in   the $2$-category $ \mathcal{C}^{  +}$ is the following diagram  $\mathscr{ D}_f^r:=$
  \begin{equation}\label{eq:diagram-2}
      \xy 0;/r.22pc/:
 (0,0)*+{ \rho_f }="1";
(30,0)*+{ \rho_{g f} }="2";
(60,0)*+{ \rho_{\underline{gg^*}}}="3";
(90,0)*+{ \rho_{\underline{hg^*}}}="4";
(120,0)*+{ \rho_{\underline{hh^*}}}="5";
(150,0)*+{ \rho_{\underline{kh^*}}}="6";
(180,0)*+{\rho_{\underline{kk^*}}}="7";
 (110,-25)*+{\rho_{\underline{kg^*}} }="8"; (80,-25)*+{\rho_{kh}  \rho_{\underline{gg^*}} }="9";
  (140,-25)*+{\rho_{\underline{kg^*}}\rho_{ h^*k^*}}="10";
  (40,-25)*+{  \rho_{kh}  \rho_{g f }  }="11";
  (60,-40)*+{  \rho_{\underline{kf} }  }="12";
  (10,-25)*+{  \rho_{kh }  }="13";  (20,-40)*+{  \rho_{\underline{kg} }  }="14";
  (130,-40)*+{  \rho_{\underline{kf} }\rho_{h^*k^*}  }="15"; (160,-40)*+{  \rho_{\underline{kf} }\rho_{\underline{g^*k^*}}  }="16";
   {\ar@{->}^{\phi_{g,f}} "1"; "2"};
{\ar@{-->}^{\phi_{g f,g^*}} "2"; "3"};
{\ar@{-->}^{\phi_{h, \underline{gg^*}}} "3"; "4"};
{\ar@{->}^{\phi_{\underline{hg^*}, h^*}} "4"; "5"};
{\ar@{->}^{\phi_{k,\underline{hh^*}}} "5"; "6"};
{\ar@{->}^{\phi_{\underline{kh^*},k^*}} "6"; "7"};
 {\ar@{=}^{}(50,-15)*^{} ;(60,-15)*^{}   }; {\ar@{=}^{}(20,-15)*^{} ;(30,-15)*^{}   };{\ar@{=}^{}(100,-35)*^{} ;(120,-35)*^{}   };
  {\ar@{-->}|-{\phi_{k, \underline{hg^*}    }} "4"; "8"};
{\ar@{->}|-{\phi_{\underline{kg^*},  h^*}} "8"; "6"};
 {\ar@{-->}_{\scriptscriptstyle\phi_{k, h    }} "3"; "9"};
{\ar@{-->}^{\phi_{kh, \underline{gg^*} }} "9"; "8"};
{\ar@{->}^{\phi_{h^*,k^*   }} "8"; "10"};
{\ar@{->}|-{\phi_{\underline{kg^*},h^*k^*}} "10"; "7"};
{\ar@{-->}^{\scriptscriptstyle\phi_{k ,h    }} "2"; "11"};
{\ar@{-->}^{\phi_{ g f, g^* }} "11"; "9"};
{\ar@{-->}_{\scriptscriptstyle\phi_{ kh, g f   }} "11"; "12"};
{\ar@{-->}_{\scriptscriptstyle\phi_{ \underline{kf} ,g^*}} "12"; "8"};
{\ar@{->}_{\phi_{k,h    }} "1"; "13"};
{\ar@{->}^{\scriptscriptstyle\phi_{   g, f }} "13"; "11"};
{\ar@{->}_{\phi_{kh ,g   }} "13"; "14"};
{\ar@{->}_{\phi_{   \underline{kg}, f }} "14"; "12"};
{\ar@{->}_{\phi_{h^*,k^*   }} "12"; "15"};
{\ar@{->}|-{\phi_{   \underline{kf}, g^* }} "15"; "10"};
{\ar@{->}_{\phi_{g^*,h^*k^*   }} "15"; "16"};
{\ar@{->}_{\phi_{   \underline{kf}, \underline{g^*k^*} }} "16"; "7"};
{\ar@{=>}^{  \widetilde{\Lambda}_1} (115,-5)*{}; (115,-20)*{}};
{\ar@{=>}^{  \widetilde{\Lambda}_2} (80,-5)*{}; (80,-20)*{}};
{\ar@{=>}^{  \widetilde{\Lambda}_3} (150,-5)*{}; (143,-20)*{}};
{\ar@{=>}^{  \widetilde{\Lambda}_4} (60,-27)*{}; (60,-39)*{}};
 {\ar@{=>}^{  \widetilde{\Lambda}_5} (30,-27)*{}; (30,-39)*{}};
{\ar@{=>}^{  \widetilde{\Lambda}_6} (148,-25)*{}; (157,-38)*{}};
\endxy.
 \end{equation}
Let $\mathscr D_1^r $  be the corresponding  diagram in $\mathcal{C}^+$ with   $f$ replaced by $1$, by using adjoint operations as in
 (\ref{eq:upper}). Then the $2$-isomorphism in $ \mathcal{C}^{  +}$ corresponding to the morphism
$  [\Gamma_{kh ,g} {\#} (\Gamma_{k,h}\circ \psi_g)](\chi)$ in $\mathbb{ {T}}r_2  \rho_f $
  is
 \begin{equation}\label{eq:D-r}
   \mathscr{ D}_1^r \xy0;/r.22pc/:
  (0,0)*+{ }="1";
(7,0)*+{ }="2";
{\ar@{->}^{\chi }   "1"+(0,0);"2"+(0,0)};
   \endxy  \mathscr{ D}_f^r.
 \end{equation}
\subsection{The proof of the associativity} Let us show the identity (\ref{eq:cat-tr-0}), i.e.,  that diagrams $\mathscr D_1^l\xrightarrow{\chi} \mathscr D_f^l$ in (\ref{eq:D-l}) and $\mathscr D_1^r\xrightarrow{\chi} \mathscr D_f^r$ in (\ref{eq:D-r})
are   identical in  the   $2$-category $ \mathcal{C}^{ +}$,  by using the $3$-cocycle identity (\ref{eq:Omega}) repeatedly. This proof corresponds to that of the  $1$-dimensional case     in Section  6.1 step by step.

Apply the $3$-cocycle identity (\ref{eq:Omega}) to the dotted diagram in (\ref{eq:diagram-2}) with $g_4=k,g_3=h,g_2=gf,g_1=g^*$ to get wavy isomorphisms in the
following diagram
   \begin{equation}\label{eq:diagram-3}
      \xy 0;/r.22pc/:
 (0,0)*+{ \rho_f }="1";
(30,0)*+{ \rho_{g f} }="2";
(60,0)*+{ \rho_{\underline{gg^*}}}="3";
(90,0)*+{ \rho_{\underline{hg^*}}}="4";
(120,0)*+{ \rho_{\underline{hh^*}}}="5";
(150,0)*+{ \rho_{\underline{kh^*}}}="6";
(180,0)*+{\rho_{\underline{kk^*}}}="7";
 (110,-25)*+{\rho_{\underline{kg^*}} }="8"; (80,-25)*+{\rho_{ \underline{hf} } }="9";
  (140,-25)*+{\rho_{\underline{kg^*}}\rho_{ h^*k^*}}="10";
  (40,-25)*+{  \rho_{kh}  \rho_{g f }  }="11";
  (60,-40)*+{  \rho_{\underline{kf} }  }="12";
  (10,-25)*+{  \rho_{kh }  }="13";  (20,-40)*+{  \rho_{\underline{kg} }  }="14";
  (130,-40)*+{  \rho_{\underline{kf} }\rho_{h^*k^*}  }="15"; (160,-40)*+{  \rho_{\underline{kf} }\rho_{\underline{g^*k^*}}  }="16";
   {\ar@{->}^{\scriptscriptstyle\phi_{g,f}} "1"; "2"};
{\ar@{->}^{\scriptscriptstyle\phi_{g f,g^*}} "2"; "3"};
{\ar@{->}^{\scriptscriptstyle\phi_{h, \underline{gg^*}}} "3"; "4"};
{\ar@{->}^{\scriptscriptstyle\phi_{\underline{hg^*}, h^*}} "4"; "5"};
{\ar@{->}^{\scriptscriptstyle\phi_{k,\underline{hh^*}}} "5"; "6"};
{\ar@{-->}^{\scriptscriptstyle\phi_{\underline{kh^*},k^*}} "6"; "7"};
   {\ar@{=}^{}(20,-15)*^{} ;(30,-15)*^{}   };{\ar@{=}^{}(100,-35)*^{} ;(120,-35)*^{}   };
  {\ar@{->}^{\scriptscriptstyle\phi_{k, \underline{hg^*}    }} "4"; "8"};
{\ar@{-->}^{\scriptscriptstyle\phi_{\underline{kg^*},  h^*}} "8"; "6"};
 {\ar@{<~}_{\scriptscriptstyle\phi_{  h ,gf   }} "9"; "2"};
{\ar@{~>}^{\scriptscriptstyle\phi_{ \underline{hf}, g^* }} "9"; "4"};
{\ar@{-->}^{\scriptscriptstyle\phi_{h^*,k^*   }} "8"; "10"};
{\ar@{-->}|-{\scriptscriptstyle\phi_{\underline{kg^*},h^*k^*}} "10"; "7"};
{\ar@{->}^{\scriptscriptstyle\phi_{k ,h    }} "2"; "11"};
{\ar@{~>}_{\scriptscriptstyle\phi_{k, \underline{hf} }} "9"; "12"};
{\ar@{->}_{\scriptscriptstyle\phi_{kh, g f   }} "11"; "12"};
{\ar@{-->}_{\scriptscriptstyle\phi_{ \underline{kf} ,g^*}} "12"; "8"};
{\ar@{->}_{\scriptscriptstyle\phi_{k,h    }} "1"; "13"};
{\ar@{->}_{\scriptscriptstyle\phi_{   g, f }} "13"; "11"};
{\ar@{->}_{\scriptscriptstyle\phi_{kh ,g   }} "13"; "14"};
{\ar@{->}_{\scriptscriptstyle\phi_{   \underline{kg}, f }} "14"; "12"};
{\ar@{-->}_{\scriptscriptstyle\phi_{h^*,k^*   }} "12"; "15"};
{\ar@{-->}_{\scriptscriptstyle\phi_{   \underline{kf}, g^* }} "15"; "10"};
{\ar@{-->}_{\scriptscriptstyle\phi_{g^*,h^*k^*   }} "15"; "16"};
{\ar@{-->}_{\scriptscriptstyle\phi_{   \underline{kf}, \underline{g^*k^*} }} "16"; "7"};
{\ar@{=>}_{  \widetilde{\Lambda}_3} (150,-5)*{}; (143,-20)*{}};
{\ar@{=>}^{  \widetilde{\Lambda}_6} (148,-25)*{}; (157,-38)*{}};
\endxy
 \end{equation}Note that $\widetilde{\Lambda}_3 $ in (\ref{eq:tPhi-3}) and $ \widetilde{\Lambda}_6$ in (\ref{eq:tPhi-6}) are the inverse of associators.
Apply the $3$-cocycle identity, the inverse version of (\ref{eq:Omega}) (the lower and upper boundaries are exchanged), to the above dotted diagram   with
$g_4=\underline{kf},g_3=g^*,g_2=h^*,g_1=k^*$ to get wavy isomorphisms in the following:
  \begin{equation}\label{eq:diagram-4}
      \xy 0;/r.22pc/:
 (0,0)*+{ \rho_f }="1";
(30,0)*+{ \rho_{g f} }="2";
(60,0)*+{ \rho_{\underline{gg^*}}}="3";
(90,0)*+{ \rho_{\underline{hg^*}}}="4";
(120,0)*+{ \rho_{\underline{hh^*}}}="5";
(150,0)*+{ \rho_{\underline{kh^*}}}="6";
(180,0)*+{\rho_{\underline{kk^*}}}="7";
 (110,-25)*+{\rho_{\underline{kg^*}} }="8"; (80,-25)*+{\rho_{ \underline{hf} } }="9";
  (145,-25)*+{\rho_{\underline{kf}}\rho_{g^* h^* }}="10";
  (40,-25)*+{  \rho_{kh}  \rho_{g f }  }="11";
  (60,-40)*+{  \rho_{\underline{kf} }  }="12";
  (10,-25)*+{  \rho_{kh }  }="13";  (20,-40)*+{  \rho_{\underline{kg} }  }="14";
  (130,-40)*+{  \rho_{\underline{kf} }\rho_{h^*k^*}  }="15"; (160,-40)*+{  \rho_{\underline{kf} }\rho_{\underline{g^*k^*}}  }="16";
   {\ar@{->}^{\scriptscriptstyle\phi_{g,f}} "1"; "2"};
{\ar@{->}^{\scriptscriptstyle\phi_{g f,g^*}} "2"; "3"};
{\ar@{->}^{\scriptscriptstyle\phi_{h, \underline{gg^*}}} "3"; "4"};
{\ar@{-->}^{\scriptscriptstyle\phi_{\underline{hg^*}, h^*}} "4"; "5"};
{\ar@{-->}^{\scriptscriptstyle\phi_{k,\underline{hh^*}}} "5"; "6"};
{\ar@{->}^{\scriptscriptstyle\phi_{\underline{kh^*},k^*}} "6"; "7"};
   {\ar@{=}^{}(20,-15)*^{} ;(30,-15)*^{}   };
  {\ar@{-->}^{\scriptscriptstyle\phi_{k, \underline{hg^*}    }} "4"; "8"};
{\ar@{-->}_{\scriptscriptstyle\phi_{\underline{kg^*},  h^*}} "8"; "6"};
 {\ar@{<-}_{ \scriptscriptstyle\phi_{ h,gf    }} "9"; "2"};
{\ar@{-->}^{\scriptscriptstyle\phi_{ \underline{hf},g^* }} "9"; "4"};
{\ar@{~>}_{\scriptscriptstyle\phi_{g^*,h^*    }} "12"; "10"};
{\ar@{~>}_{\scriptscriptstyle\phi_{\underline{kf},g^*h^* }} "10"; "6"};
{\ar@{->}^{\scriptscriptstyle\phi_{k ,h    }} "2"; "11"};
{\ar@{-->}_{\scriptscriptstyle\phi_{ k, \underline{hf} }} "9"; "12"};
{\ar@{->}_{\scriptscriptstyle\phi_{kh, g f   }} "11"; "12"};
{\ar@{-->}^{\scriptscriptstyle\phi_{ \underline{kf} ,g^*}} "12"; "8"};
{\ar@{->}_{\scriptscriptstyle\phi_{k,h    }} "1"; "13"};
{\ar@{->}_{\scriptscriptstyle\phi_{   g, f }} "13"; "11"};
{\ar@{->}_{\scriptscriptstyle\phi_{kh ,g   }} "13"; "14"};
{\ar@{->}_{\scriptscriptstyle\phi_{   \underline{kg}, f }} "14"; "12"};
{\ar@{->}_{\scriptscriptstyle\phi_{h^*,k^*   }} "12"; "15"};
{\ar@{~>}|-{\scriptscriptstyle\phi_{    g^*h^*,k^* }}"10";"16"};
{\ar@{->}_{\scriptscriptstyle\phi_{g^*,h^*k^*   }} "15"; "16"};
{\ar@{->}_{\scriptscriptstyle\phi_{   \underline{kf}, \underline{g^*k^*} }} "16"; "7"};
{\ar@{=>}^{  \widetilde{\Lambda}_1} (116,-5)*{}; (116,-20)*{}};
{\ar@{=>} (93,-9)*{}; (93,-25)*{}};{\ar@{=>}^{\widehat{\Lambda}} (115,-25)*{}; (135,-25)*{}};
\endxy
 \end{equation}where $\widehat{\Lambda}$ is the inverse of a whiskered associator.
Note that the commutative cube in (\ref{eq:Omega-cube}) implies the following identity.
$$\label{eq:Omega-2}\xy 0;/r.20pc/:
   (-10,15 )*+{\bullet }="1";
(40, 15)*+{ \bullet }="2";
( 40,55)*+{\bullet }="3";
(55,0)*+{\bullet  }="4";
(-10,55 )*+{ \bullet  }="5";(10,0 )*+{ \bullet  }="7";
(55,40)*+{\bullet }="8";
(50,20)*+{\scriptscriptstyle A_2^{'-1} }="08";
(50,9)*+{\scriptscriptstyle \phi_b   }="08";
 {\ar@{->}^{\phi_{ g_4,g_3g_2  } }  "1";"2" };
{\ar@{<-}_{\scriptscriptstyle\phi_a } "3";"2" };
{\ar@{<-}_{ \phi_{g_3g_2,g_1}} "5";"1" };
{\ar@{->}^{ } "4";"2" };
 {\ar@{->} "8";"3"_{\phi_{g_4 g_3 , g_2 g_1 } } };
{\ar@{->}^{\phi_{   g_4, g_3g_2 g_1} } "5";"3" };
{\ar@{<-}^{\phi_{ g_2, g_1} } "8";"4" };
  {\ar@{=>}^{ A_2  } (0,45)*{};(35,25)*{}} ;
 {\ar@{=>}_{  } (43, 15)*{};(52, 40)*{}} ;
 {\ar@{->}^{\phi_{  g_3,g_2 } } "7";"1" };
 {\ar@{<-}^{\phi_{  g_4, g_3  } } "4";"7" };
  {\ar@{=>}^{   A_3   } (10,12)*{};(35,3)*{}} ;
 \endxy \xy0;/r.22pc/:
  (3, 0)*+{ }="1";
(13,0)*+{ }="2";
{\ar@{=}^{ }   "1"+(0,20);"2"+(0,20)};
   \endxy
     \xy 0;/r.20pc/:
   (-10,15 )*+{ \bullet}="1";
( 15,0)*+{\bullet }="3";
(55,0)*+{ \bullet }="4";
(-10,55 )*+{ \bullet  }="5";
(  15,40)*+{ \bullet }="7";
(55,40)*+{\bullet }="8";
   (35,55 )*+{ \bullet}="9";
{\ar@{->}^{\phi_{  g_3,g_2 } } "3";"1" };
{\ar@{<-} "5";"1"_{ \phi_{g_3g_2,g_1}} };
 {\ar@{<-}^{\scriptscriptstyle\phi_{  g_2,g_1 } } "7";"3" };
{\ar@{->}_{\phi_{  g_4, g_3 } } "3";"4" };
{\ar@{<-}^{\scriptscriptstyle\phi_{ g_2, g_1} } "8";"4" };
 {\ar@{->}^{\phi_c } "7";"5" };
 {\ar@{->}_{\phi_{g_4, g_3}  } "7";"8" };
 {\ar@{=>}^{ A_1^{-1}  } (-5,15)*{};(10,  35  )*{}} ; {\ar@{<=}_{ A_1' } (42, 43)*{};(12,50)*{}} ;
 {\ar@{=}|-{     } (43 ,20)*{};(31, 20)*{}} ; {\ar@{<-}^{\phi_{ g_4g_3,g_2 g_1} } "9";"8" };{\ar@{->}^{\phi_{ g_4,g_3g_2 g_1 } } "5" ;"9" };
 \endxy
$$ where $\phi_a:=\phi_{ g_4 g_3 g_2,g_1}, \phi_b:=\phi_{ g_4 g_3, g_2}, \phi_c:=\phi_{ g_3, g_2 g_1}$.
The left-hand side is the back, bottom and right (this $2$-isomorphism is inverted) faces of the cube in (\ref{eq:Omega-cube}), while the right-hand side is the left (this $2$-isomorphism is inverted), top and front faces of the cube.
Apply this identity to the dotted-wavy diagram in (\ref{eq:diagram-4}) with $g_4=k$, $g_3=\underline{hf}$, $g_2=g^*$,$g_1=h^*$  to get wavy isomorphisms in the
following:
 \begin{equation}
      \xy 0;/r.22pc/:
 (0,0)*+{ \rho_f }="1";
(30,0)*+{ \rho_{g f} }="2";
(60,0)*+{ \rho_{\underline{gg^*}}}="3";
(90,0)*+{ \rho_{\underline{hg^*}}}="4";
(120,0)*+{ \rho_{\underline{hh^*}}}="5";
(150,0)*+{ \rho_{\underline{kh^*}}}="6";
(180,0)*+{\rho_{\underline{kk^*}}}="7";
 (110,-25)*+{\rho_{ \underline{hf} }\rho_{g^* h^* } }="8"; (80,-25)*+{\rho_{ \underline{hf} } }="9";
  (140,-25)*+{\rho_{\underline{kf}}\rho_{g^* h^* }}="10";
  (40,-25)*+{  \rho_{kh}  \rho_{g f }  }="11";
  (60,-40)*+{  \rho_{\underline{kf} }  }="12";
  (10,-25)*+{  \rho_{kh }  }="13";  (20,-40)*+{  \rho_{\underline{kg} }  }="14";
  (130,-40)*+{  \rho_{\underline{kf} }\rho_{h^*k^*}  }="15"; (160,-40)*+{  \rho_{\underline{kf} }\rho_{\underline{g^*k^*}}  }="16";
   {\ar@{-->}^{\phi_{g,f}} "1"; "2"};
{\ar@{->}^{\scriptscriptstyle\phi_{g f,g^*}} "2"; "3"};
{\ar@{->}^{\scriptscriptstyle\phi_{h, \underline{gg^*}}} "3"; "4"};
{\ar@{->}^{\scriptscriptstyle\phi_{\underline{hg^*}, h^*}} "4"; "5"};
{\ar@{->}^{\scriptscriptstyle\phi_{k,\underline{hh^*}}} "5"; "6"};
{\ar@{->}^{\scriptscriptstyle\phi_{\underline{kh^*},k^*}} "6"; "7"};
   {\ar@{=}^{}(20,-15)*^{} ;(30,-15)*^{}   };{\ar@{=}^{}(80,-30)*^{} ;(100,-30)*^{}   };
  {\ar@{~>}^{\scriptscriptstyle\phi_{  g^* ,  h^*   }} "9"; "8"};
{\ar@{~>}_{\scriptscriptstyle\phi_{ \underline{hf}, g^* h^*}} "8"; "5"};
 {\ar@{<--}_{\scriptscriptstyle\phi_{  h ,gf   }} "9"; "2"};
{\ar@{->}^{\scriptscriptstyle\phi_{ \underline{hf},g^* }} "9"; "4"};
{\ar@{->}_{\scriptscriptstyle\phi_{g^*,h^*    }} "12"; "10"};
{\ar@{->}_{\scriptscriptstyle\phi_{\underline{kf},g^*h^* }} "10"; "6"};
{\ar@{-->}^{\scriptscriptstyle\phi_{k ,h    }} "2"; "11"};
{\ar@{-->}_{\scriptscriptstyle\phi_{ k, \underline{hf} }} "9"; "12"};
{\ar@{-->}_{\scriptscriptstyle\phi_{kh, g f   }} "11"; "12"};
{\ar@{~>}^{\scriptscriptstyle\phi_{ k,\underline{hf}  }} "8"; "10"};
{\ar@{-->}_{\scriptscriptstyle\phi_{k,h    }} "1"; "13"};
{\ar@{-->}^{\scriptscriptstyle\phi_{   g, f }} "13"; "11"};
{\ar@{-->}_{\scriptscriptstyle\phi_{kh ,g   }} "13"; "14"};
{\ar@{-->}_{\scriptscriptstyle\phi_{   \underline{kg}, f }} "14"; "12"};
{\ar@{->}_{\scriptscriptstyle\phi_{h^*,k^*   }} "12"; "15"};
{\ar@{->}^{\scriptscriptstyle\phi_{    g^*h^*,k^* }}"10";"16"};
{\ar@{->}_{\scriptscriptstyle\phi_{g^*,h^*k^*   }} "15"; "16"};
{\ar@{->}_{\scriptscriptstyle\phi_{   \underline{kf}, \underline{g^*k^*} }} "16"; "7"};
 {\ar@{=>}^{  \widetilde{\Lambda}_5} (30,-27)*{}; (30,-39)*{}};
  {\ar@{=>} (60,-20)*{}; (48,-28)*{}};
\endxy
 \end{equation}
Apply the $3$-cocycle identity (\ref{eq:Omega}) to the above dotted diagram   with $g_4=k$, $g_3=h$, $g_2=g $, $g_1=f$  to get wavy isomorphisms in the following  diagram $\widetilde{\mathscr D}_f^r:=$
 \begin{equation}\label{eq:Omega-000}
      \xy 0;/r.22pc/:
 (0,0)*+{ \rho_f }="1";
(30,0)*+{ \rho_{g f} }="2";
(60,0)*+{ \rho_{\underline{gg^*}}}="3";
(90,0)*+{ \rho_{\underline{hg^*}}}="4";
(120,0)*+{ \rho_{\underline{hh^*}}}="5";
(150,0)*+{ \rho_{\underline{kh^*}}}="6";
(180,0)*+{\rho_{\underline{kk^*}}}="7";
 (110,-25)*+{\rho_{ \underline{hf} }\rho_{g^* h^* } }="8"; (80,-25)*+{\rho_{ \underline{hf} } }="9";
  (140,-25)*+{\rho_{\underline{kf}}\rho_{g^* h^* }}="10";
  (40,-25)*+{  \rho_{ h g   }  }="11";
  (60,-40)*+{  \rho_{\underline{kf} }  }="12";
  (10,-25)*+{  \rho_{kh }  }="13";  (20,-40)*+{  \rho_{\underline{kg} }  }="14";
  (130,-40)*+{  \rho_{\underline{kf} }\rho_{h^*k^*}  }="15"; (160,-40)*+{  \rho_{\underline{kf} }\rho_{\underline{g^*k^*}}  }="16";
     {\ar@{->}^{\phi_{g,f}} "1"; "2"};
{\ar@{->}^{\phi_{g f,g^*}} "2"; "3"};
{\ar@{->}^{\phi_{h, \underline{gg^*}}} "3"; "4"};
{\ar@{->}^{\phi_{\underline{hg^*}, h^*}} "4"; "5"};
{\ar@{->}^{\phi_{k,\underline{hh^*}}} "5"; "6"};
{\ar@{->}^{\phi_{\underline{kh^*},k^*}} "6"; "7"};
     {\ar@{->}^{\scriptscriptstyle\phi_{  g^*  ,  h^*  }} "9"; "8"};
{\ar@{->}_{\scriptscriptstyle\phi_{ \underline{hf}, g^* h^*}} "8"; "5"};
 {\ar@{<-}_{\scriptscriptstyle\phi_{  h ,gf   }} "9"; "2"};
{\ar@{->}^{\scriptscriptstyle\phi_{ \underline{hf},g^* }} "9"; "4"};
{\ar@{->}_{\scriptscriptstyle\phi_{ g^*,  h^*    }} "12"; "10"};
{\ar@{->}_{\scriptscriptstyle\phi_{ \underline{hf},g^*h^* }} "10"; "6"};
{\ar@{~>}^{\scriptscriptstyle\phi_{ h,g    }} "1"; "11"};
{\ar@{->}_{\scriptscriptstyle\phi_{k,\underline{hf}  }} "9"; "12"};
{\ar@{~>}^{\scriptscriptstyle\phi_{ h g ,f   }} "11"; "9"};
{\ar@{->}^{\scriptscriptstyle\phi_{ k,\underline{hf} }} "8"; "10"};
{\ar@{->}_{\scriptscriptstyle\phi_{k,h    }} "1"; "13"};
{\ar@{~>}^{\scriptscriptstyle\phi_{   k,hg  }} "11";"14"};
{\ar@{->}_{\scriptscriptstyle\phi_{kh ,g   }} "13"; "14"};
{\ar@{->}_{\scriptscriptstyle\phi_{   \underline{kg}, f }} "14"; "12"};
{\ar@{->}_{\scriptscriptstyle\phi_{h^*,k^*   }} "12"; "15"};
{\ar@{->}^{\scriptscriptstyle\phi_{  g^*h^*,k^*}}"10";"16"};
{\ar@{->}_{\scriptscriptstyle\phi_{g^*,h^*k^*   }} "15"; "16"};
{\ar@{->}_{\scriptscriptstyle\phi_{   \underline{kf}, \underline{g^*k^*} }} "16"; "7"};{\ar@{=}^{}(80,-30)*^{} ;(100,-30)*^{}   };{\ar@{=>}_{   \Xi_2   }
(140,-29)*{};(130,-37)*{}} ;
 {\ar@{=>}_{ \Xi_1} (33,-25)*{}; (15,-25)*{}};
\endxy
 \end{equation}

With $f$ replaced by $1$, by using adjoint operations as in
 (\ref{eq:upper}),  the diagram $ {\mathscr D}_1^r $ corresponding to the upper half   is identically  changed to
the following  diagram $\widetilde{\mathscr D}_1^r:=$
 \begin{equation}\label{eq:Omega-00}
      \xy 0;/r.22pc/:
 (0,0)*+{ \rho_1 }="1";
(-30,0)*+{ \rho_{g 1} }="2";
(-60,0)*+{ \rho_{\underline{gg^*}}}="3";
(-90,0)*+{ \rho_{\underline{hg^*}}}="4";
(-120,0)*+{ \rho_{\underline{hh^*}}}="5";
(-150,0)*+{ \rho_{\underline{kh^*}}}="6";
(-180,0)*+{\rho_{\underline{kk^*}}}="7";
 (-110,-25)*+{\rho_{ \underline{h1} }\rho_{g^* h^* } }="8"; (-80,-25)*+{\rho_{ \underline{h1} } }="9";
  (-140,-25)*+{\rho_{\underline{k1}}\rho_{g^* h^* }}="10";
  (-40,-25)*+{  \rho_{ h g   }  }="11";
  (-60,-40)*+{  \rho_{\underline{k1} }  }="12";
  (-10,-25)*+{  \rho_{kh }  }="13";  (-20,-40)*+{  \rho_{\underline{kg} }  }="14";
  (-130,-40)*+{  \rho_{\underline{k1} }\rho_{h^*k^*}  }="15"; (-160,-40)*+{  \rho_{\underline{k1} }\rho_{\underline{g^*k^*}}  }="16";
     {\ar@{<-}_{\scriptscriptstyle\phi^{-1}_{g,1}} "1"; "2"};
{\ar@{<-}_{\scriptscriptstyle\phi^{-1}_{g 1,g^*}} "2"; "3"};
{\ar@{<-}_{\scriptscriptstyle\phi^{-1}_{h, \underline{gg^*}}} "3"; "4"};
{\ar@{<-}_{\scriptscriptstyle\phi^{-1}_{\underline{hg^*}, h^*}} "4"; "5"};
{\ar@{<-}_{\scriptscriptstyle\phi^{-1}_{k,\underline{hh^*}}} "5"; "6"};
{\ar@{<-}_{\scriptscriptstyle\phi^{-1}_{\underline{kh^*},k^*}} "6"; "7"};
     {\ar@{<-}_{\scriptscriptstyle\phi^{-1}_{  g^*  ,  h^*  }} "9"; "8"};
{\ar@{<-}_{\scriptscriptstyle\phi^{-1}_{ \underline{h1}, g^* h^*}} "8"; "5"};
 {\ar@{->}_{\scriptscriptstyle\phi^{-1}_{  h ,g1   }} "9"; "2"};
{\ar@{<-}_{\scriptscriptstyle\phi^{-1}_{ \underline{h1},g^* }} "9"; "4"};
{\ar@{<-}^{\scriptscriptstyle\phi^{-1}_{ g^*,  h^*    }} "12"; "10"};
{\ar@{<-}_{\scriptscriptstyle\phi^{-1}_{ \underline{h1},g^*h^* }} "10"; "6"};
{\ar@{<-}_{\scriptscriptstyle\phi^{-1}_{ h,g    }} "1"; "11"};
{\ar@{<-}|-{\scriptscriptstyle\phi^{-1}_{k,\underline{h1}  }} "9"; "12"};
{\ar@{<-}_{\scriptscriptstyle\phi^{-1}_{ h g ,1   }} "11"; "9"};
{\ar@{<-}_{\scriptscriptstyle\phi^{-1}_{ k,\underline{h1} }} "8"; "10"};
{\ar@{<-}^{\scriptscriptstyle\phi^{-1}_{k,h    }} "1"; "13"};
{\ar@{<-}|-{\scriptscriptstyle\phi^{-1}_{   k,hg  }} "11";"14"};
{\ar@{<-}^{\scriptscriptstyle\phi^{-1}_{kh ,g   }} "13"; "14"};
{\ar@{<-}^{\scriptscriptstyle\phi^{-1}_{   \underline{kg}, 1 }} "14"; "12"};
{\ar@{<-}^{\scriptscriptstyle\phi^{-1}_{h^*,k^*   }} "12"; "15"};
{\ar@{<-}_{\scriptscriptstyle\phi^{-1}_{  g^*h^*,k^*}}"10";"16"};
{\ar@{<-}^{\scriptscriptstyle\phi^{-1}_{g^*,h^*k^*   }} "15"; "16"};
{\ar@{<-}^{\scriptscriptstyle\phi^{-1}_{   \underline{k1}, \underline{g^*k^*} }} "16"; "7"};{\ar@{=}^{}(-80,-30)*^{} ;(-100,-30)*^{}   };{\ar@{=>}_{   \Xi_2^\dag   }
(-140,-27)*{};(-130,-37)*{}} ;
 {\ar@{=>}^{ \Xi_1^\dag} (-35,-25)*{}; (-14,-25)*{}};
\endxy,
 \end{equation} where $\Xi_j^\dag$ is the adjoint of $\Xi_j $, $j=1,2$.
Then the whole diagram $ {  \mathscr D}_1^r\xrightarrow{\chi}  { \mathscr D}_f^r$ in (\ref{eq:D-r}) is identically changed to
 $$
     \widetilde{  \mathscr D}_1^r\xrightarrow{\chi}\widetilde{ \mathscr D}_f^r,
 $$namely,
  \begin{equation}\label{eq:D-chi-D}
      \xy 0;/r.22pc/:
 (0,0)*+{ \rho_f }="1"; (10,0)*+{ \cdots\cdots }="01";
  (10,-25)*+{  \rho_{kh }  }="13";  (20,-40)*+{  \rho_{\underline{kg} }   }="14";
(40,-25)*+{  \rho_{ h g  }  }="11";(50,-25)*+{ \cdots \cdots }="011";
{\ar@{->}^{\phi_{ h,g    }} "1"; "11"};
{\ar@{->}_{\phi_{k,h    }} "1"; "13"};
{\ar@{->}^{\phi_{   k,hg  }} "11";"14"};
{\ar@{->}_{\phi_{kh ,g   }} "13"; "14"};
(-30,0)*+{ \rho_1}="10";(-40,0)*+{\cdots \cdots }="010";
  (-40,-25)*_{  \rho_{kh }  }="130";  (-50,-40)*+{ \rho_{\underline{kg} }}="140";
(-70,-25)*+{ \rho_{ h g  } }="110";(-80,-25)*+{ \cdots\cdots   }="0110";
{\ar@{<-}_{\phi^{-1}_{ h,g    }} "10"; "110"};
{\ar@{<-}^{\phi^{-1}_{k,h    }} "10"; "130"};
{\ar@{<-}_{\phi^{-1}_{   k,hg  }} "110";"140"};
{\ar@{<-}^{\phi^{-1}_{kh ,g   }} "130"; "140"};{\ar@{->}^{\chi} "10"; "1"};{\ar@{<=}^{   \Xi_1  } (15,-25)*{};(33,-25)*{}} ;
{\ar@{<=}_{   \Xi_1^\dag  } (-45,-25)*{};(-65,-25)*{}} ;
\endxy
 \end{equation}

 Note that (\ref{eq:Omega-000}) is exactly ${  \mathscr D}_f^l$ in (\ref{eq:diagram-1}) with two extra $2$-isomorphisms   $\Xi_1 $ and $\Xi_2$.
 But by definition, the $2$-isomorphisms $\Xi_1 $ and $\Xi_1^\dag$ are the associators (\ref{eq:diagonal}) corresponding to the $3$-isomorphisms in $\mathcal{C}$, which change
  $$
      
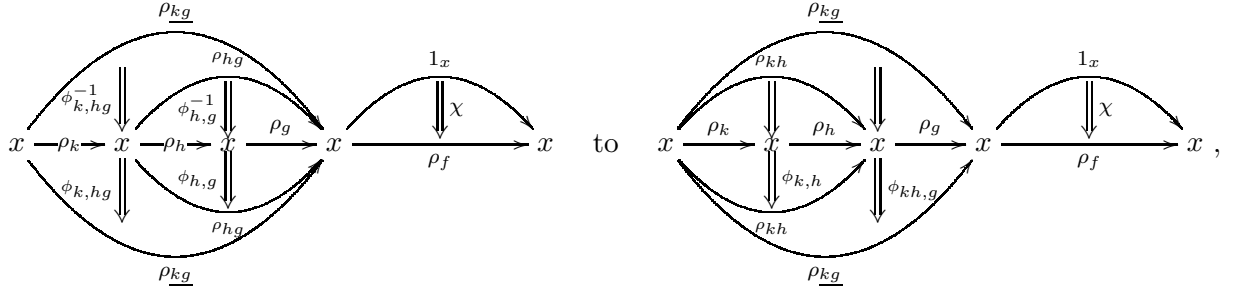
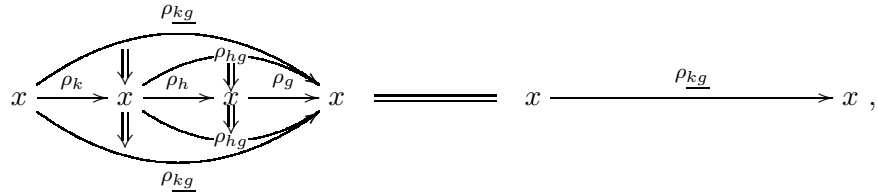
\xy 0;/r.22pc/:
 (-15,0)*+{ {x}}="1";
(15,0)*+{ {x} }="2";
(-30,0)*+{ x}="3";
(-45,0)*+{ x}="5";
(-60,0)*+{ x}="7";
{\ar@{->}^{\rho_g} "3"; "1"};
{\ar@{->}|-{\rho_h} "5"; "3"};
{\ar@{->}|-{\rho_k} "7"; "5"};
 {\ar@{->}_{\rho_f} "1";"2"};
 {\ar@/_3.53pc/_{\rho_{\underline{kg} }} "7";"1"};
  {\ar@/_2.13pc/_{\scriptscriptstyle\rho_{hg }} "5";"1"};
   {\ar@/^3.53pc/^{\rho_{\underline{kg} }} "7";"1"};
  {\ar@/^2.13pc/^{\scriptscriptstyle\rho_{hg }} "5";"1"};
 {\ar@/^2.13pc/^{ {1_x}} "1";"2"};{\ar@{=>}^{   \chi   } (0,9)*{};(0,1)*{}} ;
 {\ar@{=>}_{\scriptscriptstyle  \phi_{   k,hg  }^{-1}  } (-45,11)*{};(-45,2)*{}} ; {\ar@{=>}_{\scriptscriptstyle \phi_{   k,hg  } } (-45,-2)*{};(-45,-11)*{}} ;{\ar@{=>}_{\scriptscriptstyle \phi_{    h,g  }^{-1}   } (-30,9)*{};(-30,1)*{}} ;{\ar@{=>}_{\scriptscriptstyle \phi_{    h,g  } } (-30,-1)*{};(-30,-9)*{}} ;
   \endxy
 \quad {\rm  to}
  \quad
      \xy 0;/r.22pc/:
 (-15,0)*+{ {x}}="1";
(15,0)*+{ {x} }="2";
(-30,0)*+{ x}="3";
(-45,0)*+{ x}="5";
(-60,0)*+{ x}="7";
{\ar@{->}^{\rho_g} "3"; "1"};
{\ar@{->}^{\rho_h} "5"; "3"};
{\ar@{->}^{\rho_k} "7"; "5"};
 {\ar@{->}_{\rho_f} "1";"2"};
 {\ar@/_3.53pc/_{\rho_{\underline{kg} }} "7";"1"};
  {\ar@/_2.13pc/_{\scriptscriptstyle\rho_{kh }} "7";"3"};
      {\ar@/^3.53pc/^{\rho_{\underline{kg} }} "7";"1"};
  {\ar@/^2.13pc/^{\scriptscriptstyle\rho_{kh }} "7";"3"};
  {\ar@/^2.13pc/^{ {1_x}} "1";"2"};
  {\ar@{=>}^{   \chi   } (0,9)*{};(0,1)*{}} ;
 {\ar@{=>}|-{      } (-30,11)*{};(-30,2)*{}} ;
 {\ar@{=>}^{  \scriptscriptstyle \phi_{   k h,g  }   } (-30,-2)*{};(-30,-11)*{}} ;{\ar@{=>}|-{      } (-45,9)*{};(-45,1)*{}} ;{\ar@{=>}^{\scriptscriptstyle  \phi_{   k,h   }    } (-45,-1)*{};(-45,-9)*{}} ;
   \endxy,
 $$
 and we have
   $$ \xy 0;/r.22pc/:
 (-15,0)*+{ {x}}="1";
 (-30,0)*+{ x}="3";
(-45,0)*+{ x}="5";
(-60,0)*+{ x}="7";
{\ar@{->}^{\rho_g} "3"; "1"};
{\ar@{->}^{\rho_h} "5"; "3"};
{\ar@{->}^{\rho_k} "7"; "5"};
  {\ar@/_2.03pc/_{\rho_{\underline{kg} }} "7";"1"};
  {\ar@/_1.33pc/|-{\rho_{hg }} "5";"1"};
   {\ar@/^2.03pc/^{\rho_{\underline{kg} }} "7";"1"};
  {\ar@/^1.33pc/|-{\rho_{hg }} "5";"1"};
 {\ar@{=>}|-{      } (-45,7)*{};(-45,2)*{}} ; {\ar@{=>}|-{      } (-45,-2)*{};(-45,-7)*{}} ;{\ar@{=>}|-{      } (-30,5)*{};(-30,1)*{}} ;{\ar@{=>}|-{      } (-30,-1)*{};(-30,-5)*{}} ;
   \endxy \xy0;/r.22pc/:
  (0,0)*+{ }="1";
(20,0)*+{ }="2";
{\ar@{=}^{ }   "1" ;"2" };
   \endxy
      \xy 0;/r.22pc/:(-60,0)*+{ x}="7";
 (-15,0)*+{ {x}}="1";
 {\ar@{->}^{\rho_{\underline{kg} }} "7"; "1"};
   \endxy,
 $$by cancellation (\ref{eq:cancel}).  So $\Xi_1 $ and $\Xi_1^\dag$ are cancelled. More precisely, as a $3$-isomorphism,   $\Xi_1^\dag \#_0\chi\#_0\Xi_1$ is
 $$
    (\Xi_1^\dag \#_0\chi)\#_1(\Xi_1\#_0\chi)=(\Xi_1^\dag \#_1 \Xi_1)\#_0\chi =1_{\rho_{\underline{kg}}}\#_0\chi,$$
and we have
  $$
   (\phi_{   k,hg  }^{-1}\#_0 \phi_{    h,g  }^{-1})\#_0(\phi_{    h,g  }\#_0\phi_{   k,hg  })\#_0\chi=(\phi_{   k,hg  }^{-1}\#_0 \phi_{    h,g  }^{-1})\#_0\chi\#_0(\phi_{    h,g  }\#_0\phi_{   k,hg  })  $$in the $2$-category $\mathcal{C}^+$, up to whiskering, by the interchange  law. Namely,  $\widetilde{  \mathscr D}_1^r\xrightarrow{\chi}\widetilde{ \mathscr D}_f^r$ in (\ref{eq:D-chi-D}) is identical to
 $$
      
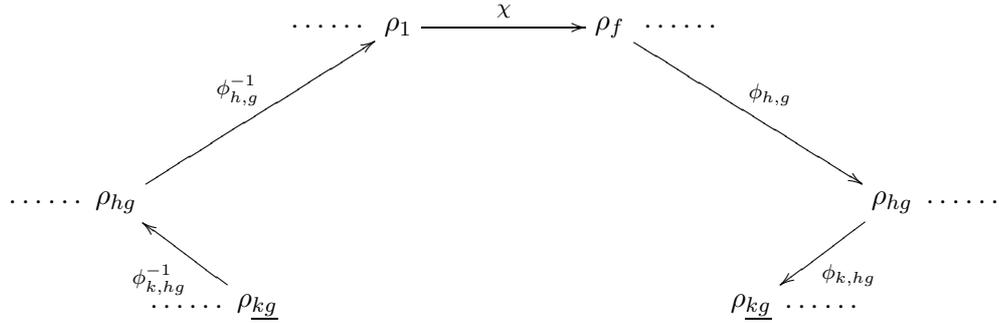
\xy 0;/r.22pc/:(20,-40)*+{  \rho_{\underline{kg} }   }="14";  (-50,-40)*+{ \rho_{\underline{kg} }}="140";
 (0,0)*+{ \rho_f }="1"; (10,0)*+{ \cdots\cdots }="01";
  (30,-40)*+{  \cdots \cdots }="014";
(40,-25)*+{  \rho_{ h g  }  }="11";(50,-25)*+{ \cdots \cdots }="011";
{\ar@{->}^{\phi_{ h,g    }} "1"; "11"};
 {\ar@{->}^{\phi_{   k,hg  }} "11";"14"};
(-30,0)*+{ \rho_1}="10";(-40,0)*+{\cdots \cdots }="010";
  (-60,-40)*+{\cdots \cdots   }="0140";
(-70,-25)*+{ \rho_{ h g  } }="110";(-80,-25)*+{ \cdots\cdots   }="0110";
{\ar@{<-}_{\phi^{-1}_{ h,g    }} "10"; "110"};
{\ar@{<-}_{\phi^{-1}_{   k,hg  }} "110";"140"};
{\ar@{->}^{\chi} "10"; "1"};
\endxy
 $$
 Similarly, the $2$-isomorphisms
 $\Xi_2  $
 (\ref{eq:Omega-000}) and $\Xi_2^\dag  $ in (\ref{eq:Omega-00}) are also cancelled. The resulting diagram is exactly the diagram
 $\mathscr D_1^l\xrightarrow{\chi} \mathscr D_f^l$ in (\ref{eq:D-l}). This completes the proof of Theorem \ref{thm:tr-cat}.

\end{document}